\documentclass{amsart}
\usepackage{amsaddr}
\usepackage[utf8]{inputenc}
\usepackage[T1]{fontenc}
\usepackage{amsmath}
\usepackage{amsfonts}
\usepackage{amssymb}
\usepackage{mathtools}
\usepackage{hyperref}   
\usepackage{url}
\usepackage{booktabs}
\usepackage{nicefrac}
\usepackage{microtype}
\usepackage{fancyhdr}   
\usepackage{amsthm}
\usepackage{enumerate}
\usepackage{chngcntr}
\usepackage{apptools}
\usepackage{xcolor}
\usepackage[doi=false, url =false , isbn = false,style=ieee,sorting=nyt, backend = biber]{biblatex}
\AtAppendix{\counterwithin{lemma}{section}}

\DeclareMathOperator{\id}{id}
\DeclareMathOperator{\Ad}{Ad}
\DeclareMathOperator{\modnew}{mod}
\DeclareMathOperator{\Aut}{Aut}
\DeclareMathOperator{\vn}{vN}

\DeclareMathOperator{\type}{type}
\DeclareMathOperator{\Int}{Int}
\DeclareMathOperator{\Out}{Out}
\DeclareMathOperator{\Cnt}{Cnt}

\begin{document}

\title{Classification of regular subalgebras of injective type III factors}

\author{Soham Chakraborty}
\address{Department of Mathematics, KU Leuven \\
02.32, Celestijnenlaan 200B, Leuven 3001, Belgium} 
\email{soham.chakraborty@kuleuven.be}

\newtheorem{theorem}{Theorem}[section]
\newtheorem{corollary}[theorem]{Corollary}
\newtheorem{lemma}[theorem]{Lemma}
\newtheorem{proposition}[theorem]{Proposition}
\theoremstyle{definition}
\newtheorem{definition}[theorem]{Definition}
\theoremstyle{remark}
\newtheorem{remark}[theorem]{Remark}
\newtheoremstyle{named}{}{}{\itshape}{}{\bfseries}{.}{.5em}{#3}
\theoremstyle{named}
\newtheorem*{namedtheorem}{Theorem}

\begin{abstract}
We provide a complete classification for regular subalgebras $B \subset M$ of injective factors satisfying a natural relative commutant condition. We show that such subalgebras are classified by their associated amenable discrete measured groupoid $\mathcal{G}= \mathcal{G}_{B \subset M}$ and the action $\modnew(\alpha)$ of $\mathcal{G}$ on the flow of weights induced by the cocycle action $(\alpha,u)$ of $\mathcal{G}$ on $B$. We obtain a similar result for triple inclusions $A \subset B \subset M$ where $M$ is an injective factor, $A$ is a Cartan subalgebra of $M$, and $B \subset M$ is regular, showing that such inclusions are also classified by their associated groupoid $\mathcal{G} = \mathcal{G}_{B \subset M}$ and the induced action on the flow of weights. Given such a discrete measured amenable groupoid $\mathcal{G}$, we also construct a model action of $\mathcal{G}$ on a field of Cartan inclusions with prescribed action on the associated field of flows. 
\end{abstract}

\maketitle

\section{Introduction}
\label{Intro}

For precise definitions and terminology used in the introduction, we direct the reader to the preliminaries in Section \ref{prelim}.  

\textit{Historical overview of the classification of group and groupoid actions}: In the theory of operator algebras, one central point of interest has always been the classification of automorphisms and group actions. The first step in this direction was taken when Connes provided a classification of automorphisms of the hyperfinite $\textrm{II}_{1}$ factor $R$ in \cite{MR394228} and \cite{MR448101}. Soon after, the classification of actions of finite groups on $R$ was done by Jones in \cite{MR587749} and then extended to actions of discrete amenable groups by Ocneanu in \cite{MR807949}, where he proved that any two outer actions of a discrete amenable group on $R$ are cocycle conjugate. The problem of obtaining the correct invariants for discrete amenable group actions on injective type $\textrm{III}$ factors was then solved by Sutherland and Takesaki in \cite{MR842413} and \cite{MR990219} for type $\textrm{III}_{\lambda}$ with $\lambda \neq 1$. They also resolved the remaining $\textrm{III}_{1}$ case together with Kawahigashi in \cite{MR1179014} for discrete abelian groups and with Katayama in \cite{MR1621416} for discrete amenable groups, hence obtaining the complete set of invariants for such actions. The methods in the above papers depend heavily on the type of the factor and the lack of a discrete decomposition in the type $\textrm{III}_{1}$ case makes it more challenging. In \cite{MR2353241}, Masuda gave an Evans-Kishimoto type intertwining argument (\cite{MR1432548}) and classified centrally free actions of discrete amenable groups on injective factors. His proof is independent of the type of the factor. In \cite{MR3181546}, he extended his results to general outer actions of discrete amenable groups on injective factors, hence providing a unified approach to the proof of such a classification theorem. 

Extending the classification theory to actions of discrete measured groupoids on injective factors was first motivated by the classification problem of compact abelian group actions on injective factors. In \cite{MR766264}, Jones and Takesaki classified compact abelian group actions, by first using Pontrjagin duality to reduce the problem to the classification of discrete abelian groups on semifinite von Neumann algebras, and then reducing that to the classification of groupoid actions on semifinite factors. This was further developed by Sutherland and Takesaki in \cite{MR842413} and Kawahigashi and Takesaki in \cite{MR1156672}. For discrete amenable groupoid actions on injective factors of type $\textrm{III}$, the complete set of invariants was obtained in \cite{MR4484234} by Masuda. 

\textit{Regular subalgebras and connections to the classification of actions:} In his seminal paper \cite{MR454659}, Connes showed that a finite von Neumann algebra is amenable if and only if it is approximately finite-dimensional (AFD). One important consequence of this theorem is that a crossed product von Neumann algebra $B \rtimes_{(\alpha,u)} \mathcal{G}$ arising from a free cocycle action of a discrete amenable groupoid $\mathcal{G}$ on a tracial injective von Neumann algebra, which preserves the trace and acts ergodically on the center $\mathcal{Z}(B)$ is always isomorphic to the hyperfinite $\textrm{II}_{1}$ factor $R$. If the subalgebra $B$ and the groupoid $\mathcal{G}$ uniquely determine such a crossed product decomposition of $R$ thus not depending on the action $(\alpha,u)$, was naturally an interesting question to consider. The question can also be rephrased as whether any two regular subalgebras $B_{1}$ and $B_{2}$ satisfying $\mathcal{Z}(B_{i}) = B_{i}' \cap R$ with isomorphic associated groupoids $\mathcal{G}_{B_{i} \subset R}$ are conjugate by an automorphism of $R$. Hence this establishes a bridge between the classification of regular subalgebras of $R$ satisfying the relative commutant condition and the classification of cocycle actions of an amenable discrete measured groupoid on $R$. 

In the case when $B$ is abelian, i.e., a Cartan subalgebra of $R$, Feldman-Moore theory (\cite{MR578656} and \cite{MR578730}) shows that the associated groupoid is principal, i.e. a countable p.m.p. ergodic equivalence relation.  
The uniqueness problem was solved in this case by Connes, Feldman, and Weiss in \cite{MR662736}.
On the other hand, when $B$ is a factor, $\mathcal{G}_{B \subset R} = \mathcal{N}_{R}(B)$ is a discrete amenable group and the uniqueness theorem follows from Ocneanu's classification in \cite{MR807949}. In the general case, even though groupoid actions were classified back in \cite{MR842413}, the uniqueness problem of all cocycle actions $(\alpha,u)$ preserving the trace was recently solved by Popa, Shlyakhtenko, and Vaes who showed in \cite{MR4124434} that any such 2-cocycle untwists. Combined with \cite{MR842413}, this proves that there is a unique trace-preserving cocycle action of $\mathcal{G}$ on $B$, or equivalently that two such inclusions $B_{i} \subset R$ that have isomorphic associated groupoids are  conjugate by an automorphism of $R$. 

\textit{First main theorem and a sketch of the proof:} The results in \cite{MR4124434} motivated us to study the classification problem for such subalgebras of a general injective factor (typically of type $\textrm{III}$). We investigate when two regular subalgebras $B_{i} \subset M$ for an injective factor $M$ are conjugate by an automorphism of $M$.  The first main result of this paper settles this question in the case when they satisfy a relative commutant condition at the level of the continuous cores. In this case, the invariant is the Connes-Takesaki module map of the associated discrete measured groupoid action. To be more precise, in Theorem \ref{RegularInclusionsClassification}, we prove the following:

\begin{namedtheorem}[Theorem A]
    Let $M$ be an injective factor and for $i \in \{1,2\}$, let $B_{i} \subset M$ be regular subalgebras with conditional expectation $E_{i}: M \rightarrow B_{i}$ and satisfying $\mathcal{Z}(\widetilde{B}_{i}) = \widetilde{B}_{i} \cap \widetilde{M}$. Let $\mathcal{G}_{i}$ be the associated discrete measured amenable groupoids and $(\alpha_{i},u_{i})$ be the associated cocycle actions on $B_{i}$. Then there exists an automorphism $\theta \in \Aut(M)$ such that $\theta(B_{1}) = B_{2}$ if and only if $\type(B_{1}) = \type(B_{2})$ and there exists a groupoid isomorphism $\sigma: \mathcal{G}_{1} \rightarrow \mathcal{G}_{2}$ which conjugates the modules, i.e.,  $\modnew(\alpha_{1}) \sim_{\sigma} \modnew(\alpha_{2})$. 
\end{namedtheorem}

We remark here that for such a regular subalgebra $B$ of a factor $M$, almost every factor in the ergodic decomposition of $B$ has a constant type, that we denote by $\type(B)$. We explain this rigorously in Lemma \ref{constanttype} and Definition \ref{defconstanttype}. In Lemma \ref{CentrallyFreeCondition}, we show that the relative commutant condition on the canonical cores as above is equivalent to the associated actions being centrally free, in the sense of Ocneanu \cite{MR807949}. 
Therefore the proof of Theorem A boils down to a 2-cohomology vanishing problem and the classification of centrally free actions of a discrete measured amenable groupoid on a measurable field of injective factors. We provide in Theorem \ref{GroupoidActionsClassification} a new proof of this classification result which is a special case of the more general classification of free actions due to Masuda \cite{MR4484234}. For the proof, we first reduce the problem to the classification of genuine actions, noting in Theorem \ref{2CohomologyVanishingGroupoidsAlternate} that any 2-cocycle on an infinite factor is a coboundary by using a Packer-Raeburn \cite{MR1066817} type stabilization trick. Then we apply the known classification results up to cocycle conjugacy for amenable group actions on injective factors (\cite{MR807949}, \cite{MR842413}, \cite{MR990219}, \cite{MR1179014}, \cite{MR1621416}, \cite{MR2353241}) to the isotropy groups $\Gamma_{x} = \{g \in \mathcal{G} \; | \; s(g) = t(g) = x\}$ of an amenable groupoid $\mathcal{G}$. To extend this to actions of the groupoid, we need to make equivariant choices of cocycle conjugacies for the isotropy groups with respect to isomorphisms $\Gamma_{y} \rightarrow \Gamma_{x}$ given by conjugating with a groupoid element $g$ such that $s(g) = x$ and $t(g) = y$ as in \cite{MR4124434}.  We use a technical cohomology lemma (\cite[Theorem 3.5]{MR4124434}) to show that such an equivariant choice exists if every cocycle self conjugacy for each $\Gamma_{x}$ is approximately inner. More precisely, suppose $\alpha$ is an action of an amenable group $G$ on an injective factor $M$. Suppose $c$ is a 1-cocycle for $\alpha$ and $\theta \in \Aut(M)$ such that $\theta \circ \alpha_{g} \circ \theta^{-1} = \Ad(c_{g}) \circ \alpha_{g}$. Then we show the existence of a sequence of unitaries $w_{n} \in M$ such that $\Ad(w_{n}) \rightarrow \theta$ in the u-topology of $\Aut(B)$ and $w_{n}^{*}\alpha_{g}(w_{n}) \rightarrow c_{g}$ in the topology of pointwise convergence in the space of 1-cocycles of $\alpha$. We prove this in Lemma \ref{CocycleSelfConjugacyNonCartan} by using some ultrapower techniques and Ocneanu's 1-cohomology vanishing result (\cite[Proposition 7.2]{MR807949}). 

\textit{Regular subalgebras containing a Cartan subalgebra:} The next part of this paper deals with a relative version of Theorem A, similar to the case of the hyperfinite $\textrm{II}_{1}$ factor in \cite{MR4124434}. In Proposition \ref{StrongNormality}, we prove that if $A \subset M$ is a Cartan subalgebra of an injective factor, then an intermediate subalgebra $A \subset B \subset M$ is regular in $M$ if and only if the Feldman-Moore equivalence relation $\mathcal{R}_{A \subset B} \subset \mathcal{R}_{A \subset M}$ is strongly normal in the sense of \cite{MR1007409}. The question that we are interested in then can be stated as follows. Given two such inclusions $A_{i} \subset B_{i} \subset M$, with $B_{i}$ regular, when are they conjugate by an automorphism in $M$, i.e., there exists $\theta \in \Aut(M)$ with $\theta(B_{1}) = B_{2}$ and $\theta(A_{1}) = A_{2}$. As before, this boils down to classifying actions of the associated discrete measured groupoids $\mathcal{G}_{B_{i} \subset M}$ on measurable fields of Cartan inclusions into injective factors and thus on their associated ergodic equivalence relations up to cocycle conjugacy.  

\textit{Historical overview of the classification of group and groupoid actions on ergodic equivalence relations:} Connes and Krieger in \cite{MR0444900} gave a classification up to outer conjugacy, of outer automorphisms of an ergodic hyperfinite equivalence relation preserving a finite (type $\textrm{II}_{1}$) or infinite (type $\textrm{II}_{\infty}$) measure. This was extended to type $\textrm{III}$ (with no invariant measure) ergodic hyperfinite equivalence relations by Bezuglyi and Golodets in \cite{MR780502}. The next step was to extend the classification results to outer actions of amenable groups. This became possible after the existence and uniqueness problem of cocycles of hyperfinite equivalence relations with dense ranges in amenable groups was solved in \cite{golodets1983existence} by Golodets and Sinelshchikov. The complete set of invariants for amenable group actions on ergodic hyperfinite equivalence relations up to cocycle conjugacy was then found by Bezuglyi and Golodets in \cite{MR864169} (for measure preserving equivalence relations), \cite{MR906080} (for type $\textrm{III}_{\lambda}$ with $\lambda \neq 0$) and finally in \cite{MR1002120} (for the type $\textrm{III}_{0}$ case). Recently in \cite{https://doi.org/10.48550/arxiv.2210.15916}, Masuda gave a new unified proof of the classification theorem. For the convenience of the reader, we state this result as Theorem \ref{conjugaceyOfGroupActionsOnEquivalenceRelation} in this paper.

\textit{Second main theorem and sketch of the proof:} In \cite{MR4124434} the authors showed that two cocycle actions of a discrete measured groupoid $\mathcal{G}$ with unit space $X$ on a measurable field of Cartan inclusions $(A_{x} \subset B_{x})_{x \in X}$, such that the action globally preserves the field of Cartan subalgebras $(A_{x})_{x \in X}$ and with each $B_{x}$ being the hyperfinite $\textrm{II}_{1}$ factor such that the induced actions on the field of ergodic equivalence relations $(\mathcal{S}_{x})_{x \in X}$ are outer, are cocycle conjugate. This in turn shows that two pairs of triple inclusions $A_{i} \subset B_{i} \subset R$ with $A_{i} \subset R$ being Cartan subalgebras and $B_{i} \subset R$ being regular are  conjugate an automorphism in $M$ if and only if the associated groupoids are isomorphic. In this paper, we prove a more general classification theorem for any field of injective factors with Cartan subalgebras and show that the invariant is the action on the field of flows associated with the equivalence relations. More precisely, in Theorem \ref{classification of inclusionsCartan}, we prove the following: 

\begin{namedtheorem}[Theorem B]
    Let $M$ be an injective factor and for $i \in \{1,2\}$, let $A_{i} \subset B_{i} \subset M$ be a pair of inclusions, where $A_{i} \subset M$ are Cartan subalgebras and $B_{i} \subset M$ are regular. Let $\mathcal{R}_{i}$ be the corresponding Feldman-Moore equivalence relations for the Cartan inclusions $A_{i} \subset B_{i}$. Let $\mathcal{G}_{i}$ be the discrete measured amenable groupoids associated to the inclusions $B_{i} \subset M$ and $(\alpha_{i},u_{i})$ be the associated cocycle actions on $\mathcal{R}_{i}$. Then there exists an automorphism $\theta \in \Aut(M)$ such that $\theta(B_{1}) = B_{2}$ and $\theta(A_{1}) = A_{2}$ if and only if $\type(B_{1}) = \type(B_{2})$ and there exists a groupoid isomorphism $\sigma: \mathcal{G}_{1} \rightarrow \mathcal{G}_{2}$ which conjugates the actions on the flows, i.e.,  $\modnew(\alpha_{1}) \sim_{\sigma} \modnew(\alpha_{2})$. 
\end{namedtheorem}

We follow a similar strategy as earlier to prove Theorem B. In Lemma \ref{CohomologyVanishingGroupoidsCartanAlternate}, we show that for an outer cocycle action of a discrete measured groupoid on a field of infinite equivalence relations, the cocycle is a coboundary, hence reducing the problem to the classification of genuine actions. We use the Bezuglyi-Golodets classification results for actions of amenable groups and apply them to the isotropy groups $\Gamma_{x}$ of the groupoid. Once again we can use \cite[Theorem 3.5]{MR4124434} to extend this to actions of amenable groupoids if we can show that every cocycle self conjugacy for $\Gamma_{x}$ is approximately inner. In this context of Cartan inclusions, we formulate this more precisely as follows. Let $\alpha$ be an action of an amenable group $G$ on an injective factor $M$ such that it preserves a Cartan subalgebra $A \subset M$ and such the associated action on the equivalence relation in outer. Let $c$ be a 1-cocycle for $\alpha$ taking values in $\mathcal{N}_{M}(A)$ and $\theta \in \Aut(M)$ with $\theta(A) = A$ such that $\theta \circ \alpha_{g} \circ \theta^{-1} = \Ad(c_{g}) \circ \alpha_{g}$. Then in Lemma \ref{approximateInnernessOfSelfConjugacyCartan} we show that there exists a sequence of unitaries $w_{n} \in \mathcal{N}_{M}(A)$ such that $\Ad(w_{n}) \rightarrow \theta$ in the u-topology in $\Aut(M)$ and $w_{n}^{*}c_{g}\alpha_{g}(w_{n}) \rightarrow c_{g}$ in the topology of pointwise convergence. In Theorem \ref{CocycleCOnjugacyForActionsOnEquivalenceRelations} we thus show that two actions of an amenable groupoid on a field of Cartan inclusions are cocycle conjugate if and only if the actions on the associated field of flows are conjugate. We note that this is a generalization of \cite[Theorem 2.4]{https://doi.org/10.48550/arxiv.2210.15916} for amenable groupoid actions on fields of ergodic equivalence relations. 

\textit{Model actions and third main theorem:} In the final section, we construct examples of actions realizing the invariants. We use the notion of adjoint flows introduced by Vaes and Verjans in \cite{vaes_verjans_2022}, where they construct for any given ergodic flow $(F_{t})_{t \in \mathbb{R}}$, a canonical countable ergodic equivalence relation $\mathcal{R}_{F}$ whose associated flow is the adjoint $\hat{F}_{t}$ of the given flow. We show in Proposition \ref{modelAutomorphism} that any $\mathbb{R}$-equivariant nonsingular isomorphism $\psi: F \rightarrow F'$ between two flows induces a canonical `adjoint' isomorphism $\hat{\psi}: \hat{F} \rightarrow \hat{F'}$ between their adjoint flows and canonically lifts to a nonsingular isomorphism between the ergodic equivalence relations $\Psi: \mathcal{R}_{F} \rightarrow \mathcal{R}_{F'}$ which induces the adjoint isomorphism on the associated flows, i.e., $\modnew(\Psi) = \hat{\psi}$. We note in Theorem \ref{modelAction} that the construction of the adjoint flow is sufficiently canonical to extend to groupoid actions on fields of flows. As an immediate corollary, we can construct regular subalgebras realizing the invariants obtained in Theorem B as follows: 

\begin{namedtheorem}[Theorem C]
    Let $\mathcal{G}$ be a discrete measured ergodic groupoid that is amenable with unit space $(X,\mu)$. Let $(F_{x})_{x \in X}$ be a field of ergodic flows such that the set of points $x \in X$ for which $F_{x}$ is the translation action $\mathbb{R} \curvearrowright \mathbb{R}$ has measure zero. Let $\psi$ be an action of $\mathcal{G}$ on the field of flows $(F_{x})_{x \in X}$. Then there exists an injective factor $M$ with a Cartan subalgebra $A \subset M$ and an intermediate regular subalgebra $A \subset B \subset M$ of type $\textrm{III}$ such that there is a groupoid isomorphism $\sigma: \mathcal{G}_{B \subset M} \rightarrow \mathcal{G}$ and the associated cocycle action $(\Psi,u)$ of $\mathcal{G}_{B \subset M}$ on the field of ergodic equivalence relations given by $A \subset B$ satisfy $\modnew(\Psi) \sim_{\sigma} \psi$. 
\end{namedtheorem}

We remark that the triple inclusion $A\subset B \subset M$ in the statement of the theorem is unique in the sense that if there exists another triple $A_{1}\subset B_{1} \subset M_{1}$ satisfying the conclusion of the theorem, there is a $*$-isomorphism $\theta: M \rightarrow M_{1}$ such that $\theta(B) = B_{1}$ and $\theta(A) = A_{1}$, essentially by Theorem B. Now, Theorem A states that for a regular subalgebra $B \subset M$, if $B$ satisfies the relative commutant condition at the level of the continuous cores, then it is uniquely determined by its associated groupoid and the action on the associated field of flows. Since by Theorem \ref{modelAction}, any discrete measured amenable groupoid admits a model action on $B$ fixing a Cartan subalgebra $A$ of $M$, we prove in \ref{CartanIffRelativeCommutant} the following corollary to Theorem A and Theorem C. 

\begin{namedtheorem}[Corollary D]
    Let $M$ be an injective factor and $B$ be a regular subalgebra of $M$ of type $\textrm{III}$ with a conditional expectation $E:M \rightarrow B$. Then $\mathcal{Z}(\widetilde{B}) = \widetilde{B}' \cap \widetilde{M}$ if and only if there exists a Cartan subalgebra $A$ of $M$ such that $A \subset B \subset M$. 
\end{namedtheorem}

Most of our results are type III generalizations of the type $\textrm{II}_{1}$ results in \cite{MR4124434}. Therefore Sections \ref{sectionnoncartan} and \ref{sectioncartan} in our paper are structured similarly as \cite{MR4124434}. The main difficulty and novelty in our paper stems from the fact that outer actions of amenable groups on type $\textrm{III}$ factors are no longer unique, but classified by modular data that we have to handle. For clarity, we set up our notation and terminology in a similar way as the authors of \cite{MR4124434}.

\section{Preliminaries}
\label{prelim}
\subsection{Injective factors and ergodic decompositions }
We start with some basic definitions and known results from the theory of injective factors. For all results in this paper, we shall only consider von Neumann algebras acting on separable Hilbert spaces.  

\begin{definition}
A von Neumann algebra $M$ acting on a separable Hilbert space $\mathcal{H}$ is called \textit{injective} if there exists a projection $E: \mathcal{B}(\mathcal{H}) \rightarrow M$ with $\|E\| = 1$. It follows from this that $E \geq 0$ and by \cite{MR96140}, $E$ is a (not necessarily normal) conditional expectation.
\end{definition}

\begin{definition}
    A von Neumann subalgebra $B \subset M$ is called \textit{regular} if the von Neumann algebra generated by the normalizer $\mathcal{N}_{M}(B) = \{u \in \mathcal{U}(M) \; | \; uBu^{*} = B\}$ is equal to $M$.  
\end{definition}

Connes in \cite[Theorem 6]{MR454659} showed that a von Neumann algebra is injective if and only if it is approximately finite dimensional (AFD), i.e., it is the weak closure of an increasing union of finite dimensional von Neumann subalgebras. Clearly each type $\textrm{I}$ factor is hyperfinite. Murray and von Neumann in \cite{MR9096} already established the uniqueness of the hyperfinite factor of type $\textrm{II}_{1}$. The bulk of the remaining work in the classification of injective factors was also done in \cite{MR454659}, which established the uniqueness of the injective $\textrm{II}_{\infty}$ factor \cite[Corollary 4]{MR454659} and the uniqueness of the injective type $\textrm{III}_{\lambda}$ factors, for $\lambda \in (0,1)$ \cite[Theorem 8]{MR454659}, proving that each of them is isomorphic to the Powers' factor $R_{\lambda}$ \cite{MR218905}. He also showed that injective type $\textrm{III}_{0}$ factors are classified by their ergodic flow of weights, we shall describe this in further detail in this section. The final step of the classification program was completed by Haagerup in  \cite{MR880070} where he proved that any injective factor of type $\textrm{III}_{1}$ is isomorphic to the Araki-Woods factor $R_{\infty}$ \cite{MR0244773}.  

We now recall the construction of the Connes-Takesaki crossed product \cite{MR480760}. We begin with a faithful, normal state $\phi$ and let $\sigma^{\phi}_{t}$ be the modular automorphism group of $\phi$. Looking at $\sigma^{\phi}_{t}$ as an action of $\mathbb{R}$ on $M$, we construct the crossed product $M \rtimes_{\sigma^{\phi}} \mathbb{R}$ represented on the Hilbert space $L^{2}(\mathcal{H}, \mathbb{R}) = \mathcal{H} \otimes L^{2}(\mathbb{R})$ where $M \subseteq \mathcal{B}(\mathcal{H})$. We recall that for $\xi \in L^{2}(\mathcal{H},\mathbb{R})$ and $x \in M$, the crossed product is generated by the operators $\{ \pi_{\phi}(x)\; | \; x \in M\}$ and $\{u_{t} \; | \; t \in \mathbb{R}\}$ which acts as follows:
\begin{equation}
    (\pi_{\phi}(x)\xi)(t) = \sigma^{\phi}_{-t}(x)\xi(t), \nonumber
\end{equation}
\begin{equation}
    (u_{t}\xi)(s) = \xi(s-t). \nonumber
\end{equation}

Using Connes' Radon-Nikodym cocycle, one can show that this construction is canonical, and it does not depend on the choice of the state. For a factor of type \textrm{III}, we call this crossed product the \textit{canonical continuous core} of $M$ and denote it by $\widetilde{M} = M \rtimes_{\sigma^{\phi}} \mathbb{R}$. In \cite{MR438149}, Takesaki proved the following:

\begin{theorem}(See \cite[Theorem 4.5]{MR438149})
Let $M$ be a factor of type \textrm{III} and $\phi$ be a faithful normal state on $M$, such that $\widetilde{M}$ is the crossed product $M \rtimes_{\sigma^{\phi}} \mathbb{R}$. Then $\widetilde{M}$ is a $\textrm{II}_{\infty}$ von Neumann algebra and there exists a semifinite faithful normal trace $\tau$ on $\widetilde{M}$ and a one-parameter group of automorphisms $(\theta_{s})_{s \in \mathbb{R}}$ such that: 
\begin{enumerate}[(i)]
    \item $\tau \circ \theta_{s} = e^{-s}\tau$,
    \item $M$ is isomorphic to $\widetilde{M} \rtimes_{\theta} \mathbb{R}$
\end{enumerate}
\end{theorem}

By the Connes-Takesaki relative commutant theorem \cite{MR480760}, we know that $C_{M} := \mathcal{Z}(\widetilde{M}) = \widetilde{M}\cap M'$ . As in \cite{MR480760} we denote for a type \textrm{III} factor $M$, the pair $(C_{M}, \theta|_{C_{M}})$ as the \textit{smooth flow of weights} on $M$, and the quartet $(\widetilde{M}, \theta, \mathbb{R}, \tau)$ as the \textit{non-commutative flow of weights} on $M$. Two important properties of the non-commutative flow of weights are:
\begin{enumerate}
    \item The fixed point algebra $\widetilde{M}^{\theta} = M$. 
    \item If $M$ is a factor, $\theta_{s}$ is ergodic on $C_{M}$ for all $s \in \mathbb{R}$. 
\end{enumerate}
 Writing $C_{M} = L^{\infty}(X)$, $\theta_{s}|_{C_{M}}$ can be written as a point-map $F_{s}$ on $X$. We shall call $(X,F_{s})$ the \textit{flow of weights} in what follows. In terms of Connes' classification of type \textrm{III} factors \cite{MR341115},  the flow of weights can be interpreted as follows: 
\begin{enumerate}
    \item $M$ is of type $\textrm{III}_{1}$ if $X$ is a singleton and the flow is trivial.
    \item $M$ is of type $\textrm{III}_{\lambda}$, for $\lambda \in (0,1)$ if $F_{s}$ is periodic with period $-2\pi \log(\lambda)$. 
    \item $M$ is of type $\textrm{III}_{0}$ otherwise.
\end{enumerate}

We define the right notions of isomorphisms of flows and weak equivalences of automorphisms before stating the final classification in the type $\textrm{III}_{0}$ case due to Krieger \cite{MR415341} and Connes \cite{MR454659}.
\begin{definition}
    1. A one-to-one bimeasurable map of a Lebesgue measure space $(X,\mathcal{B},\mu)$ onto a Lebesgue measure space $(X', \mathcal{B}', \mu')$ that carries $\mu$-null sets to $\mu'$-null sets will be called an \textit{isomorphism}. An isomorphism $T:(X,\mathcal{B},\mu) \rightarrow (X,\mathcal{B},\mu)$ will be called an \textit{automorphism} of the measure space.
    
    2. Given two countable group actions $G \curvearrowright (X,\mu)$ and $G' \curvearrowright (X', \mu')$, they are called orbit equivalent and denoted $G \sim_{w} G'$ if there exists an isomorphism $U:(X,\mu) \rightarrow (X', \mu')$ such that for a.e. $x \in X$, we have $U(G \cdot x) = G'\cdot U(x)$. Given a measure space $(X,\mu)$, two automorphisms $T$ and $T'$ are called orbit (or weakly) equivalent and denoted $T \sim_{w} T'$ if they generate weakly equivalent groups of automorphisms. 
\end{definition}

\begin{definition}
    Two ergodic flows $F_{s}$ and $F'_{s}$ on measure spaces $(X,\mu)$ and $(X', \mu')$ respectively are called \textit{isomorphic} if there is an isomorphism $T: X \rightarrow X'$ such that for all $s\in \mathbb{R}$ and for a.e. $x \in X$, $F'_{s}\circ T(x) = T \circ F_{s}(x)$. 
\end{definition}

It was shown in \cite{MR415341} and \cite{MR454659} that any two injective factors of type $\textrm{III}_{0}$ are isomorphic if and only if their associated ergodic flow of weights are isomorphic.

Since we shall be considering subalgebras of injective factors in this paper, we are going to use the central decomposition extensively. Given a separable Hilbert space $\mathcal{H}$, we denote the set of all von Neumann algebras acting on $\mathcal{H}$ by $\vn(\mathcal{H})$. This space has a certain natural topological structure. This goes back to the works of Effros \cite{MR192360} and Marachel \cite{MR317063} and is defined as the weakest topology on $\vn(\mathcal{H})$ that makes the maps $M \mapsto ||\phi|_{M}||$ continuous for every $\phi \in \mathcal{B}(\mathcal{H})_{*}$. The corresponding Borel structure is known as the Effros-Borel structure on $\vn(\mathcal{H})$. The works of Effros \cite{MR185456}, \cite{MR192360} and Marachel \cite{MR317063} and subsequent works by Nielsen \cite{MR333750}, Sutherland \cite{MR461161} and Woods \cite{MR361817} concluded that the factors of type $\textrm{I}_{n}, n \in \mathbb{N}$, $\textrm{I}_{\infty}$, $\textrm{II}_{1}$, $\textrm{II}_{\infty}$, $\textrm{III}_{\lambda}, \lambda \in [0,1]$ are all Borel sets with respect to this structure. By \cite[Theorem 2.2]{MR461161}, given a Borel set $B \subset [0,1]$, the set of factors $\{M \in \vn(\mathcal{H}) \; |\; M \text{ is of type } \textrm{III}_{\lambda} \text{ where } \lambda \in B \}$ is Borel. Suppose $M \subset \mathcal{B}(\mathcal{H})$ is an injective factor of type $\textrm{III}$ and suppose $B \subset M$ is a regular subalgebra. Let $\mathcal{Z}(B) = L^{\infty}(X,\mu)$ for some standard probability space. Then we have an ergodic central decomposition of $B$ into factors, namely $B = \int^{\oplus}_{x \in X} B_{x}$ with $B_{x} \subset \mathcal{B}(\mathcal{H}_{x})$, such that $\mathcal{H} = \int_{x \in X}^{\oplus} \mathcal{H}_{x}$ and the map $X \ni x \mapsto B_{x} \in \vn(\mathcal{H}_{x})$ is a Borel map with respect to the Effros structure.

\cite[Proposition 6.5]{MR454659} gives us immediately that if $B$ has central decomposition $B = \int^{\oplus}_{x\in X} B_{x}$ where $\mathcal{Z}(B) = L^{\infty}(X,\mu)$, then $B_{x}$ is an injective factor for a.e $x \in X$. We now prove the following lemma that gives a necessary condition for a direct integral of factors to be non-constant. 

\begin{lemma}
\label{constanttype}
Suppose $M$ is an injective factor and $B \subset M$ be a regular subalgebra satisfying $\mathcal{Z}(B) = B' \cap M$ with a conditional expectation $E:M \rightarrow B$ and central decomposition $B = \int^{\oplus}_{x\in X} B_{x}$. Then the measurable field $(B_{x})_{x\in X}$ can be non-constant only if $B_{x}$ is of type $\textrm{III}_{0}$ for a.e. $x \in X$.   
\end{lemma}

\begin{proof}
Since the inclusion is regular, we have a discrete measured amenable groupoid $\mathcal{G}$ with a cocycle action on $B$ such that $B \rtimes \mathcal{G}$ is isomorphic to $M$. Let $\mathcal{R}$ be the countable amenable equivalence relation on $(X,\mu)$ given by $x \sim y$ if and only if there exists $g \in \mathcal{G}$ with $s(g) = x$ and $t(g) = y$. Since $M$ is a factor, $\mathcal{R}$ is ergodic. Let $E$ denote the set $\{\textrm{I}_{n}, n \in \mathbb{N}\} \cup \{\textrm{I}_{\infty}, \textrm{II}_{1}, \textrm{II}_{\infty}\}  \cup \{\textrm{III}_{\lambda}, \lambda \in [0,1]\}$. Let $\mathcal{F}(\mathcal{H}) \subset \vn(\mathcal{H})$ denote the set of factors acting on a separable Hilbert space $\mathcal{H}$. By \cite[Theorem 2.2]{MR461161}, the map $\type: \mathcal{F}(\mathcal{H}) \rightarrow E$ is Borel. Now we consider the map from $X$ to $E$ given by $x \mapsto \type(B_{x})$. Clearly this is $\mathcal{R}$-invariant as when $(x,y) \in \mathcal{R}$, we have that $B_{x} \cong B_{y}$ and hence they clearly have the same type. Hence this map is essentially constant, as required.    
\end{proof}

\begin{definition}
\label{defconstanttype}
    A regular von Neumann subalgebra $B \subset M$ of an injective factor $M$ satisfying $\mathcal{Z}(B) = B'\cap M$ with central decomposition $B = \int^{\oplus}_{x \in X} B_{x}$ is said to be of type $k$ with $k \in \{\textrm{I}_{n} \text{ for } n\in \mathbb{N}, \textrm{I}_{\infty}, \textrm{II}_{1}, \textrm{II}_{\infty}, \textrm{III}\}$ if for a.e. $x \in X$, $B_{x}$ is of type $k$.    
\end{definition}

\subsection{Amenable groupoids and cocycle actions}

The concept of discrete measured groupoids goes back to \cite{MR165034}. All definitions in this section are taken from \cite{MR1855241}. 

\begin{definition} 1. A \textit{groupoid} is a collection of morphisms $\mathcal{G}$ between a set of objects $\mathcal{G}^{(0)}$ such that every morphism is invertible along with the following four canonical maps: 
\begin{itemize}
    \item The source map $s: \mathcal{G} \rightarrow \mathcal{G}^{(0)}$ sending $(f:x \rightarrow y)$ to $x$. 
    \item The target map $t: \mathcal{G} \rightarrow \mathcal{G}^{(0)}$ sending $(f:x \rightarrow y)$ to $y$.
    \item The inverse map $i: \mathcal{G} \rightarrow \mathcal{G}$ sending $(f:x \rightarrow y)$ to $(f^{-1}:y \rightarrow x)$. 
    \item The composition map $\circ : \mathcal{G}^{(2)} \rightarrow G$ sending $(f,g)$ to $g \circ f$ where $\mathcal{G}^{(2)} = \{(f,g) | f,g \in \mathcal{G}, s(g) = t(f)\}$. 
\end{itemize}
We shall often denote $g \circ f$ as $gf$ for convenience. The unit space $\mathcal{G}^{(0)}$ can be seen as a subset of $\mathcal{G}$ via identity morphisms.

2. A \textit{discrete measurable groupoid} is a groupoid $\mathcal{G}$ equipped with a structure of a standard Borel space such that the inverse and composition maps are Borel, and $s^{-1}(x)$ (and hence also $t^{-1}(x)$) is countable for every $x \in \mathcal{G}$.  
\end{definition}

We note here some important properties of discrete measurable groupoids that are fairly straightforward: 

\begin{itemize}
    \item The source and target maps of a discrete measurable groupoid are measurable and $\mathcal{G}^{(0)} \subset \mathcal{G}$ is a Borel subset.
    \item  Let $\mu$ be a probability measure on the set of objects $\mathcal{G}^{(0)}$ of a discrete measurable groupoid $\mathcal{G}$. Then for any measurable subset $A \subset \mathcal{G}$, the function 
\begin{equation}
    \mathcal{G}^{(0)} \rightarrow \mathbb{C}, \; x \rightarrow \# (s^{-1}(x) \cap A) \nonumber
\end{equation}
is measurable and the measure $\mu_{s}$ on $\mathcal{G}$ defined by 
\begin{equation}
    \mu_{s}(A) = \int_{\mathcal{G}^{(0)}} \#(s^{-1}(x) \cap A)d\mu(x)  \nonumber
\end{equation}
is $\sigma$-finite. We call it the \textit{left counting measure} of $\mu$. The same statement holds if we replace $s$ by $t$ and the corresponding measure $\mu_{t}$ is called the \textit{right counting measure} of $\mu$.

\item The following conditions on $\mu$ (as above) are equivalent: 
\begin{enumerate}
    \item $\mu_{s} \sim \mu_{t}$.
    \item $i_{*}(\mu_{s}) \sim \mu_{s}$.
    \item  For every Borel subset $\mathcal{E} \subset \mathcal{G}$ such that $s|_{\mathcal{E}}$ and $t|_{\mathcal{E}}$ are injective, we have that  $\mu(s(\mathcal{E})) = 0$ if and only if $\mu(t(\mathcal{E})) = 0$.
\end{enumerate} 
\end{itemize} 

\begin{definition}
 A probability measure $\mu$ on $\mathcal{G}^{(0)}$ is called \textit{quasi-invariant} if it satisfies one of the (and hence all) conditions of Lemma 1.5. A \textit{discrete measured groupoid} is a discrete measurable groupoid $\mathcal{G}$ together with a quasi-invariant measure $\mu$ on $\mathcal{G}^{(0)}$. Let $\mathcal{G}$ be a discrete measured groupoid with quasi-invariant measure $\mu$ and $A \subset \mathcal{G}^{(0)}$ be a Borel subset. Then $\mathcal{G}|_{A} = s^{-1}(A) \cap t^{-1}(A)$ with the normalized measure $\frac{1}{\mu(A)}\mu|_{A}$ is a discrete measured groupoid called the \textit{restriction} of $\mathcal{G}$ to $A$.
\end{definition}  

A functor between groupoids is called a \textit{groupoid homomorphism}. A groupoid homomorphism with an inverse is called an \textit{isomorphism}. A measurable groupoid isomorphism between discrete measurable groupoids has a measurable inverse. A measurable homomorphism $f: \mathcal{G} \rightarrow \mathcal{H}$ between discrete measured groupoids with quasi-invariant measures $\mu_{\mathcal{G}^{(0)}}$ and $\mu_{\mathcal{H}^{(0)}}$ on $\mathcal{G}^{(0)}$ and $\mathcal{H}^{(0)}$ is called \textit{non-singular} if $f_{*}\mu_{\mathcal{G}^{(0)}} \sim \mu_{\mathcal{H}^{(0)}}$. An \textit{isomorphism of discrete measured groupoids} is a non-singular groupoid isomorphism.

A discrete measured groupoid $\mathcal{G}$ with unit space $\mathcal{G}^{(0)} = X$ has two important associated objects:
\begin{itemize}
    \item The associated \textit{isotropy bundle} $\Gamma = (\Gamma_{x})_{x \in X}$ where $\Gamma_{x} = \{g \in \mathcal{G}, \; s(g) = t(g) = x\}$ is called an \textit{isotropy group} of $x$.
    \item The associated nonsingular equivalence relation $\mathcal{R}$ on $(X,\mu)$ given by $x \sim y$ if and only if there exists an element $g \in \mathcal{G}$ with $s(g) = x$ and $t(g) = y$. 
\end{itemize}
 
We recall here from Section 3.2 in \cite{MR1855241} the various equivalent notions of amenability of a discrete measured groupoid. We also recall from \cite[Section 5.3]{MR1855241}, that amenability of a discrete measured groupoid is equivalent to amenability of the associated equivalence relation and amenability of the isotropy groups $\Gamma_{x}$ for a.e. $x \in X$. When the associated equivalence relation is ergodic, we say that the groupoid is \textit{ergodic}.  

For a discrete measured groupoid $\mathcal{G}$, we define $\mathcal{G}^{(3)} = \{(g,h,k) \; | \; t(g) = s(h), t(h) = s(k)\}$. In the rest of the subsection, we define `cocycle actions' and `crossed products' exactly in the same way as in the case of p.m.p. groupoids and trace-preserving actions in \cite[Section 2]{MR4124434}. The notation used in this section is very similar to the notation used in \cite{MR4124434}.    

\begin{definition}(See \cite[Definition 2.2]{MR4124434})
A \textit{cocycle action} $(\alpha, u)$ of a discrete measured groupoid $\mathcal{G}$ with unit space $\mathcal{G}^{(0)} = X$ on a measurable field $(B_{x})_{x \in X}$ of von Neumann algebras is given by a measurable field of $*$-isomorphisms between the source and the targets $\mathcal{G} \ni g \mapsto \alpha_{g}: B_{s(g)} \rightarrow B_{t(g)}$ and a 2-cocycle, i.e. a measurable field of unitaries $\mathcal{G}^{(2)} \ni (g,h) \mapsto u(g,h) \in \mathcal{U}(B_{t(g)})$ satisfying the following conditions:
\begin{itemize}
\item  $\alpha_{g} \circ \alpha_{h} = \Ad(u(g,h)) \circ \alpha_{gh}$ for all $(g,h) \in \mathcal{G}^{(2)}$. 
\item $\alpha_{g}(u(h,k))u(g,hk) = u(g,h)u(gh,k)$ for all $(g,h,k) \in \mathcal{G}^{(3)}$.
\item $\alpha_{g} = \id$ when $g \in \mathcal{G}^{(0)}$.
\item $u(g,h) = 1$ when $g \in \mathcal{G}^{(0)}$ or $h \in \mathcal{G}^{(0)}$.
\end{itemize}
\end{definition}

Recall that an automorphism $\alpha \in \Aut(B)$ is outer if there is no unitary element $v \in \mathcal{U}(B)$ such that $\alpha = \Ad(v)$. We shall call an automorphism $\alpha \in \Aut(B)$ \textit{free} or \textit{properly outer} as in \cite{MR4124434}, if for every nonzero element $v \in B$ (not necessarily a unitary anymore), there exists $x \in B$ such that $vx \neq\alpha(x)v$. It is easy to check that an automorphism is outer if and only if it is properly outer when $B$ is a factor. In the context of cocycle actions, we call $(\alpha,u)$ \textit{free} if $\alpha_{g}$ is free for a.e. $x \in X$ and for all $g \neq \id$ in the isotropy groups $\Gamma_{x}$. We repeat the definition of cocycle conjugacy as in \cite[Section 2]{MR4124434}.

\begin{definition}
\label{DefCocycleConjugacy}
\begin{enumerate}[1.]
    \item Two cocycle actions $(\alpha,u)$ and $(\beta,v)$ of a countable group $G$ on von Neumann algebras $B$ and $D$ respectively are said to be cocycle conjugate if there exists a *-isomorphism $\theta:B \rightarrow D$ and unitaries $w_{g} \in \mathcal{U}(D)$ such that
    \begin{itemize}
        \item $\theta \circ \alpha_{g} \circ \theta^{-1}= \Ad(w_{g})\circ \beta_{g}$ for all $g \in G$.
        \item $\theta(u(g,h)) = w_{g}\beta_{g}(w_{h})v(g,h)w_{gh}^{*}$ for all $g,h \in G$.
    \end{itemize}
    
    \item Two cocycle actions $(\alpha,u)$ and $(\beta,v)$ of a discrete measured groupoid $\mathcal{G}$ on fields of von Neumann algebras $(B_{x})_{x \in X}$ and $(D_{x})_{x \in X}$ respectively are said to be \textit{cocycle conjugate} if there exists a measurable field of $*$-isomorphisms $x \mapsto \theta_{x}: B_{x} \rightarrow D_{x}$ and a measurable field of unitaries $g \mapsto w_{g} \in \mathcal{U}(D_{t(g)})$ that satisfies:
 \begin{itemize}
     \item  $\theta_{t(g)} \circ \alpha_{g} \circ \theta^{-1}_{s(g)} = \Ad(w_{g}) \circ \beta_{g}$ for all $g \in \mathcal{G}$.
     \item  $\theta_{t(g)}(u(g,h)) = w_{g}\beta_{g}(w_{h})v(g,h)w^{*}_{gh}$ for all $(g,h) \in \mathcal{G}^{(2)}$.
 \end{itemize}
\end{enumerate}

\end{definition}
One important reason for classifying actions up to cocycle conjugacy, is that it is the right notion of equivalence between  actions that makes the corresponding crossed products von Neumann algebras isomorphic. The bijection between regular inclusions and crossed products can be seen in various ways, for example in \cite{MR4266136} which uses inverse semigroups. We shall use the algebraic correspondence between inverse semigroups and discrete measured groupoids (c.f. \cite{MR1724106}). We shall discuss this in greater detail now. 

Let $\mathcal{G}$ be a discrete measured groupoid with a quasi-invariant probability measure $\mu$ on $\mathcal{G}^{(0)}$. We define the \textit{full pseudogroup} $[[\mathcal{G}]]$ of $\mathcal{G}$ as the set of all Borel subsets $\mathcal{U} \subset \mathcal{G}$ such that the source and the target maps restricted the  $\mathcal{U}$ are injective. We also identify $\mathcal{U}$ and $\mathcal{U}'$ if the symmetric difference is a set of measure zero. The composition of two such subsets $\mathcal{U}$ and $\mathcal{V}$ in $[[\mathcal{G}]]$ is defined as 
\begin{equation}
    \mathcal{U} \cdot \mathcal{V} = \{gh \; | \; g \in \mathcal{U}, h \in \mathcal{V}, s(g) = t(h)\} \nonumber
\end{equation}
We note that $[[\mathcal{G}]]$ is an inverse semigroup with the unique inverse of $\mathcal{U}$ given by $\mathcal{U}^{-1} = \{ g^{-1} \; | \; g \in \mathcal{U}\}$. Clearly, $\mathcal{U}\mathcal{U}^{-1}\mathcal{U} = \mathcal{U}$ and $\mathcal{U}^{-1}\mathcal{U}\mathcal{U}^{-1} = \mathcal{U}^{-1}$. Indeed $\mathcal{U}^{-1}\mathcal{U}$ consists of the identity morphisms between the elements of $s|_{\mathcal{U}}$ and $\mathcal{U}\mathcal{U}^{-1}$ consists of the identity morphisms between elements of $t|_{\mathcal{U}}$.

Suppose now that $(\alpha,u)$ is a cocycle action of $\mathcal{G}$ on the field of von Neumann algebras $B = (B_{x})_{x \in X}$, with $\mathcal{G}^{(0)}= X$. We now describe the \textit{cocycle crossed product} which produces, for every such cocycle action $(\alpha,u)$, a von Neumann algebra $M$ such that $B \subset M$ is regular and comes with faithful normal conditional expectation $E:M \rightarrow B$. We define the cocycle crossed product $M$ from the full pseudogroup $[[\mathcal{G}]]$ of the groupoid. For every $\mathcal{U} \in [[\mathcal{G}]]$, we define the Borel maps $\phi_{\mathcal{U}}: s(\mathcal{U}) \rightarrow t(\mathcal{U})$ by $\phi_{\mathcal{U}}(s(g)) = t(g)$ for all $g \in \mathcal{U}$. Writing 
\begin{equation}
    B = \int_{X}^{\oplus }B_{x} \; d \mu(x), \nonumber
\end{equation}
the cocycle action $(\alpha,u)$ of $\mathcal{G}$ on the field $(B_{x})_{x \in X}$ induces an action of the inverse semigroup $[[\mathcal{G}]]$ as follows: every subset $\mathcal{U} \in [[\mathcal{G}]]$ gives a $*$-isomorphism $\alpha_{\mathcal{U}}: B1_{s(\mathcal{U})} \rightarrow B1_{t(\mathcal{U})}$, namely   
\begin{equation}
    (\alpha_{\mathcal{U}}(b))(t(g)) = \alpha_{g}(b(s(g))) \text{ for all } b \in B1_{s(\mathcal{U})} \text{ and } g \in \mathcal{U}. \nonumber
\end{equation}
The partial isometries $u(\mathcal{U})$ have source projections $u(\mathcal{U})^{*}u(\mathcal{U}) = 1_{s(\mathcal{U})}$, range projections $u(\mathcal{U})u(\mathcal{U})^{*} = 1_{t(\mathcal{U})}$ and expectation $E(u(\mathcal{U}))= 1_{\mathcal{U}\cap \mathcal{G}^{(0)}}$. For an element $b \in B1_{s(\mathcal{U})}$, these partial isometries satisfy $u(\mathcal{U})bu(\mathcal{U})^{*} = \alpha_{\mathcal{U}}(b)$.

Similarly, we have a 2-cocycle for the inverse semigroup action as follows: for every pair of elements $\mathcal{U}, \mathcal{V} \in [[\mathcal{G}]]$, there is a unitary element $u(\mathcal{U},\mathcal{V}) \in B1_{t(\mathcal{U}.\mathcal{V})}$ given by
\begin{equation}
    u(\mathcal{U},\mathcal{V})_{t(gh)} = u(g,h) \text{ for all } g \in \mathcal{U} \text{ and } h \in \mathcal{V} \text{ with } s(g) = t(h). \nonumber
\end{equation}
For two elements $\mathcal{U}, \mathcal{V} \in [[\mathcal{G}]]$, we have that $u(\mathcal{U})u(\mathcal{V}) = u(\mathcal{U},\mathcal{V})u(\mathcal{U}.\mathcal{V})$. We consider the von Neumann algebra generated by $B$ and partial isometries $u(\mathcal{U})$ for all such sets $\mathcal{U} \in [[\mathcal{G}]]$ and call it the `cocycle crossed product'.  

We claim that $B$ is a regular subalgebra of $M$. To prove our claim, let $v$ be a partial isometry in $M$ with $v^{*}v$ and $vv^{*}$ in $\mathcal{Z}(B)$ and satisfying $vBv^{*} \subset B$ and $v^{*}Bv \subset B$. Since for all $\mathcal{U} \subset [[\mathcal{G}]]$, the partial isometry $u(\mathcal{U})$ is of this form, it is enough to prove that $v$ is of the form $uq$ with $u \in \mathcal{N}_{M}(B)$ and $q = v^{*}v \in \mathcal{Z}(B)$. Now let $\{v_{i}, i \in \mathbb{N}\}$ be a maximal family of partial isometries in $M$ such that the source projections $\{p_{i} = v_{i}^{*}v_{i}\}$ are pairwise orthogonal and belong to $\mathcal{Z}(B)$, the range projections $\{q_{i} = v_{i}v_{i}^{*}\}$ are pairwise orthogonal and belong to $\mathcal{Z}(B)$ and $v_{1} = v$. Now, by looking at the ergodic decomposition of $M$, we can see that the projections $p_{0} = 1 - \sum_{i}p_{i}$ and $q_{0} = 1 - \sum_{i}q_{i}$ are equivalent in $M$ since $1 - p_{0}$ and $1 - q_{0}$ are equivalent in $M$. Now assuming $p_{0}$ and $q_{0}$ are not equal to zero, we can construct in exactly the same way as \cite[Lemma 12.1.2]{ADP_An_introduction} a partial isometry $w \in M$ with $w^{*}w$ and $ww^{*}$ in $\mathcal{Z}(B)$ such that $wBw^{*} \subset B$, $w^{*}Bw \subset B$, and $pw = w = wq$. This then contradicts the maximality of the family that we chose. Hence $p_{0} = q_{0} = 0$ and we get a unitary $u = \sum_{i}v_{i}$ in $\mathcal{N}_{M}(B)$ such that $uq = v$, as required. We can check that the relative commutant $B'\cap M$ is isomorphic to $L^{\infty}(\mathcal{G}^{(0)})$ if and only if $B_{x}$ is a factor a.e. and the cocycle action is free.

To see the converse, we begin with an inclusion $B \subset M$, where $B$ is regular, with a faithful and normal conditional expectation $E: M \rightarrow B$ such that $B$ satisfies the relative commutant condition. We consider the standard probability space $(X, \mu)$ such that the centre $\mathcal{Z}(B)$ is isomorphic to $L^{\infty}(X, \mu)$, and the ergodic decomposition $B = \int_{x \in X}^{\oplus}B_{x}$. Let $v \in M$ be the partial isometries such that $vBv^{*} = Bq$ and the source and range projections $v^{*}v$ and $vv^{*}$ lie in $\mathcal{Z}(B)$, and let $\mathcal{P}$ be the set of all such partial isometries. We then equip $\mathcal{P}$ with an equivalence relation, namely $v \sim w$ for two such partial isometries if and only if $v \in Bw$. Now we consider the quotient $\mathcal{I} = \mathcal{P}/ \sim$ and note that there is a natural inverse semigroup structure on $\mathcal{I}$. Using the algebraic correspondence between groupoids and inverse semigroups, we can uniquely construct a discrete measured groupoid $\mathcal{G}$ with $\mathcal{G}^{(0)} = X$ such that the full pseudogroup $[[\mathcal{G}]]$ corresponds with $\mathcal{I}$. Now, note that $\mathcal{P}$ acts on $B = \int_{x \in X}B_{x} \; d \mu$ as follows: for any partial isometry $v \in \mathcal{P}$ with $e_{v} = v^{*}v$ and $f_{v} = vv^{*}$, we define $\alpha_{v}: Be_{v} \rightarrow Bf_{v}$ by $\alpha_{v}(b) = v^{*}bv$. On passing to $\mathcal{I}$, since there is no unique lift $\mathcal{I} \rightarrow \mathcal{P}$, we get a 2-cocycle and a corresponding cocycle action of the inverse semigroup $\mathcal{I}$ on $B$. This in turn gives a cocycle action of the groupoid $\mathcal{G}$ on $(B_{x})_{x \in X}$ and by construction, we have that $B \subset M$ is isomorphic with the crossed product inclusion $B \subset (B_{x})_{x \in X} \rtimes \mathcal{G}$. There is an isomorphism between two such crossed product inclusions if and only if the associated discrete measured groupoids are isomorphic and via this isomorphism, their corresponding cocycle actions are cocycle conjugate. It can be checked that $M = B \rtimes_{(\alpha,u)} \mathcal{G}$ being a factor is equivalent to $\mathcal{G}$ being ergodic. 

\subsection{Equivalence relations and Cartan subalgebras}
\label{Equivalence relations and Cartan subalgebras}
As a special case of the discussion in the previous section, if the regular subalgebra $B$ in $M$ is in particular abelian, then the corresponding discrete measured groupoid has trivial isotropy groups almost everywhere, i.e., it is a countable equivalence relation. Such subalgebras are called Cartan subalgebras and the crossed product is isomorphic to $M$ as in the previous section and is called the von Neumann algebra of the equivalence relation. This was described first by Feldman and Moore in \cite{MR578730} and \cite{MR578656}. 

\begin{definition}
If $M$ is a von Neumann algebra and $A$ is an abelian subalgebra, we say that $A$ is a \textit{Cartan subalgebra} if $A$ satisfies the following:
\begin{itemize}
    \item $A$ is maximal abelian,
    \item $A$ is regular, i.e. $\mathcal{N}_{M}(A)$ generates $M$,
    \item there exists a faithful normal conditional expectation of $M$ onto $A$.
\end{itemize}
\end{definition}

A \textit{countable equivalence relation} is an equivalence relation $\mathcal{R} \subset X \times X$ on a standard Borel space $X$, which is a Borel subset of $X \times X$ and all equivalence classes are countable. We denote by $\mathcal{R}(x)$ for the equivalence class of $x \in X$. We also define $\mathcal{R}(A) = \cup_{x \in A}\mathcal{R}(x)$. Hence $\mathcal{R}$ is countable when $\mathcal{R}(x)$ is countable for every $x \in X$. A \textit{partial isomorphism} $\phi: A \rightarrow B$ between two Borel subsets of $X$ is a Borel isomorphism from $A$ onto $B$. We denote like in the case of groupoids, the full pseudogroup by $[[\mathcal{R}]]$ as the set of such $\phi$ whose graph is contained in $\mathcal{R}$. The domain of $\phi$ is written as $D(\phi)$ and the range is written as $R(\phi)$.  We denote by $\mathcal{R}(A) = p_{1}^{-1}(p_{2}^{-1}(A)) = p_{2}^{-1}(p_{1}^{-1}(A))$ and call this the \textit{$\mathcal{R}$-saturation of $A$}, where $p_{1}$ and $p_{2}$ are the left and right projections of $\mathcal{R}$ onto $X$. Clearly $\mathcal{R}(A)$ is Borel if and only if $A$ is Borel. Given a $\sigma$-finite measure $\mu$ on $X$, we say that the measure is \textit{quasi-invariant} for $\mathcal{R}$ if it has the property that $\mu(\mathcal{R}(A)) = 0 \iff \mu(A)=0$. In this case we say that $\mathcal{R}$ is \textit{non-singular} with respect to $\mu$.  We say that a Borel subset $A \subset X$ is \textit{invariant (or saturated)} if $\mathcal{R}(A) = A$ upto a null set. The relation $(\mathcal{R},\mu)$ is called \textit{ergodic} if every invariant Borel subset is either null or co-null. 

Suppose now that $\mu$ is a $\sigma$-finite measure on $X$ and $\mathcal{R}$ is a countable non-singular equivalence relation on $X$. We define the $\sigma$-finite measure $\nu_{l}$ on $\mathcal{R}$ by 
\begin{equation}
    \nu_{l}(C) = \int_{X} |C^{x}| \; d\mu(x) \nonumber
\end{equation}
where $C^{x}$ denotes the cardinality of the set $\{(x,y) \in C \; | \; (x,y) \in \mathcal{R} \}$. Similarly we can define the measure $\nu_{r}$ by $C \rightarrow \int_{X}|C_{x}| \; d\mu(x)$ where $C_{x} = \{(y,x) \in C \; | \; (x,y) \in \mathcal{R}\}$. It is easy to see that $\nu_{l}$ and $\nu_{r}$ are mutually absolutely continuous and we define the \textit{Radon-Nikodym 1 cocycle} of $\mathcal{R}$ with respect to $\mu$ as the function $D_{\mu}: \mathcal{R} \rightarrow \mathbb{R}^{+}_{*}$ defined by $D_{\mu}(x,y) = \frac{d\nu_{l}}{d\nu_{r}}(x,y)$. The Radon Nikodym 1 cocycle $D_{\mu}$ is 1 if and only if $\mathcal{R}$ preserves the measure $\mu$. 

We recall now the construction (c.f. \cite{MR578656}) of the von Neumann algebra of an equivalence relation with a 2-cocycle $L(\mathcal{R},c)$. It can be shown that $L(\mathcal{R},c)$ is a factor if and only if $R$ is ergodic. We note that $L^{\infty}(X,\mu) \subset L(\mathcal{R},c)$ is a Cartan subalgebra. The main result in \cite{MR578656} asserts that the converse is true: any Cartan inclusion $A \subset M$ has a unique associated nonsingular countable equivalence relation $\mathcal{R}$ on a Borel space $(X,\mu)$ and a 2-cocycle $c$ such that $A \subset M$ is isomorphic as an inclusion to $L^{\infty}(X,\mu) \subset L(\mathcal{R},c)$. The Cartan subalgebra $A$ is hyperfinite if and only if the equivalence relation $\mathcal{R}$ is hyperfinite, in which case the 2-cocycle vanishes (Theorem 6, \cite{MR578730}).  

Suppose $\mathcal{R}$ is a countable ergodic non-singular equivalence relation on $(X,\mu)$. Now suppose $\mathcal{S}$ is a countable, non-singular subequivalence relation of $\mathcal{R}$. We define the space $\Aut(\mathcal{R})$ as the space of all measure-space automorphisms $\phi \in \Aut(X,\mu)$ such that $(\phi \times \phi)(\mathcal{R}) = \mathcal{R}$. We also define the full group $[\mathcal{R}]$ or $\Int(\mathcal{R})$ of $\mathcal{R}$ as follows: $\phi \in [\mathcal{R}]$ if $\phi \in \Aut(\mathcal{R})$ and $(\phi(x),x) \in \mathcal{R} \; \mu$ a.e. We also denote $\Aut_{\mathcal{R}}(\mathcal{S}) = \Aut(\mathcal{S}) \cap [\mathcal{R}]$. Lastly, the full pseudogroup $[[\mathcal{R}]]$ of $\mathcal{R}$ is defined as the set of Borel isomorphisms $\phi: A \rightarrow B$ between subsets of $X$ such that the graph of $\phi$ is contained in $\mathcal{R}$. 

Now let $M(X, \mu)$ be the set of complex-valued measures that are absolutely continuous with respect to $\mu$. We endow $M(X, \mu)$ with a norm by putting $\|\nu\| = |\nu|(X)$ where $|\nu|$ is the total variation of $\nu$. With respect to this norm, $M(X,\mu)$ is a Banach space. We consider now the map $L^{1}(X,\mu) \rightarrow M(X, \mu)$ given by $g \mapsto \nu_{g}$ where 
\begin{equation}
    \nu_{g}(f) = \int_{X}g(x)f(x)\; d\mu(x) \nonumber
\end{equation}
This gives a Banach space isomorphism between $L^{1}(X, \mu)$ and $M(X,\mu)$. For $\alpha \in \Aut(X, \mu)$ and $g \in L^{1}(X, \mu)$, we define $\alpha_{\mu}(g) \in L^{1}(X, \mu)$ by 
\begin{equation}
    \alpha_{\mu}(g)(x) = g(\alpha^{-1}(x))\frac{d(\mu \circ \alpha^{-1})}{d\mu}(x) \nonumber
\end{equation}
It can be checked that $\alpha_{\mu}$ is an isometry of $L^{1}(X,\mu)$ and that $(\alpha\beta)_{\mu} = \alpha_{\mu}\beta_{\mu}$. Under our identification, $\alpha_{\mu}(g)$ corresponds to $\alpha(\nu_{g}) \coloneqq \nu_{g} \circ \alpha^{-1}$. In \cite{MR617740}, Hamachi and Osikawa introduced the u-topology on $\Aut(\mathcal{R})$. For $\alpha_{n}, \beta \in \Aut(\mathcal{R})$, we say that $\alpha_{n} \rightarrow \beta$ weakly if 
\begin{equation}
    \lim_{n \rightarrow \infty} \|\alpha_{n}(g) - \beta(g) \| = 0 \nonumber
\end{equation}
for all $g \in L^{1}(X, \mu)$. Now we define the metric $d_{\mu}$ on $\Aut(\mathcal{R})$ as $d_{\mu}(\alpha,\beta) \coloneqq \mu(\{x \; | \; \alpha(x) \neq \beta(x)\})$. We say that $\alpha_{n}$ converges to $\beta$ uniformly if $d_{\mu}(\alpha_{n},\beta) \rightarrow 0$. We now say that a sequence $\alpha_{n}$ converges to $\beta$ in the u-topology if $\alpha_{n} \rightarrow \beta$ weakly and $\alpha_{n}\phi\alpha_{n}^{-1} \rightarrow \beta \phi \beta^{-1}$ uniformly for all $\phi \in [\mathcal{R}]$. It can be checked that this topology indeed corresponds to the u-topology on $\Aut(L(\mathcal{R}))$, hence justifying the terminology. It is also shown in \cite{MR617740} that $\Aut(\mathcal{R})$ is a Polish group with respect to the u-topology. We now state the following definition as in \cite{MR1007409}.

\begin{definition}
A subequivalence relation $\mathcal{S} \subset \mathcal{R}$ is called \textit{strongly normal} if the graphs of the elements of $\Aut_{\mathcal{R}}(\mathcal{S})$ generate $\mathcal{R}$. 
\end{definition}

Now suppose $A \subset M$ is a Cartan inclusion isomorphic to $L^{\infty}(X,\mu) \subset L(\mathcal{R},c)$. By \cite[Theorem 1.1]{MR1978219}, for any intermediate von Neumann subalgebra $A\subset B \subset M$, there exists a unique faithful normal conditional expectation $E_{B}: M \rightarrow B$ and $A$ is a Cartan subalgebra of $B$. From this, it follows that intermediate von Neumann subalgebras $A\subset B \subset M$ are in one-to-one correspondence with subequivalence relations $\mathcal{S} \subset \mathcal{R}$ and the bijection is given by considering $B = L(\mathcal{S},u)$. We note that $B'\cap M \subset A' \cap M = A \subset B$ and hence the relative commutant condition $B'\cap M = \mathcal{Z}(B)$ is always satisfied in this situation. An action $\alpha: \Gamma \curvearrowright \mathcal{S}$ of a countable group on a non-singular equivalence relation $\mathcal{S}$ (on a probability space $(X,\mu))$ is called \textit{outer} if for all $g \neq e \in \Gamma$, we have that $(x,\alpha_{g}(x)) \notin \mathcal{S}$ for a.e. $x \in X$. We now prove the following folklore result, along the same lines as Proposition 5.1, \cite{MR4124434} which proves the tracial case.  

\begin{proposition}
\label{StrongNormality}
In the above setting, $B$ is a regular subalgebra of $M$ if and only if the corresponding subequivalence relation $\mathcal{S}$ is strongly normal in $\mathcal{R}$. 
\end{proposition}

Before proving the above proposition we need the following lemma: 

\begin{lemma}
\label{LemmaForStrongNormality}
    Let $B$ be a von Neumann algebra of type $\textrm{II}_{\infty}$ or type $\textrm{III}$ and $A_{1}$ and $A_{2}$ be Cartan subalgebras of $B$, and let $b$ be a partial isometry such that $b^{*}b \in A_{1}$, $bb^{*} \in A_{2}$ and $bA_{1}b^{*} = A_{2}bb^{*}$. If $z$ is the central support of $b^{*}b$ in $B$, then $A_{1}z$ and $A_{2}z$ are conjugate by a unitary in $Bz$.  
\end{lemma}
\begin{proof}
First, we shall prove this in the case where $B$ is a factor of type $\textrm{III}$. Let $p = b^{*}b$ and $q = bb^{*}$ be projections in $A_{1}$ and $A_{2}$ respectively. Let $\mathcal{R}_{1}$ and $\mathcal{R}_{2}$ be the countable ergodic nonsingular equivalence relations corresponding to the Cartan inclusions $A_{1} \subset B$ and $A_{2} \subset B$ on probability spaces $(X_{1}, \mu_{1})$ and $(X_{2}, \mu_{2})$ respectively. Let $U_{1}$ and $U_{2}$ denote the range sets of the projections $p$ and $q$ in $X_{1}$ and $X_{2}$ respectively. Assume first that $p$ and $q$ are both strictly smaller than 1. By ergodicity, we can choose elements of the full pseudogroups $\phi \in [[\mathcal{R}_{1}]]$ and $\psi \in [[\mathcal{R}_{2}]]$ such that $D(\phi) = U_{1}$, $R(\phi) = X_{1} \backslash U_{1}$, $D(\psi) = U_{2}$ and $R(\psi) = X_{2} \backslash U_{2}$. Let $u_{\phi}$ (with source and range projections $p$ and $1-p$ respectively) and $u_{\psi}$ (with source and range projections $q$ and $1-q$ respectively)  be the corresponding partial isometries in $B$ normalizing the Cartan subalgebras $A_{1}$ and $A_{2}$ respectively. 
    
We now consider the element $u = b + u_{\psi}bu_{\phi}^{*}$. This is clearly a unitary in $B$ since it is the sum of two partial isometries with orthogonal domains and ranges. We claim that $u$ conjugates $A_{1}$ to $A_{2}$. To prove the claim we write an element $a \in A_{1}$ as $a = ap + a(1-p)$ and notice first that $b(ap)b^{*} \in A_{2}q$. We also have that $(u_{\psi}bu_{\phi}^{*})a(1-p) (u_{\phi}b^{*}u_{\psi}^{*}) = u_{\psi}b(u_{\phi}^{*}au_{\phi})b^{*}u_{\psi}^{*}$. Note that since $u_{\phi}$ normalizes $A_{1}$, hence $u_{\phi}^{*}au_{\phi} \in A_{1}$. Now as we have that $bA_{1}b^{*} = A_{2}q$, and finally as $u_{\psi}$ normalizes $A_{2}$, we have that $u_{\psi}A_{2}qu_{\psi}^{*} \subseteq A_{2}(1-q)$. Hence  $uA_{1}u_{*} \subseteq A_{2}$ and a similar argument shows that $u^{*}A_{2}u \subseteq A_{1}$ and therefore they are equal. 

Now for the remaining cases, notice that if $b^{*}b = bb^{*} = 1$, then $b$ is the required unitary. Suppose now that $b^{*}b = 1$ and $bb^{*} < 1$. We can choose a projection $p_{1} \neq 0$ in $A_{1}$ such that $p_{1} < 1$ and consider the partial isometry $b_{1} = bp_{1}$. Now we have that $p = b_{1}^{*}b_{1} \in A_{1}$, $q = b_{1}b_{1}^{*} \in A_{2}$, $b_{1}A_{1}b_{1}^{*} = A_{2}b_{1}b_{1}^{*}$ and $p$ and $q$ are both strictly smaller than 1. By the preceding case, we then have that $A_{1}$ and $A_{2}$ are unitarily conjugate. A similar argument clearly works if $b^{*}b < 1$ and $bb^{*} = 1$.  

Now suppose $B$ is of type $\textrm{II}_{\infty}$ and $\tau(p) = \tau(q) < \infty$, where $\tau$ is the faithful normal semifinite trace. By ergodicity of the equivalence relations, we can find a sequence of elements $\phi_{n} \in [[\mathcal{R}_{1}]]$ such that $U_{1} \bigsqcup \sqcup_{n}\phi_{n}(U_{1}) = X_{1}$ where $D(\phi_{n}) = U_{1}$ and similarly a sequence of elements $\psi_{n} \in [[\mathcal{R}_{2}]]$ where $D(\psi_{n}) = U_{2}$ such that $U_{2} \bigsqcup \sqcup_{n}\psi_{n}(U_{2}) = X_{2}$. As earlier we denote $u_{\phi_{n}}$ and $u_{\psi_{n}}$ to be the corresponding partial isometries. Now we consider the element $u = b + \sum_{n} u_{\psi_{n}}bu_{\phi_{n}}^{*}$. Since this is a sum of partial isometries with orthogonal domains and ranges, $u$ is a unitary. As for each $n \in \mathbb{N}$, we know that $u_{\phi_{n}}$ normalizes $A_{1}$ and $u_{\psi_{n}}$ normalizes $A_{2}$, it is easy to see as in the type $\textrm{III}$ case that $u$ conjugates $A_{1}$ to $A_{2}$. If $p$ and $q$ are infinite: we choose a finite projection $p_{1} \in A_{1}$ and consider the partial isometry $b_{1} = bp_{1}$. Now we have that $b_{1}A_{1}b_{1}^{*} = A_{2}b_{1}b_{1}^{*}$ and the source and the range projections $b_{1}^{*}b_{1} \in A_{1}$ and $b_{1}b_{1}^{*} \in A_{2}$ have finite trace. Thus we are reduced to the previous case. 

Suppose now that $B \subseteq \mathcal{B}(\mathcal{H})$ be an arbitrary von Neumann algebra of type $\textrm{II}_{\infty}$ or $\textrm{III}$. Let the centre be $\mathcal{Z}(B) = L^{\infty}(Y, \nu)$ and let the ergodic decomposition be $B = \int^{\oplus}_{y \in Y}B_{y}$. We can write the Cartan subalgebras $A_{1}$ and $A_{2}$ as direct integrals of Cartan subalgebras $\int_{y \in Y}^{\oplus} A_{1,y}$ and $\int_{y \in Y}^{\oplus} A_{2.y}$. Let $p$ and $q$ be the source and range projections $b^{*}b$ and $bb^{*}$ respectively, $z$ be the central support of $p$ and $q$, and $E \subset Y$ be the range set of $z$. Given $y \in E$, by the factor case, we can find a unitary $u_{y} \in B_{y}$ such that $u_{y}A_{1,y}u_{y}^{*} = A_{2,y}$. Letting $u = \int_{y \in E}^{\oplus} u_{y}$ be the corresponding unitary in $Bz$, clearly $u$ conjugates $A_{1}z$ to $A_{2}z$.           
\end{proof}

\begin{proof}[Proof of Proposition \ref{StrongNormality}]
If $B$ is of type $\textrm{II}_{1}$, the result follows from \cite[Proposition 5.1]{MR4124434}. Now suppose $B \subset M$ be regular. We notice that there is a natural map from $\mathcal{N}_{M}(B) \cap \mathcal{N}_{M}(A)$ to $\Aut_{\mathcal{R}}(\mathcal{S})$. Thus it is enough to prove that for any $u \in \mathcal{N}_{M}(B)$, there is a unitary element $b \in \mathcal{U}(B)$ such that $bu \in \mathcal{N}_{M}(A)$. 

Now as in \cite{MR4124434}, for any $u \in \mathcal{N}_{M}(B)$ and a fixed $v \in \mathcal{N}_{M}(A)$, we look at the automorphisms $\Aut(B) \ni \beta : x \mapsto uxu^{*}$ and $\Aut(A) \ni \alpha: x \mapsto vxv^{*}$. We define $e_{v} = E_{B}(vu^{*})$ and note that $e_{v} \beta(a) = \alpha(a)e_{v}$ for all $a \in A$. Letting $b_{v}$ denote the partial isometry in the polar decomposition of $e_{v}$, we note that $b_{v}\beta(a) = \alpha(a)b_{v}$ for all $a \in A$. We denote by $z_{v} \in \mathcal{Z}(B)$, the central support of $b_{v}^{*}b_{v} \in B$. For each projection $a \in A$, since we also have $\beta(a)b_{v}^{*} = b_{v}^{*}\alpha(a)$, we have for all projections $a \in A$, 
\begin{equation}
    \alpha(a)b_{v}b_{v}^{*} = b_{v}\beta(a)b_{v}^{*} = b_{v}b_{v}^{*}\alpha(a)
\nonumber
\end{equation}
Since $A \subset M$ is a Cartan subalgebra, we have that $A' \cap M = A$ and hence that $b_{v}b_{v}^{*} \in A$. Thus we can find a projection $q \in A$ such that $\alpha(q) = b_{v}b_{v}^{*}$. By the same argument, as $\beta(A) \subset M$ is a Cartan subalgebra, we can find a projection $p \in A$ such that $\beta(p) = b_{v}^{*}b_{v}$. We note now that the Cartan subalgebras $\beta(A)z_{v}$ and $Az_{v}$ satisfy all the hypotheses of Lemma \ref{LemmaForStrongNormality} and therefore $\beta(A)z_{v}$ and $Az_{v}$ are unitarily conjugate as Cartan subalgebras of $Bz_{v}$. Since $A$ is regular in $M$, clearly the join of all the central projections $z_{v}$ as $v$ varies over $\mathcal{N}_{M}(A)$ is 1. Thus we can conclude that the join of the unitaries in $Bz_{v}$ we get from Lemma \ref{LemmaForStrongNormality} is a unitary $b \in \mathcal{U}(B)$ that conjugates $\beta(A)$ to $A$ and thus $bu \in \mathcal{N}_{M}(A)$.    

On the other hand, if $\mathcal{S} \subset \mathcal{R}$ is strongly normal, $\Aut_{\mathcal{R}}(\mathcal{S})$ generates $\mathcal{R}$. Clearly if $\phi \in \Aut_{\mathcal{R}}(\mathcal{S})$, the corresponding unitaries $u_{\phi} $ belong to $ \mathcal{N}_{M}(B)$, and hence $B \subset M$ is regular. 
\end{proof}

If the conclusion of Proposition \ref{StrongNormality} is satisfied, we can write $A \subset B$ as the direct integral of Cartan inclusions $(A_{y} \subset B_{y})_{y \in Y}$ where t $\mathcal{Z}(B) = L^{\infty}(Y,\mu)$, $B_{y}$ is a factor a.e. and $(\mathcal{S}_{y})_{y \in Y}$ is the associated field of ergodic equivalence relations. Then $M$ can be realized as the crossed product of a free cocycle action $(\alpha,u)$ of a discrete measured groupoid $\mathcal{G}$ on the measurable field $(A_{y} \subset B_{y})_{y \in Y}$. More precisely, $\mathcal{G}^{(0)} = Y$ and we have the following: 
\begin{itemize}
    \item a measurable field of Cartan preserving *-isomorphisms between the sources and the ranges, i.e.,  $g \mapsto \alpha_{g}:B_{s(g)} \rightarrow B_{t(g)}$ such that $\alpha_{g}(A_{s(g)}) = A_{t(g)}$.
    \item a measurable field of unitaries normalizing the Cartan subalgebras, i.e., $\mathcal{G}^{(2)} \ni (g,h) \mapsto u(g,h) \in \mathcal{N}_{B_{t(g)}}(A_{t(g)})$
\end{itemize}
that satisfy the conditions in Definition 2.2.3. Alternately we can view $\mathcal{R}$ as a semidirect product of the subequivalence relation $\mathcal{S} = (\mathcal{S}_{y})_{y \in Y}$ with an outer cocycle action of $\mathcal{G}$ on $\mathcal{S}$. We conclude this section with the definition of cocycle conjugacy for actions on Cartan inclusions. 

\begin{definition}
\label{DefCocycleConjugacyCartan}
\begin{enumerate}[1.]
    \item Two cocycle actions $(\alpha,u)$ and $(\beta,v)$ of a countable group $G$ on  Cartan inclusions $A \subset B$ and $C \subset D$ as above are said to be cocycle conjugate if there exists a *-isomorphism $\theta: B \rightarrow D$ with $\theta(A) = C$ and unitaries $w_{g} \in \mathcal{N}_{D}(C)$ such that
    \begin{itemize}
        \item $\theta \circ \alpha_{g} \circ \theta^{-1}= \Ad(w_{g})\circ \beta_{g}$ for all $g \in G$.
        \item $\theta(u(g,h)) = w_{g}\beta_{g}(w_{h})v(g,h)w_{gh}^{*}$ for all $g,h \in G$.
    \end{itemize}
    
    \item Two cocycle actions $(\alpha,u)$ and $(\beta,v)$ of a discrete measured groupoid $\mathcal{G}$ on fields of Cartan inclusions $(A_{x } \subset B_{x})_{x \in X}$ and $(C_{x} \subset D_{x})_{x \in X}$ respectively are said to be \textit{cocycle conjugate} if there exists a measurable field of $*$-isomorphisms $x \mapsto \theta_{x}: B_{x} \rightarrow D_{x}$ with $\theta_{x}(A_{x}) = C_{x}$ and a measurable field of unitaries $g \mapsto w_{g} \in \mathcal{N}_{D_{t(g)}}(C_{t(g)})$ that satisfy:
 \begin{itemize}
     \item  $\theta_{t(g)} \circ \alpha_{g} \circ \theta^{-1}_{s(g)} = \Ad(w_{g}) \circ \beta_{g}$ for all $g \in \mathcal{G}$.
     \item  $\theta_{t(g)}(u(g,h)) = w_{g}\beta_{g}(w_{h})v(g,h)w^{*}_{gh}$ for all $(g,h) \in \mathcal{G}^{(2)}$.
 \end{itemize}
\end{enumerate}

\end{definition}

\subsection{Central freeness and Connes-Takesaki module map}

Recall the continuous core $(\widetilde{M},\mathbb{R},\theta,\tau)$ associated to a factor $M$ of type $\textrm{III}$ with respect to a faithful normal state $\phi$. We develop some notation here. The group $\widetilde{\mathcal{U}}(M) = \{u \in \mathcal{U}(\widetilde{M}) \; | \; uMu^{*}=M\}$, the normalizer of $M$ in $\widetilde{M}$ is called the \textit{extended unitary group} of $M$. For each $u \in \widetilde{\mathcal{U}}(M)$, we write 
\begin{equation}
    \widetilde{Ad}(u) = \Ad(u)|_{M} \in \Aut(M) \nonumber
\end{equation}
and set $\Cnt_{r}(M) = \{\widetilde{Ad}(u) \; | \; u \in \Tilde{\mathcal{U}}(M)\}$ which is a normal subgroup of $\Aut(M)$. Each $\alpha \in \Aut(M)$ can be extended to an automorphism $\Tilde{\alpha} \in \Aut(\widetilde{M})$ given by: \begin{equation}
    \Tilde{\alpha}(xu_{t}) = \alpha(x)(D\phi \circ \alpha^{-1}: D \phi)_{t}u_{t} \; \nonumber
\end{equation}
where $(D\phi \circ \alpha^{-1}: D \phi)_{t}$ is the Connes cocycle derivative, the trace $\tau$ satisfies $\tau \circ \Tilde{\alpha} = \tau$ and for every $s \in \mathbb{R}$, $\Tilde{\alpha} \circ \theta_{s} = \theta_{s} \circ \Tilde{\alpha}$. The restriction of $\Tilde{\alpha}$ to the center $\mathcal{C}_{M}$ is called the \textit{Connes-Takesaki module map} of $\alpha$ and is denoted by $ \modnew(\alpha) \in \Aut(\mathcal{C_{M}})$ which commutes with the smooth flow of weights $\theta|_{\mathcal{C_{M}}}$. Similarly for an action $\alpha: G \curvearrowright M$, we denote the action on the flow of weights by $ \modnew(\alpha): G \curvearrowright \mathcal{C}_{M}$.  

The automorphism group $\Aut(M)$ is identified with the following subgroup $\Aut_{\tau,\theta}(M)$ of $\Aut(\widetilde{M})$: 
\begin{equation}
    \Aut_{\tau,\theta}(M) = \{\sigma \in \Aut(\widetilde{M}) \; | \; \sigma \circ \theta_{s} = \theta_{s} \circ \sigma, \; \tau\circ\sigma = \tau\} \nonumber
\end{equation}
In other words, for each $\alpha \in \Aut(M)$, $\Tilde{\alpha} \in \Aut_{\tau,\theta}(M)$ and if an element $\sigma \in \Aut_{\tau,\theta}(M)$ leaves each element of $M$ invariant, then $\sigma = \id$. 

For an action $\alpha: G \rightarrow \Aut(M)$, if $w: G \rightarrow \mathcal{U}(M)$ satisfies $w(g)\alpha_{g}(w(h)) = w(gh)$, then we say $w$ is a \textit{1-cocycle} for $\alpha$. We denote by $Z^{1}_{\alpha}(G,\mathcal{U}(M))$ the set of all 1-cocycles for $\alpha$. Recall that two cocycle actions $(\alpha,u)$ and $(\beta,v)$ are said to be \textit{cocycle conjugate} if there exists $w:G \rightarrow \mathcal{U}(M)$ and $\theta \in \Aut(M)$ such that 
\begin{equation}
    \Ad(w_{g})\circ \alpha_{g} = \theta \circ \beta_{g} \circ \theta^{-1} \text{  and  } w(g)\alpha_{g}(w(h))u(g,h)w(gh)^{*} = \theta(v(g,h)), \nonumber
\end{equation}

\begin{definition}
\label{DefStrongCocycleConjugacy}
\begin{enumerate}[1.]
    \item Let $G$ be a countable group and let $(\alpha,u)$ and $(\beta,v)$ be two cocycle actions of $G$ on injective factors $B$ and $D$ respectively. We say that $(\alpha,u)$ and $(\beta,v)$ are \textit{strongly cocycle conjugate} if they are cocycle conjugate as in Definition \ref{DefCocycleConjugacy} and $\theta: B \rightarrow D$ can be chosen such that the induced isomorphism on the flow of weights is trivial, i.e.,  $ \modnew(\theta) = 1$.
    
    \item Let $(X, \mu)$ be a standard probability space and $\mathcal{G}$ be a discrete measured groupoid with unit space $\mathcal{G}^{0} = X$. Let $(\alpha,u)$ and $(\beta,v)$ be two cocycle actions of $\mathcal{G}$ on fields of injective  factors $(B_{x})_{x \in X}$ and $(D_{x})_{x \in X}$. We say that $\alpha$ and $\beta$ are \textit{strongly cocycle conjugate} if they are cocycle conjugate as in Definition \ref{DefCocycleConjugacy} and moreover for almost every $x \in X$, $\theta_{x}: B_{x} \rightarrow D_{x}$ can be chosen such that the induced isomorphims on the flow of weights are trivial, i.e.,  $ \modnew(\theta_{x}) = 1$. 

\end{enumerate}
\end{definition}

For an injective factor $M$ of type $\textrm{III}$, the Connes-Takesaki module map is surjective and the following exact sequence splits: 
\begin{equation}
    1\rightarrow \ker(\modnew) \rightarrow \Aut(M) \xrightarrow[]{\modnew} \Aut_{\theta}(\mathcal{C}) \rightarrow 1. \nonumber
\end{equation}
An important object associated with an action $\alpha: G \rightarrow \Aut(M)$ is the normal subgroup of $G$ defined by: 
\begin{equation}
    N_{\alpha} = \{g \in G \; | \; \alpha_{g} \in \Cnt_{r}(M)\} \nonumber
\end{equation}
We note that $ \modnew(\alpha_{n}) = \id$ for $n \in N_{\alpha}$ and hence $ \modnew(\alpha)$ can be regarded as an action of $G/ N$ on $\mathcal{C}_{M}$. The classification of discrete amenable group actions on injective factors up to cocycle conjugacy has been developed over the years by many hands (c.f. \cite{MR587749}, \cite{MR807949}, \cite{MR842413}, \cite{MR990219}, \cite{MR1179014}, \cite{MR1621416}). In particular, the classification theorem states that any free action $\alpha$ of a discrete amenable group on an injective factor of type $\textrm{III}$ is classified by three invariants: the normal subgroup $N_{\alpha}$, the action on the flow of weights $ \modnew(\alpha)$ and a characteristic invariant $\chi(\alpha)$. 

Moreover if the subgroup $N_{\alpha}$ is trivial, then the characteristic invariant $\chi(\alpha)$ vanishes, and such actions are classified just by the actions on the flow of weights. We shall focus precisely on this picture. We begin with the notion of central freeness introduced by Ocneanu in \cite{MR807949} for group actions on von Neumann algebras.

\begin{definition}
For a factor $M$, an automorphism $\theta \in \Aut(M)$ is called \textit{centrally trivial} if for any centralizing sequence $x_{n} \in M$, i.e., which is norm bounded and satisfies $\lim_{n \rightarrow \infty} \|[\phi,x_{n}]\| = 0$ for any $\phi \in M_{*}$, we have $\theta(x_{n})-x_{n} \rightarrow 0$ $*$-strongly. A discrete group action $\alpha: G \rightarrow \Aut(M)$ is called \textit{centrally free} if for any nontrivial $g \in G$, $\alpha_{g}$ is not centrally trivial.
\end{definition}

By \cite[Theorem 1]{MR1179014}, in case of an injective factor $M$ of type $\textrm{III}$, an automorphism $\theta \in \Aut(M)$ is centrally trivial if and only if $\theta = \Ad(u)\circ \overline{\sigma}^{\phi}_{c}$ where $\overline{\sigma}^{\phi}_{c}$ is an extended modular automorphism for a dominant weight $\phi$ on $M$, $c$ is a $\theta$ cocycle on $\mathcal{U}(\mathcal{C}_{M})$ and $u \in \mathcal{U}(M)$. By Proposition 5.4, \cite{MR1064693}, this is equivalent to $\Tilde{\theta}$ being inner on $\widetilde{M}$ for an infinite factor. Hence we conclude that centrally free actions are precisely the ones that are outer as actions on the continuous core. We shall now introduce the invariants for actions on ergodic flows. 

\begin{definition}
    \begin{enumerate}[1.]
        \item Let $(F_{s})_{s \in \mathbb{R}}$ and $(F'_{s})_{s \in \mathbb{R}}$ be ergodic flows on $(X,\mu)$ and $(X',\mu')$ respectively. Let $\delta$ and $\delta'$ be a two actions of a countable group $G$ on $F_{s}$ and $F'_{s}$ respectively. We say that $\delta$ and $\delta'$ are conjugate and denote it by $\delta \sim \delta'$ if there exists a nonsingular $\mathbb{R}$-equivariant isomorphism $\sigma: X \rightarrow X'$ such that $\sigma \circ \delta_{g} \circ \sigma^{-1} = \delta'_{g}$ for all $g \in G$.  
        \item Let $(Y, \nu)$ be a standard probability space and let $F_{y}$ on $(X_{y}, \mu_{y})_{y \in Y}$ and $F'_{y}$ on $(X'_{y}, \mu'_{y})_{y \in Y}$ be two measurable fields of ergodic flows. Let $\mathcal{G}$ be a discrete measured groupoid with $\mathcal{G}^{(0)} = X$ and let $\delta$ and $\delta'$ be two actions of $\mathcal{G}$ on the fields of flows $(F_{y})_{y \in Y}$ and $(F'_{y})_{y \in Y}$ respectively. We say that $\delta$ and $\delta'$ are conjugate and denote by $\delta \sim \delta'$ if there exists a measurable field of nonsingular $\mathbb{R}$-equivariant isomorphisms $\sigma_{y}: X_{y} \rightarrow X'_{y}$ such that for a.e. $g \in \mathcal{G}$, we have that $\sigma_{t(g)} \circ \delta_{g} \circ \sigma_{s(g)}^{-1} = \delta'_{g}$.   
    \end{enumerate}
\end{definition}

From the previous discussion, the classification theorem for centrally free actions can be stated as follows: 

\begin{theorem}(\cite[Theorem 2.3]{MR3181546})
\label{StrongCocycleConjugacy}
Let $M$ be an injective factor and $(\alpha,u)$ and $(\beta,v)$ be two centrally free cocycle actions of a discrete amenable group $G$ on $M$. Then we have the following: 
\begin{enumerate}[(i)]
    \item The cocycle actions $(\alpha,u)$ and $(b,v)$ are cocycle conjugate if and only if $ \modnew(\alpha) \sim  \modnew(\beta)$.
    \item The cocycle actions $(\alpha,u)$ and $(b,v)$ are strongly cocycle conjugate if and only if $ \modnew(\alpha) =  \modnew(\beta)$.
\end{enumerate}
\end{theorem}

We conclude this section with the definition of centrally free actions for a discrete measured groupoid on a field of factors. 

\begin{definition}
Let $\mathcal{G}$ be a discrete measured groupoid with $\mathcal{G}^{(0)} = X$. A cocycle action $(\alpha,u)$ of $\mathcal{G}$ on a measurable field of factors $(B_{x})_{x\in X}$ is called \textit{centrally free} if for a.e. $x \in X$ and all $g \in \Gamma_{x}$, the automorphism $\alpha_{g}$ is centrally non-trivial. 
\end{definition}

\section{Actions on fields of injective factors}
\label{sectionnoncartan}

In the case of $\textrm{II}_{1}$ factors, it has been shown that any cocycle action of an amenable group can be perturbed to a genuine action in \cite{MR4258165}. In \cite[Theorem 3.1]{MR4124434}, it has been shown that any cocycle action of an amenable discrete measured groupoid on a measurable field of $\textrm{II}_{1}$ factors can be perturbed to a genuine action. We begin this section with a 2-cohomology vanishing result for groupoid actions on fields of infinite factors. For the proof, we use a Packer-Raeburn type stabilization trick as in \cite{MR1002543}. We remark here that such a stabilization trick had already appeared in the literature before \cite{MR1002543}, for example, in \cite[Theorem 4.2.2]{MR0574031}.

\begin{theorem}
\label{2CohomologyVanishingGroupoidsAlternate}
Let $(X, \mu)$ be a standard probability space and $(B_{x})_{x \in X}$ be a field of infinite factors. Let $(\alpha,u)$ be a cocycle action of a discrete measured groupoid $\mathcal{G}$ with unit space $X$ on the field $(B_{x})_{x \in X}$. Then $u$ is a coboundary, i.e., there exists $g \mapsto w_{g} \in \mathcal{U}(B_{t(g)})$, a measurable field of unitaries such that for all $(g,h) \in \mathcal{G}^{(2)}$ we have $u(g,h) = \alpha_{g}(w_{h}^{*})w_{g}^{*}w_{gh}$.
\end{theorem}

\begin{proof}
Since for all $x \in X$, $B_{x}$ is separable, we assume without loss of generality that they are represented on the same separable Hilbert space $\mathcal{H}$. Let $\mathcal{G}_{x}$ denote the countable set of groupoid elements $\{g \in \mathcal{G} \; | \; t(g) = x\}$. Since for a.e. $x \in X$, $B_{x}$ is infinite, let $(\phi_{x})_{x \in X}$ denote a family of *-isomorphisms between $B_{x}$ and $B_{x} \otimes \mathcal{B}(l^{2}(\mathcal{G}_{x}))$. Thus for a.e. $g \in \mathcal{G}$, we have a *-isomorphism  
\begin{equation}
    \phi_{t(g)} \circ \alpha_{g} \circ \phi_{s(g)}^{-1}: B_{s(g)} \otimes \mathcal{B}(l^{2}(\mathcal{G}_{s(g)})) \rightarrow B_{t(g)} \otimes \mathcal{B}(l^{2}(\mathcal{G}_{t(g)})) \nonumber
\end{equation}  

For convenience of notation, we denote $D_{x} = B_{x} \otimes \mathcal{B}(l^{2}(\mathcal{G}_{x}))$, $\alpha'_{g} = \phi_{t(g)}\circ \alpha_{g} \circ \phi_{s(g)}^{-1}$. Now in each $\mathcal{B}(l^{2}(\mathcal{G}_{x}))$ we denote the matrix units by $e^{x}_{g,h}$ for $g,h \in \mathcal{G}_{x}$. For each $g \in \mathcal{G}$, $B_{t(g)}$ is infinite and hence the projections $1 \otimes e^{t(g)}_{g,g}$ and $\alpha'_{g}(1 \otimes e^{s(g)}_{1,1})$ are equivalent, where $e^{x}_{1,1}$ denotes the diagonal matrix unit corresponding to the identity element $1 \in \mathcal{G}_{x}$. We now pick partial isometries $w_{g} \in D_{t(g)}$ such that $w_{g}^{*}w_{g} = 1 \otimes e^{t(g)}_{g,g}$ and $w_{g}w_{g}^{*} = \alpha'_{g}(1 \otimes e^{s(g)}_{1,1})$. Now for $g \in \mathcal{G}$, we consider the element \begin{equation}
    D_{t(g)} \ni V_{g} \coloneqq \sum_{k \in \mathcal{G}_{s(g)}} \alpha'_{g}(1 \otimes e^{s(g)}_{k,1})w_{g}(1 \otimes e^{t(g)}_{g,gk}) \nonumber
\end{equation}
A simple computation shows that $\alpha_{g}'(1 \otimes e^{s(g)}_{k,1})w_{g}(1 \otimes e^{t(g)}_{g,gk})$ is a partial isometry with source projection $1 \otimes e^{t(g)}_{gk,gk}$ and range projection $\alpha'_{g}(1 \otimes e^{s(g)}_{k,k})$. Hence we have that $V_{g}$ is the sum of partial isometries with orthogonal sources and ranges, and hence $V_{g} \in \mathcal{U}(D_{t(g)})$. Now for $k,l \in \mathcal{G}_{s(g)}$, we can check that $\alpha'_{g}(1 \otimes e^{s(g)}_{k,l})= \Ad(V_{g})(1 \otimes e^{t(g)}_{gk,gl})$.

Let us denote by $(\lambda_{g})_{g \in \mathcal{G}}$ the field of unitary operators $\lambda_{g}: l^{2}(\mathcal{G}_{s(g)}) \rightarrow l^{2}(\mathcal{G}_{t(g)})$ given by $\lambda_{g}(\delta_{h}) = \delta_{gh}$. Since we have that $\Ad(V_{g}^{*})\circ \alpha'_{g}(1 \otimes e^{s(g)}_{k,l}) = 1 \otimes e^{t(g)}_{gk,gl}$ for all $g \in \mathcal{G}$ and for all $k, l \in \mathcal{G}_{s(g)}$,
we can conclude that the action $\Ad(V_{g}^{*}) \circ \alpha'_{g}$ is of the form $\beta_{g} \otimes Ad (\lambda_{g})$ where $\beta_{g}: B_{s(g)} \rightarrow B_{t(g)}$ is an action of $\mathcal{G}$ on the field $(B_{x})_{x \in X}$. We have thus constructed unitaries $V_{g} \in \mathcal{U}(D_{t(g)})$ such that 
\begin{equation}
    \alpha'_{g} = \Ad(V_{g})\circ (\beta_{g} \otimes \Ad(\lambda_{g})) \nonumber
\end{equation}
Let $\sigma_{g}\coloneqq \beta_{g} \otimes \Ad(\lambda_{g})$ and define the 2-cocycle 
\begin{equation}
    u'(g,h) = V_{g}^{*}\alpha'_{g}(V_{h}^{*})\phi_{t(g)}(u(g,h)))V_{gh} \nonumber
\end{equation}
Hence $(\sigma,u')$ defines a cocycle action of $\mathcal{G}$ on the field of factors $(D_{x})_{x \in X}$. Since we have that $\lambda_{g} \circ \lambda_{h} = \lambda_{gh}$, the unitaries $u'(g,h)$ are of the form $u_{1}(g,h) \otimes 1$ with $u_{1}(g,h) \in \mathcal{U}(B_{t(g)})$. For $g \in \mathcal{G}$ we now define the unitaries: 
\begin{equation}
    \mathcal{U}(D_{t(g)}) \ni W_{g} = \sum_{r \in \mathcal{G}_{t(g)}}u_{1}(g,g^{-1}r) \otimes e^{t(g)}_{r,r} \nonumber
\end{equation} 
Now we shall denote
$C_{g,h} = u'(g,h)^{*}\sigma_{g}(W_{h})W_{g}$ and check that $C_{g,h} = W_{gh}$.  

\begin{equation}
    \begin{split}
        C_{g,h} &= ( u_{1}(g,h)^{*} \otimes 1)(\sum_{r \in \mathcal{G}_{(t(h))}}\beta_{g}(u_{1}(h,h^{-1}r)) \otimes \Ad(\lambda_{g})(e^{t(h)}_{r,r}))(\sum_{s \in \mathcal{G}_{t(g)}}u_{1}(g,g^{-1}s) \otimes e^{t(g)}_{s,s}) \\
        &= ( u_{1}(g,h)^{*} \otimes 1)(\sum_{r \in \mathcal{G}_{(t(h))}}\beta_{g}(u_{1}(h,h^{-1}r)) \otimes e^{t(g)}_{gr,gr})(\sum_{s \in \mathcal{G}_{t(g)}}u_{1}(g,g^{-1}s) \otimes e^{t(g)}_{s,s}) \\
        &= \sum_{r \in \mathcal{G}_{(t(h))}}u_{1}(g,h)^{*}\beta_{g}(u_{1}(h,h^{-1}r))u_{1}(g,r) \otimes e^{t(g)}_{gr,gr} \\
        &= \sum_{r \in \mathcal{G}_{(t(h))}} u_{1}(gh,h^{-1}r) \otimes e^{t(g)}_{gr,gr} = \sum_{s \in \mathcal{G}_{t(g)}}u_{1}(gh,h^{-1}g^{-1}s) \otimes e^{t(g)}_{s,s} = W_{gh} 
    \end{split} \nonumber
\end{equation}

Hence $W_{g}$ is a 1 cocycle that perturbs $\sigma$ to a genuine action. Now we consider the unitaries $Z_{g} = \phi_{t(g)}^{-1}(V_{g}W_{g})$ which satisfies $\alpha_{g}(Z_{h})Z_{g} = u(g,h)Z_{gh}$ as required.
\end{proof}

For a direct integral $B = \int^{\oplus}_{x\in X} B_{x}$ of factors, we denote by $\widetilde{B}$ the direct integral of the canonical cores (not necessarily factors anymore) $\int^{\oplus}_{x \in X}\widetilde{B}_{x}$. The following lemma establishes a condition on a regular inclusion, that is equivalent to the central freeness of the corresponding groupoid action.  

\begin{lemma}
\label{CentrallyFreeCondition}
Let $M$ be an injective factor of type $\textrm{III}$ and let $B \subset M$ be a regular von Neumann subalgebra with a faithful normal conditional expectation $E: M \rightarrow B$ satisfying $\mathcal{Z}(B) = B' \cap M$. The resulting cocycle action $(\alpha,u)$ of a discrete measured groupoid $\mathcal{G}$ on $B$ is centrally free if and only if $\mathcal{Z}(\widetilde{B}) = (\widetilde{B})'\cap \widetilde{M}$.
\end{lemma}
\begin{proof}
Claim 1: $\widetilde{M}$ is canonically isomorphic to $\widetilde{B} \rtimes_{(\tilde{\alpha},u)} \mathcal{G}$. To prove this, let us assume that $B \subset M \subseteq \mathcal{B}(\mathcal{H})$ and let $\phi$ be a faithful normal state on $B$. We define now a faithful normal state $\psi$ on $M$ by $\psi = \phi \circ E$. Suppose that $\mathcal{Z}(B) = L^{\infty}(X, \mu)$ and let $B = \int_{x \in X}^{\oplus}B_{x}$ be the ergodic decomposition into factors. We can choose separable Hilbert spaces $\mathcal{H}_{x}$ such that $B_{x} \subseteq \mathcal{B}(\mathcal{H}_{x})$ and faithful normal states $\phi_{x}$ on $B_{x}$ for all $x \in X$ such that: 
\begin{itemize}
    \item The two faithful normal states $\phi$ and $\int_{x \in X} \phi_{x}$ are equal.
    \item The two Hilbert spaces $\mathcal{H}$ and $\int^{\oplus}_{x \in X} \mathcal{H}_{x}$ are isomorphic.
\end{itemize}

Next, we define the Hilbert spaces $\mathcal{K}_{x} = \mathcal{H}_{x} \otimes L^{2}(\mathbb{R})$, $\mathcal{K} = \mathcal{H} \otimes L^{2}(\mathbb{R})$ and $\mathcal{K}' = \int^{\oplus}_{x} \mathcal{K}_{x}$. Consider the natural Hilbert space isomorphism $u: \mathcal{K} \rightarrow \mathcal{K}'$ that sends an element of the form $(\xi_{x})_{x \in X} \otimes f \in (\int^{\oplus}_{x}\mathcal{H}_{x}) \otimes L^{2}(\mathbb{R})$ to $(\xi_{x} \otimes f)_{x \in X} \in \int^{\oplus}_{x}(\mathcal{H}_{x} \otimes L^{2}(\mathbb{R}))$. Clearly we have that $uu^{*} = u^{*}u = 1$. We represent $\widetilde{M} \subset \mathcal{B}(\mathcal{K})$ as the SOT-closure of the subalgebra generated by the elements of $B$ of the form $\{(b_{x})_{x \in X} \otimes 1\ \; | \; b_{x} \in B_{x}\}$, the partial isometries $\{u(\mathcal{U}) \otimes 1 \; | \; \mathcal{U} \in [[\mathcal{G}]]\}$, and the canonical unitaries $\{u^{\phi}_{t} \; |\; t \in \mathbb{R}\}$ that satisfy 

\begin{equation}
 u^{\phi}_{t}(bu(\mathcal{U}))u^{\phi*}_{t} = \sigma^{\phi}_{t}(bu(\mathcal{U})) \nonumber 
\end{equation}

On the other hand, $\widetilde{B} \rtimes_{(\tilde{\alpha},u)} \mathcal{G} \subset B(\mathcal{K}')$ can be represented as the SOT-closure of the subalgebra generated by the elements corresponding to the subalgebra $B$: $\{(b_{x} \otimes 1)_{x \in X} \; | \; b_{x} \in B_{x}\}$, the partial isometries $\{\tilde{u}(\mathcal{U}) \; | \; \mathcal{U} \in [[\mathcal{G}]]\}$ for the action $\tilde{\alpha}$ of $\mathcal{G}$ on $\int_{x \in X}^{\oplus}\tilde{B_{x}}$, and the unitaries $\{(u^{\phi_{x}}_{t})_{x \in X} \; | \; t \in \mathbb{R}\}$ which satisfy for $\mathcal{U} \in [[\mathcal{G}]]$: 
\begin{equation}
    \tilde{u}(\mathcal{U})(b_{x}u^{\phi_{x}}_{t})_{x \in s(\mathcal{U})}\tilde{u}(\mathcal{U})^{*} = \tilde{\alpha}_{\mathcal{U}}((b_{x}u^{\phi_{x}}_{t})_{x \in s(\mathcal{U})}) \nonumber
\end{equation} 

We define a map $F: \mathcal{B}(\mathcal{K}) \rightarrow \mathcal{B}(\mathcal{K}')$ by $F(x) = uxu^{*}$. It is easy to check that the map $F$ satisfies the following: 
\begin{itemize}
    \item $F((b_{x})_{x\in X} \otimes 1) = (b_{x} \otimes 1)_{x \in X}$ for $b_{x} \in \mathcal{B}_{x}$,
    \item $F(u(\mathcal{U}) \otimes 1) = \tilde{u}(\mathcal{U})$ for $u \in [[\mathcal{G}]]$,
    \item $F(u^{\phi}_{t}) = (u^{\phi_{x}}_{t})_{x \in X}$ for all $t \in \mathbb{R}$
\end{itemize}

Since $F$ is given by a unitary conjugation and sends generators to generators, $F$ gives an isomorphism between the von Neumann algebras $\widetilde{M}$ and $\widetilde{B} \rtimes_{(\tilde{\alpha},u)} \mathcal{G}$ thus proving our claim. 

Claim 2: Given an automorphism $\alpha \in \Aut(B_{x})$, the automorphism $\Tilde{\alpha} \in \Aut(\widetilde{B_{x}})$ is properly outer if and only if $\Tilde{\alpha}$ is outer. Since proper outerness clearly implies outerness, we prove here the other direction. Let us assume first that $\Tilde{\alpha}$ is not properly outer, i.e., there exists $v \in \widetilde{B}_{x}$ such that $\Tilde{\alpha}(b)v = vb$ for all $b \in \widetilde{B}_{x}$. Let $v = w|v|$ be the polar decomposition of $v$ such that $w$ is a partial isometry. We have that the projections $w^{*}w$ and $ww^{*}$ are central, and we denote $p = w^{*}w = ww^{*} \in \mathcal{Z}(\widetilde{B}_{x})$. Assuming that $w \neq 0$, we will show that $\Tilde{\alpha}$ is inner. First we note that $\tilde{\alpha}(b)w = wb$ for all $b \in \widetilde{B}_{x}$. Let $(\theta_{t})_{t \in \mathbb
R}$ be the trace scaling action of $\mathbb{R}$ on $\widetilde{B}_{x}$ which commutes with $\Tilde{\alpha}$ and acts ergodically on the center $\mathcal{Z}(\widetilde{B}_{x})$. Since the action is ergodic, we can find a sequence of real numbers $t_{n} \in \mathbb{R}$ and a sequence of projections $p_{n} \in \mathcal{Z}(\widetilde{B}_{x})$ such that $p_{n}  \leq p$ for all $n$ and
\begin{equation}
    \sum_{n=0}^{\infty} \theta_{t_{n}}(p_{n}) = 1.  \nonumber
\end{equation}
As $\theta$ commutes with $\Tilde{\alpha}$, we have that $\Tilde{\alpha}(b)\theta_{t}(w) = \theta_{t}(w)b$ for all $t \in \mathbb{R}$ and $b \in \widetilde{B}_{x}$. We consider now the element $W = \sum_{n=0}^{\infty}\theta_{t_{n}}(wp_{n})$and check that $W^{*}W = \sum_{n=0}^{\infty}\theta_{t_{n}}(p_{n}w^{*}wp_{n}) = \sum_{n=0}^{\infty}\theta_{t_{n}}(p_{n})  = 1$ and $WW^{*} = \sum_{n=0}^{\infty}\theta_{t_{n}}(wp_{n}w^{*}) = \sum_{n=0}^{\infty}\theta_{t_{n}}(p_{n}) = 1$. Now, for an arbitrary $b \in \widetilde{B}_{x}$, letting $b_{n} = b\theta_{t_{n}}(p_{n})$, we check that:
\begin{equation}
    \begin{split}
        Wb &= \sum_{n=0}^{\infty}\theta_{t_{n}}(w)\theta_{t_{n}}(p_{n})b = \sum_{n=0}^{\infty}\tilde{\alpha}(\theta_{t_{n}}(p_{n}))\tilde{\alpha}(b)\theta_{t_{n}}(w) \\
        &= \sum_{n=0}^{\infty} \tilde{\alpha}(b)\tilde{\alpha}(\theta_{t_{n}}(p_{n}))
        \theta_{t_{n}}(w) = \sum_{n=0}^{\infty} \tilde{\alpha}(b)\theta_{t_{n}}(\tilde{\alpha}(p_{n}))\theta_{t_{n}}(w) \\ &= \sum_{n=0}^{\infty}\tilde{\alpha}(b)\theta_{t_{n}}(wp_{n}) = \tilde{\alpha}(b)W \nonumber 
    \end{split} 
\end{equation}

Now it can be checked that for $(\widetilde{B})' \cap \widetilde{B} \rtimes_{(\alpha,u)} \mathcal{G} = \mathcal{Z}(B)$ if and only if for a.e. $x \in X$ and for all $g \in \Gamma_{x} \backslash \{e\}$, $\Tilde{\alpha_{g}} \in \Aut(\Tilde{B_{x}})$ is properly outer. By Claim 1 and Claim 2, we then have that $(\widetilde{B})' \cap \widetilde{M} = \mathcal{Z}(\widetilde{B})$
if and only if for a.e. $x \in X$ and for all $g \in \Gamma_{x} \backslash \{e\}$, $\Tilde{\alpha}_{g} \in \Aut(\widetilde{B_{x}})$ is outer, and by the equivalence between outerness of $\Tilde{\alpha}$ and central freeness of $\alpha$, we have the desired conclusion. 
\end{proof}

We shall now present a new proof of the classification theorem for centrally free actions of amenable discrete measured groupoids on fields of injective factors. This has already been proved in further generality, for (not necessarily centrally) free actions by Masuda in \cite{MR4484234}. We first need the following lemma, essentially due to Ocneanu on approximate innerness of cocycle self conjugacies. Our formulation is very similar to  \cite[Lemma 3.2]{MR4484234} but several similar formulations have appeared in the literature before, for example in \cite{MR766264}, \cite{MR842413}, \cite{MR1156672}. In what follows, we shall follow the same terminology as in \cite[Section 5.1]{MR807949}. In particular, we shall denote for a von Neumann algebra $M$ and a free ultrafilter $\omega$ on $\mathbb{N}$, the Ocneanu ultrapower, and the $\omega$-centralizing sequence algebra by $M^{\omega}$ and $M_{\omega}$ respectively.  

\begin{lemma}
\label{CocycleSelfConjugacyNonCartan}
Let $B$ be an injective factor, $G$ be a countable amenable group, and $\alpha$ be a centrally free action of $G$ on $B$. Suppose we have $\theta \in \Aut(B)$ such that $ \modnew(\theta) = 1$ and $v_{g} \in \mathcal{U}(B)$ a 1-cocycle for $\alpha$ that satisfies $\theta \circ \alpha_{g} \circ \theta^{-1} = Ad (v_{g}) \circ \alpha_{g}$ for all $g \in G$. Then there exists a sequence $u_{n} \in \mathcal{U}(B)$ such that
\begin{equation}
    \Ad (u_{n}) \rightarrow \theta \; \text{and} \; \lim_{n} ||u_{n}^{*}v_{g}\alpha_{g}(u_{n}) - 1||_{\phi}^{\#} = 0 \; \text{for all} \; g \in G \nonumber
\end{equation}
\end{lemma}

\begin{proof}
Since $\theta$ is trivial on the flow of weights, by \cite[Theorem 1]{MR1179014}, $\theta \in \overline{\Int}(B)$. Hence we have a sequence $w_{n} \in \mathcal{U}(B)$ such that $\Ad(w_{n})$ converges to $\theta$ in the usual topology in $\Aut(B)$. For a faithful normal state $\phi$, we then have that: 
\begin{equation}
    \| \phi \circ \Ad(w_{n}) - \phi \circ \theta \| \rightarrow 0 \nonumber 
\end{equation}
We fix now a free ultrafilter $\omega$ on $\mathbb{N}$. Clearly the above equation also converges as $n \rightarrow \omega$. Now note that as $n \rightarrow \omega$, $\|w_{n}\|_{\phi}^{\#} \nrightarrow 0$, i.e., $w_{n}$ is not a trivial sequence. Now let $x_{n}$ be a sequence such that $\|x_{n}\|^{\#}_{\phi} \rightarrow 0$ and note that as $n \rightarrow \omega$, we have that
\begin{equation}
    \phi(\Ad(w_{n})(x_{n}x_{n}^{*})) + \phi(\Ad(w_{n})(x_{n}^{*}x_{n})) \rightarrow \phi(\theta(x_{n}x_{n}^{*})) + \phi(\theta(x_{n}^{*}x_{n})) \rightarrow 0 \nonumber
\end{equation}
as $\phi \circ \theta$ is also a faithful normal state on $B$. Hence the sequence $(w_{n})$ normalizes the trivial sequences and hence $W \coloneqq (w_{n}) \in B^{\omega}$. Since $\Ad(w_{n}) \circ \alpha_{g} \circ \Ad(w_{n}^{*}) \rightarrow \Ad(v_{g}) \circ \alpha_{g}$ in $\Aut(B)$, we have that
\begin{equation}
||\phi \circ (\Ad(w_{n}) \circ \alpha_{g} \circ \Ad(w_{n}^{*})) - \phi \circ (\Ad(v_{g}) \circ \alpha_{g})|| \rightarrow 0 \text{ as } n \rightarrow \omega  \nonumber
\end{equation}  
We can check from this that we have: 
\begin{equation}
    ||[\phi,w_{n}^{*}v_{g}\alpha_{g}(w_{n})]|| \rightarrow 0 \text{ as } n \rightarrow \omega \nonumber
\end{equation}
Denoting still the induced action of $G$ on $B_{\omega}$ by $\alpha$, we note from the previous equation that the element $V_{g} = W^{*}v_{g}\alpha_{g}(W) \in \mathcal{U}(B_{\omega})$. Next, we check that: 
\begin{equation}
    \begin{split}
        V(g)\alpha_{g}(V(h))V(gh)^{*} &= W^{*}v_{g}\alpha_{g}(W)\alpha_{g}(W^{*}v_{h}\alpha_{h}(W))V(gh)^{*} \\
        &= W^{*}v_{g}\alpha_{g}(v_{h})\alpha_{gh}(W)V(gh)^{*} \\
        &= W^{*}v_{g}\alpha_{g}(v_{h})\alpha_{gh}(W)\alpha_{gh}(W)^{*}v_{gh}^{*}W\\
        &= W^{*}v_{g}\alpha_{g}(v_{h})v_{gh}^{*}W \\
        &= 1 \nonumber
    \end{split}
\end{equation}
and hence $V$ is a 1-cocycle for the action $\alpha$ of $G$ on $B_{\omega}$. Clearly, $\alpha$ is still a centrally free action on $B^{\omega}$. By the final remark in \cite[Pg. 42]{MR807949}, centrally free actions are strongly free and semiliftable (we refer the reader to \cite[Section 5.2]{MR807949} for the appropriate definitions) and thus satisfy the hypotheses of Ocneanu's one cohomology vanishing result as in \cite[Proposition 7.2]{MR807949}. Using the proposition, there exists $Z \in \mathcal{U}(B_{\omega})$ such that $Z^{*}V_{g}\alpha_{g}(Z) = 1$. 

As in \cite[Proposition 1.1.3]{MR394228}, we can choose a representative sequence $(z_{n})$ for the unitary $Z \in \mathcal{U}(B_{\omega})$ with $z_{n} \in \mathcal{U}(B)$. Since $Z$ is $\omega$-centralizing, we have that $\Ad(z_{n}) \rightarrow \id$ as $n \rightarrow \omega$. Now we define a sequence of unitaries $u_{n} \in \mathcal{U}(B)$ by $u_{n} \coloneqq w_{n}z_{n}$. Then we have that
\begin{equation}
    \lim_{n \rightarrow \omega}\Ad(u_{n}) = \lim_{n \rightarrow \omega}\Ad(w_{n}) \circ \Ad(z_{n}) = \theta \nonumber
\end{equation}
Now note that in $B_{\omega}$, we have that $Z^{*}W^{*}v_{g}\alpha_{g}(WZ) = 1$ and hence we have that $u_{n}^{*}v_{g}\alpha_{g}(u_{n}) \rightarrow 1$ as $n \rightarrow \omega$ as required. The u-topology on $\Aut(B)$ and the topology of pointwise convergence on the space of 1-cocycles are both metrizable, hence we can pass to a subsequence and assume that the conclusion of the theorem remains valid as $n \rightarrow \infty$, which completes our proof. 
\end{proof}

We shall now prove the classification theorem for centrally free actions of a discrete measured amenable  groupoid on a measurable field of injective type $\textrm{III}$ factors. At the heart of the proof is a cohomology lemma proved in \cite[Theorem 3.5]{MR4124434}. As the authors mention in \cite{MR4124434}, the formulation here is different than the cohomology lemmas that appear in \cite{MR766264} and \cite{MR842412}, but the proof is similar to the one in \cite[Theorem 5.5]{MR842412}. 

\begin{theorem}
\label{GroupoidActionsClassification}
Let $(X, \mu)$ be a standard probability space and let $\mathcal{G}$ be a discrete measured groupoid which is amenable with $\mathcal{G}^{(0)} = X$. Let $\alpha$ and $\beta$ be centrally free actions of $\mathcal{G}$ on a field of injective type $\textrm{III}$ factors $(B_{x})_{x\in X}$ such that $X = \mathcal{G}^{(0)}$. Then we have the following: 
\begin{enumerate}[(i)]
    \item The actions $\alpha$ and $\beta$ are cocycle conjugate if and only if $ \modnew(\alpha) \sim  \modnew(\beta)$,
    \item The actions $\alpha$ and $\beta$ are strongly cocycle conjugate if and only if $ \modnew(\alpha) =  \modnew(\beta)$.
\end{enumerate}
\end{theorem}

\begin{proof}
It is straightforward to check the `only if' directions of both statements. We shall first deduce (ii) from (i) and then prove (i). Now suppose that $ \modnew(\alpha) \sim  \modnew(\beta)$ and let $\sigma_{x} \in \Aut_{\theta_{x}}(\mathcal{C}_{B_{x}})$ be the automorphisms for all $x \in X$ such that $\sigma_{t(g)}\circ \modnew(\beta_{g}) \circ \sigma_{s(g)}^{-1} = \modnew(\alpha_{g})$ for a.e. $g \in \mathcal{G}$. By surjectivity of the Connes-Takesaki module map (c.f. \cite[Theorem 1.1]{MR1643188}), we can find automorphisms $\tilde{\sigma_{x}} \in \Aut(B_{x})$, such that $ \modnew(\tilde{\sigma_{x}}) = \sigma_{x}$ for all $x \in X$ and then consider the action $\beta'_{g} = \tilde{\sigma}_{t(g)}\circ \beta_{g} \circ \tilde{\sigma}_{s(g)}^{-1}$ for all $g \in \mathcal{G}$. We note here that $ \modnew(\alpha_{g}) =  \modnew(\beta'_{g})$ for a.e. $g \in \mathcal{G}$. Since $\beta$ and $\beta'$ are strongly cocycle conjugate by (ii), we have a field of automorphisms $\theta_{x} \in \Aut(B_{x})$ such that $ \modnew(\theta_{x}) = 1$ for a.e. $x \in X$ and a 1-cocycle $(c_{g})_{g \in \mathcal{G}}$ such that for a.e. $g \in \mathcal{G}$, we have
\begin{equation}
    \theta_{t(g)}\circ \beta'_{g} \circ \theta_{s(g)}^{-1} = \Ad(c_{g}) \circ \alpha_{g} \nonumber
\end{equation}
and hence $(\theta_{x} \circ \tilde{\sigma}_{x})_{x \in X}$ and $(c_{g})_{g \in \mathcal{G}}$ gives the required cocycle conjugacy between $\alpha$ and $\beta$. We prove the `if' direction of (ii) below. 

Let $\mathcal{R}$ be the non singular countable equivalence relation induced by $\mathcal{G}$ defined by
\begin{equation}
    \mathcal{R} = \{(t(g),s(g)) \; | \; g \in \mathcal{G}\}. \nonumber
\end{equation}
The amenability of $\mathcal{G}$ implies that $\mathcal{R}$ is an amenable equivalence relation and by the \cite{MR662736}, we know that $X$ can be partitioned into $X_{1}$ and $X_{2}$ where each $X_{i}$ is $\mathcal{R}$ invariant, $\mathcal{R}|_{X_{1}}$ has finite orbits and hence admits a fundamental domain and $\mathcal{R}|_{X_{2}}$ is generated by a single non-singular automorphism. From this it follows that there exists a measurable lift $q: \mathcal{R} \rightarrow \mathcal{G}$ such that 
\begin{itemize}
    \item $s(q(x,y)) = y$ and $t(q(x,y)) = x$, and
    \item $q$ is a morphism, i.e., $q(x,y)q(y,z) = q(x,z)$ for a.e. $(x,y) \in \mathcal{R}$
\end{itemize}
We let $\Gamma$ denote the family $(\Gamma_{x})_{x \in X}$ and define an action $\delta$ of $\mathcal{R}$ on $\Gamma$ as follows: For all $(x,y) \in \mathcal{R}$, let us denote by $\delta_{(x,y)}: \Gamma_{y} \rightarrow \Gamma_{x}$ the group isomorphism given by $\delta_{(x,y)}(g) = q(x,y)gq(x,y)^{-1}$ for all $g \in \Gamma_{y}$. Clearly $\delta_{(x,y)}\circ \delta_{(y,z)} = \delta_{(x,z)}$ as $q$ is a morphism. Hence the groupoid $\mathcal{G}$ is isomorphic to a semidirect product $\Gamma \rtimes_{\delta} \mathcal{R}$. 
For convenience of notation we shall use $\alpha_{(x,y)}$ and $\beta_{(x,y)}$ to denote $\alpha_{q(x,y)}$ and $\beta_{q(x,y)}$.

Let $Z^{1}_{\alpha}(\Gamma_{x}, \mathcal{U}(B_{x}))$ be the space of all 1-cocycles $c_{g}$ for the action $\alpha^{x}$. We denote by $P_{x}$ the Polish space of `strong' cocycle conjugacies between the actions $\alpha^{x}$ and $\beta^{x}$ on $B_{x}$. More precisely, let $P_{x} \subset \Aut(B_{x}) \times Z^{1}_{\alpha}(\Gamma_{x}, \mathcal{U}(B_{x}))$ be the collection of all $(\theta,(c_{g})_{g \in \Gamma_{x}})$ such that the following conditions are satisfied: 
\begin{itemize}
    \item $\theta \in \Aut(B_{x})$ and $ \modnew(\theta) = 1$,
    \item $c_{g} \in \mathcal{U}(B_{x})$ for all $g \in \Gamma_{x}$ and $(c_{g})_{g \in \Gamma_{x}}$ is a 1 cocycle for $\alpha^{x}$,
    \item $\theta \circ \beta_{g} \circ \theta^{-1} = \Ad(c_{g}) \circ \alpha_{g}$ for all $g,h \in \Gamma_{x}$.
\end{itemize}

The u-topology on $\Aut(B_{x})$ and the topology of pointwise convergence on the space of $1$-cocycles endows $P_{x}$ with a natural topology. We now note that since for the restricted actions of the isotropy groups $ \alpha^{x}, \beta^{x}:\Gamma_{x} \curvearrowright B_{x}$, $ \modnew(\alpha^{x}) =  \modnew(\beta^{x})$, the actions are strongly cocycle conjugate by Theorem \ref{StrongCocycleConjugacy}.  Thus we have that $P_{x}$ is non-empty. We have a family of Polish groups $\mathcal{U}(B_{x})$ acting on $P_{x}$ as follows: for $w \in U(B_{x})$ and $(\theta,v_{g}) \in P_{x}$, we define
\begin{equation}
w\cdot\theta = \Ad(w) \circ \theta \; \text{and} \; (w\cdot v)_{g} = wv_{g}\alpha_{g}(w^{*}) \nonumber    
\end{equation}
It is easy to check that $w \cdot \theta \circ \alpha_{g} \circ w\cdot \theta^{-1} = \Ad(w\cdot v)_{g} \circ \alpha_{g}$ and hence this action is well defined. We claim that this action has dense orbits. To check this let $(\theta,c)$ and $(\delta, d)$ be two elements of $P_{x}$. By Lemma \ref{CocycleSelfConjugacyNonCartan}, we consider the sequences of unitaries $w_{n}$ and $z_{n}$ such that the following holds: 
\begin{itemize}
    \item $\Ad(w_{n}) \rightarrow \theta$ and $\Ad(z_{n}) \rightarrow \delta$ in $\Aut(B_{x})$ 
    \item $\|w_{n}^{*}c_{g}\alpha_{g}(w_{n}) - 1\|^{\#}_{\phi} \rightarrow 0$ and $\|z_{n}^{*}d_{g}\alpha_{g}(z_{n}) - 1\|^{\#}_{\phi} \rightarrow 0$
\end{itemize}
Since $\Aut(B_{x})$ is a Polish group, the sequence of unitaries $u_{n} = z_{n}w_{n}^{*}$ satisfy $\Ad(u_{n}) \circ \theta \rightarrow \delta$. We recall from the proof of Lemma \ref{CocycleSelfConjugacyNonCartan} that for a free ultrafilter $\omega$ on $\mathbb{N}$, we can assume that the sequences $W \coloneqq (w_{n})$ and $Z \coloneqq (z_{n})$ are in $\mathcal{U}(B_{x}^{\omega})$. Since the sequence $z_{n}^{*}d_{g}\alpha_{g}(z_{n}) - 1 \rightarrow 0$ *-strongly and $Z \in B_{x}^{\omega}$, we also have that $z_{n}\alpha_{g}(z_{n}^{*}) \rightarrow d_{g}$ *-strongly as $n \rightarrow \omega$. Now for the same reason, as $n \rightarrow \omega$ we also have: 
\begin{equation}
\|z_{n}w_{n}^{*}c_{g}\alpha_{g}(w_{n}z_{n}^{*}) - d_{g}\|^{\#}_{\phi} \leq \|z_{n}w_{n}^{*}c_{g}\alpha_{g}(w_{n}z_{n}^{*}) - z_{n}\alpha_{g}(z_{n}^{*})\|^{\#}_{\phi} + \|z_{n}\alpha_{g}(z_{n}^{*}) - d_{g}\|^{\#}_{\phi} \rightarrow 0 \nonumber
\end{equation}
Therefore the unitaries $u_{n}$ also satisfy $u_{n} \cdot c_{g} \rightarrow d_{g}$ in the topology of pointwise convergence in $Z^{1}_{\alpha}(\Gamma_{x},\mathcal{U}(B_{x}))$ as $n \rightarrow \omega$. Once again by passing to a subsequence, we have proved our claim. Now, we follow the same notation and the same strategy as in \cite[Theorem 5.7]{MR4124434} and define the measurable field of continuous group isomorphisms $\mathcal{R} \ni (x,y) \rightarrow \alpha_{(x,y)}: \mathcal{U}(B_{y}) \rightarrow \mathcal{U}(B_{x})$ and the homeomorphisms $\mathcal{R} \ni (x,y) \rightarrow \gamma_{(x,y)}: P_{y} \rightarrow P_{x}$ such that to an element $(\theta,(c_{g})_{g \in \Gamma_{y}})$, the map $\gamma_{(x,y)}$ does the following: 
\begin{itemize}
    \item $\theta \in \Aut(B_{y})$ is sent to the element $\alpha_{(x,y)}\circ \theta \circ \beta_{(y,x)} \in \Aut(B_{x})$,
    \item $(c_{g})_{g \in \Gamma_{y}}$ is sent to the cocycle $(\alpha_{(x,y)}(c_{\delta_{(y,x)}(g)}))_{g \in \Gamma_{x}}$.
\end{itemize}

It can be easily checked that $\alpha_{(x,y)} \circ \theta \circ \beta_{(y,x)}$ and the 1-cocycle $g \rightarrow \alpha_{(x,y)}(c_{\delta_{(y,x)}(g)})$ form a cocycle conjugacy between the actions $\alpha^{x}$ and $\beta^{x}$ of the isotropy groups $\Gamma_{x}$ on $B_{x}$. Since $ \modnew(\alpha_{(x,y)}) =  \modnew(\beta_{(x,y)})$ for all $(x,y) \in \mathcal{R}$ and $ \modnew(\theta) = 1$, we have that this is indeed a strong cocycle conjugacy and hence $\gamma_{(x,y)}(\theta,(c_{g})_{g \in \Gamma_{y}})$ is an element in $P_{x}$. 

Now \cite[Theorem 3.5]{MR4124434} gives us a measurable field of cocycle conjugacies $\pi_{x} = (\theta_{x},(c_{g})_{g \in \Gamma_{x}})$ together with a 1-cocycle $\mathcal{R} \ni (x,y) \rightarrow c_{(x,y)} \in \mathcal{U}(B_{x})$ that satisfies $\pi_{x} = c_{(x,y)} \cdot \gamma_{(x,y)}(\pi_{y})$. Since $\pi$ is a section, $\theta_{x} \in \Aut(B_{x})$ clearly satisfies $ \modnew(\theta_{x}) = 1$. For an element $g \in \Gamma_{x}$, we then get that $c_{(x,y)}\alpha_{(x,y)}(c_{\delta_{(y,x)}(g)})\alpha_{g}(c_{(x,y)}^{*}) = c_{g}$. Hence it follows that: 

\begin{equation}
\label{eq1}
c_{(x,y)}\alpha_{(x,y)}(c_{\delta_{(y,x)}(g)}) = c_{g}\alpha_{g}(c_{(x,y)})
\end{equation}

Now any element $b \in \mathcal{G}$ with $s(b) = y$ and $t(b) = x$ can be written as $b = q(x,y)b_{0}$ for a unique element $b_{0} \in \Gamma_{y}$. We define a new field of unitaries $(d_{g} \in \mathcal{U}(B_{t(g)}))_{g \in \mathcal{G}}$ as follows: we put $d_{g} \coloneqq c_{g}$ if $g \in \Gamma_{x}$ for some $x \in X$, we put $d_{q(x,y)} \coloneqq c_{(x,y)}$ and finally for any element of the groupoid $b$ as above, we define $d(g) \coloneqq c_{(x,y)}\alpha_{(x,y)}(c_{b_{0}})$. Note that $b$ can also be written as $\delta_{(x,y)}(b_{0})q(x,y)$ and by Equation \ref{eq1}, $d_{g} = c_{\delta_{(x,y)}(b_{o})}\alpha_{\delta_{(x,y)}(b_{0})}(c_{(x,y)})$ and hence this is well defined. We now show that it is indeed a 1-cocycle for $\alpha$. Let $a,b$ be elements of $\mathcal{G}$ with $s(a) = z$, $t(a) = s(b) = y$ and $t(b) = x$. Note that: 
\begin{equation}
    ba = q(x,y)b_{0}q(y,z)a_{0} = \delta_{(x,y)}(b_{0})q(x,z)a_{0} = \delta_{(x,y)}(b_{0})\delta_{(x,z)}(a_{0})q(x,z) = q(x,z)\delta_{(y,z)}^{-1}(b_{0})a_{0} \nonumber
\end{equation}
Now we calculate that:
\begin{equation}
    \begin{split}
        d_{b}\alpha_{b}(d_{a}) & = c_{(x,y)}\alpha_{(x,y)}(c_{b_{0}}\alpha_{b_{0}}(d_{q(y,z)a_{0}})) =  c_{(x,y)}\alpha_{(x,y)}(c_{b_{0}}\alpha_{b_{0}}(d_{\delta_{(y,z)}(a_{0})q(y,z)}))\\ &= c_{(x,y)}\alpha_{(x,y)}(c_{b_{0}}\alpha_{b_{0}}(c_{\delta_{(y,z)}(a_{0})})\alpha_{b_{0}}\alpha_{\delta_{(y,z)}(a_{0})}(c_{(y,z)})) \\ &= c_{(x,y)}\alpha_{(x,y)}(c_{b_{0}\delta_{(y,z)}(a_{0})}\alpha_{b_{0}\delta_{(y,z)}(a_{0})}(c_{(y,z)})) \\& = c_{(x,y)}\alpha_{(x,y)}(c_{(y,z)}\alpha_{(y,z)}(c_{\delta^{-1}_{(y,z)}(b_{0})a_{0}})) \\ &= c_{(x,z)}\alpha_{(x,z)}(c_{\delta^{-1}_{(y,z)}(b_{0})a_{0}})) = d_{ba}
    \end{split}
    \nonumber
\end{equation}

This means that $(d_{g})_{g \in \mathcal{G}}$ is a 1-cocycle for $\alpha$ and along with $(\theta_{x})_{x \in X}$, this gives a strong cocycle conjugacy between $\alpha$ and $\beta$.
\end{proof}

Before we state the classification theorem for regular subalgebras, we define the following notation for convenience. Suppose we have two discrete measured groupoids $\mathcal{G}_{1}$ and $\mathcal{G}_{2}$ with unit spaces $X_{1}$ and $X_{2}$ respectively. Suppose $\sigma: \mathcal{G}_{1} \rightarrow \mathcal{G}_{2}$ is a groupoid isomorphism such that $\sigma(X_{1}) = X_{2}$. Suppose $\alpha_{1}$ and $\alpha_{2}$ are actions of $\mathcal{G}_{1}$ and $\mathcal{G}_{2}$ on fields of factors $(B_{x})_{x \in X_{1}}$ and $(D_{y})_{y \in X_{2}}$ respectively. We say that $ \modnew(\alpha_{1})$ and $ \modnew(\alpha_{2})$ are conjugate via $\sigma$ and denote it by $ \modnew(\alpha_{1}) \sim_{\sigma}  \modnew(\alpha_{2})$ if there exists a family of automorphisms $\phi_{x}: \mathcal{C}_{B_{x}} \rightarrow \mathcal{C}_{D_{\sigma(x)}}$ such that for a.e. $g \in \mathcal{G}$, we have 
\begin{equation}
    \phi_{t(g)} \circ  \modnew(\alpha_{g}) \circ \phi_{s(g)}^{-1} =  \modnew(\beta_{\sigma(g)}) \nonumber
\end{equation}
 
\begin{theorem}
\label{RegularInclusionsClassification}
For $i \in \{1,2\}$, let $B_{i} \subset M_{i}$ be two inclusions of von Neumann algebras, where $M_{i}$ are injective factors and $B_{i}$ are regular subalgebras with faithful normal conditional expectations $E_{i}: M \rightarrow B_{i}$ satisfying $\mathcal{Z}(\tilde{B_{i}}) = (\tilde{B_{i}})'\cap \widetilde{M}$. Let $\mathcal{Z}(B_{i}) = L^{\infty}(X_{i}, \mu_{i})$ and let the associated discrete amenable groupoids be $\mathcal{G}_{i}$, with $\mathcal{G}_{i}^{(0)} = X_{i}$ and cocycle actions $(\alpha_{i},u_{i})$ on $B_{i}$. Then there exists an isomorphism $\theta: M_{1} \rightarrow M_{2}$ satisfying $\theta(B_{1}) = B_{2}$ if and only if $\type(B_{1}) = \type(B_{2})$ and there is a nonsingular isomorphism $\sigma: \mathcal{G}_{1} \rightarrow \mathcal{G}_{2}$ with $\sigma(X_{1}) = X_{2}$ such that $\modnew(\alpha_{1}) \sim_{\sigma} \modnew(\alpha_{2})$.  
\end{theorem}

\begin{proof} If $B_{1}$ and $B_{2}$ are of type $\textrm{I}_{n}$ or $\textrm{II}_{1}$, this follows from \cite{MR662736} and \cite{MR4124434}. In the remaining cases, it is enough to show that if the preceding two conditions are satisfied, then the actions $(\alpha_{1},u_{1})$ and $(\alpha_{2}, u_{2})$ are cocycle conjugate, hence giving isomorphic crossed products. By Lemma \ref{CentrallyFreeCondition}, $(\alpha_{1}, u_{1})$ and $(\alpha_{2},u_{2})$ are centrally free. By Theorem \ref{2CohomologyVanishingGroupoidsAlternate} (for type 
$\textrm{I}_{\infty}$,  $\textrm{II}_{\infty}$ and type $\textrm{III}$),  we can assume that $\alpha_{1}$ and $\alpha_{2}$ are genuine actions. If $B_{1}$ and $B_{2}$ are not of type $\textrm{III}$, then $\alpha_{1}$ and $\alpha_{2}$ are cocycle conjugate by \cite[Theorem 1.2]{MR842413}. If they are both of type $\textrm{III}$, then cocycle conjugacy follows by Theorem \ref{GroupoidActionsClassification}. The `only if' direction is easy to check and we avoid the details here.   
\end{proof}

\section{Actions on fields of ergodic hyperfinite equivalence relations}
\label{sectioncartan}

We begin this section by recalling the notion of Maharam extensions \cite{MR169988} and looking at the canonical core of a type $\textrm{III}$ factor as a von Neumann algebra of an equivalence relation. Let $\mathcal{R}$ be a countable nonsingular equivalence relation on a standard probability space $(X, \mu)$. Let $\Sigma$ be a locally compact group with left Haar measure $\lambda$ and let $Z:\mathcal{R} \rightarrow \Sigma$ be a cocycle. The \textit{skew product equivalence relation} $\mathcal{R} \times_{Z} \Sigma$ is defined on the product space $(X\times \Sigma, \mu \times \lambda)$ by $(x,g) \sim (y,h)$ if and only if $(x,y) \in \mathcal{R}$ and $g^{-1}h = Z(x,y)$. We note that $\mathcal{R} \times_{Z} \Sigma$ is a countable $(\mu \times \lambda)$-nonsingular Borel equivalence relation on $X \times \Sigma$.

Recall from Section \ref{Equivalence relations and Cartan subalgebras} the Radon-Nikodym 1-cocycle $D_{\mu}$ for an equivalence relation $\mathcal{R}$ on $(X, \mu)$. We note that $\log(D_{\mu}): \mathcal{R} \rightarrow \mathbb{R}$ is also a 1 cocycle to the additive group $\mathbb{R}$. 

\begin{definition}
The skew product of $\mathcal{R}$ with the cocycle $\log(D_{\mu}): \mathcal{R} \rightarrow \mathbb{R}$ is called the \textit{Maharam extension} of $\mathcal{R}$ and is denoted by $c(\mathcal{R})$. We can easily check that  $c(\mathcal{R})$ is the countable nonsingular Borel equivalence relation on $(X \times \mathbb{R}, \mu \times \nu)$ where $\nu$ with respect to the Lebesgue measure is given by: 
\begin{equation}
    \nu(A) = \int_{A}e^{-t} \; dt \nonumber
\end{equation}
and $(x,s) \sim (y,t)$ if and only if $(x,y) \in \mathcal{R}$ and $t-s = \log(D_{\mu})(x,y)$.
\end{definition}

If we further assume that $\mathcal{R}$ is of type $\textrm{III}$ and consider the action $\mathbb{R} \curvearrowright (X \times \mathbb{R}, \mu \times \nu)$ where $d\nu = e^{-t} \;dt$ and where $\mathbb{R}$ acts by translation on the second component, we see that this action preserves the Maharam extension and hence induces an action on the abelian von Neumann algebra of $c(\mathcal{R})$-invariant functions $L^{\infty}(X \times \mathbb{R}, \mu \times \nu)^{c(\mathcal{R})}$, this is called the associated flow of the equivalence relation $\mathcal{R}$. If $\mathcal{R}$ is assumed to be ergodic, the only invariant functions with respect to this action are the constant functions. This precisely gives the smooth flow of weights on the factor $L(\mathcal{R})$. The type classification of von Neumann algebras and the classification of Borel equivalence relations agree in the sense that the nonsingular Borel equivalence relation $\mathcal{R}$ is of a particular type if and only if its associated von Neumann algebra $L(\mathcal{R})$ is of the same type. It is a well-known fact that for a nonsingular countable equivalence relation $\mathcal{R}$ on $(X,\mu)$, the von Neumann algebra $L(c(\mathcal{R}))$ associated with its Maharam extension is canonically isomorphic to the continuous core $\widetilde{L(\mathcal{R})}$ of $L(\mathcal{R})$ with respect to the faithful normal state associated to the measure $\mu$. For details about most of the discussion above, we refer the reader to \cite{MR0390172}.

Suppose we have an inclusion $A \subset B \subset M$ where $M$ is an injective factor, $A$ is a Cartan subalgebra of $M$, and $B$ is a regular subalgebra of $M$.  As in section 2.3, we can write $A \subset M$ as $L^{\infty}(X,\mu) \subset L(\mathcal{R})$ for an ergodic hyperfinite equivalence relation $\mathcal{R}$ on $(X, \mu)$. By Proposition \ref{StrongNormality}, we have that  $B$ is isomorphic to $L(\mathcal{S})$ where $\mathcal{S}$ is a strongly normal subequivalence relation of $\mathcal{R}$. As in the previous discussion, we notice that at the level of the continuous cores, $\tilde{A} \subset \widetilde{B} \subset \widetilde{M}$ is isomorphic to the inclusion $L^{\infty}(X \times \mathbb{R}, \mu \times \nu) \subset L(c(\mathcal{S})) \subset L(c(\mathcal{R}))$. In particular, $\tilde{A}$ is a Cartan subalgebra of $\widetilde{M}$ and hence maximal abelian. Now, $\widetilde{B}'\cap \widetilde{M} \subset \tilde{A}'\cap \widetilde{M} = \tilde{A} \subset \widetilde{B}$ and therefore $\mathcal{Z}(\widetilde{B}) = \widetilde{B}' \cap \widetilde{M}$. This will appear later in Corollary \ref{CartanIffRelativeCommutant}

We now shift our focus to 2-cohomology vanishing results for actions on Cartan inclusions. Similar to Theorem \ref{2CohomologyVanishingGroupoidsAlternate}, we prove the following theorem here.  

\begin{theorem}
\label{CohomologyVanishingGroupoidsCartanAlternate}
Let $(X, \mu)$ be a probability space and $\mathcal{G}$ be a discrete measured groupoid with unit space $X$. Let $(A_{x} \subset B_{x})_{x \in X}$ be a measurable field of infinite factors with Cartan subalgebras. Let $(\alpha,u)$ be a free cocycle action of $\mathcal{G}$ on the field $(A_{x} \subset B_{x})_{x\in X}$, where for $g \in \mathcal{G}$, we have that $\alpha_{g}(A_{s(g)}) = A_{t(g)}$, then the cocycle $u$ is a coboundary: there exists a measurable field of unitaries normalizing the Cartan subalgebras $\mathcal{G} \ni g \mapsto w_{g} \in \mathcal{N}_{B_{t(g)}}(A_{t(g)})$ such that for all $(g,h) \in \mathcal{G}^{(2)}$, we have 
\begin{equation}
    u(g,h) = \alpha_{g}(w_{h}^{*})w_{g}^{*}w_{gh}.  \nonumber
\end{equation}
\end{theorem}

\begin{proof}
We follow the same notation as in Theorem \ref{2CohomologyVanishingGroupoidsAlternate}. We consider $\mathcal{G}_{x} = \{g \in \mathcal{G}, \; t(g) = x\}$. For all $x \in X$, let $\mathcal{S}_{x}$ denote the corresponding ergodic hyperfinite equivalence relation on a standard probability space $(Y_{x}, \mu_{x})$ as in \cite{MR578656}. Since $\mathcal{S}_{x}$ is infinite for almost every $x \in X$, we can assume that $\mathcal{S}_{x}$ is isomorphic to $\mathcal{C}_{\mathbb{N}} \times \mathcal{S}_{x}$ where $\mathcal{C}_{\mathbb{N}}$ denotes the complete equivalence relation on $\mathbb{N}$, i.e., $m \sim n$ for all $m,n \in \mathbb{N}$. Once again by the main theorem in \cite{MR578656}, we can take a measurable field $(\phi_{x})_{x \in X}$ of *-isomorphisms $B_{x}$ and $\mathcal{B}(l^{2}(\mathcal{G}_{x}))$ such that $\phi_{x}(A_{x}) = A_{x} \otimes l^{\infty}(\mathcal{G}_{x})$ for a.e. $x \in X$ where $l^{\infty}(\mathcal{G}_{x})$ denotes the diagonal Cartan subalgebra in $\mathcal{B}(l^{2}(\mathcal{G}_{x}))$.  

Once again for $x \in X$, we denote by $D_{x}$ the infinite factor $B_{x} \otimes \mathcal{B}(l^{2}(\mathcal{G}_{x}))$, by $C_{x}$ the Cartan subalgebra $A_{x} \otimes l^{\infty}(\mathcal{G}_{x})$ and by $\alpha'_{g}$ the action
$\phi_{t(g)} \circ \alpha_{g} \circ \phi_{s(g)}^{-1} : B_{s(g)} \otimes \mathcal{B}(l^{2}(\mathcal{G}_{s(g)})) \rightarrow B_{t(g)} \otimes \mathcal{B}(l^{2}(\mathcal{G}_{t(g)}))$ such that $\alpha'_{g}(A_{s(g)} \otimes l^{\infty}(\mathcal{G}_{s(g)})) = A_{t(g)} \otimes l^{\infty}(\mathcal{G}_{t(g)})$. We consider the matrix units $e_{g,h}^{x}$ in $\mathcal{B}(l^{2}(\mathcal{G}_{x}))$, i.e., the element that sends the basis element $\delta_{h}$ to the basis element $\delta_{g}$ in $l^{2}(\mathcal{G}_{x})$. Denoting the identity element in each $\mathcal{G}_{x}$ by 1 we consider the elements $1 \otimes e^{t(g)}_{g,g} \in C_{s(g)}$ and $\alpha'_{g}(1 \otimes e^{s(g)}_{1,1}) \in C_{t(g)}$ for each $g \in \mathcal{G}$. By ergodicity of the underlying equivalence relation, we can find partial isometries $w_{g} \in D_{t(g)}$ such that $w_{g}^{*}w_{g} = 1 \otimes e^{t(g)}_{g,g}$, $w_{g}w_{g}^{*} = \alpha'_{g}(1 \otimes e^{s(g)}_{1,1})$ and $w_{g}C_{t(g)}w_{g}^{*} = C_{t(g)}$. 

We define now the unitary operator $V_{g} = \sum_{k \in \mathcal{G}_{s(g)}}\alpha'_{g}(1 \otimes e^{s(g)}_{k,1})w_{g}(1 \otimes e^{t(g)}_{g,gk})$ and note that $V_{g}^{*}C_{t(g)}V_{g} = C_{t(g)}$. As in the proof of Theorem \ref{2CohomologyVanishingGroupoidsAlternate}, we have that $V_{g}$ is the sum of partial isometries with orthogonal sources and ranges, and hence $V_{g} \in \mathcal{N}_{D_{t(g)}}(C_{t(g)})$. Now for $k,l \in \mathcal{G}_{s(g)}$, we can check that $\alpha'_{g}(1 \otimes e^{s(g)}_{k,l})= \Ad(V_{g})(1 \otimes e^{t(g)}_{gk,gl})$. We consider once again the operator $\lambda_{g}: l^{2}(\mathcal{G}_{s(g)}) \rightarrow l^{2}(\mathcal{G}_{t(g)})$ given by $\lambda_{g}(\delta_{h}) = \delta_{gh}$ and note that 
\begin{equation}
    \alpha_{g}' = \Ad(V_{g}) \circ (\beta_{g} \otimes \Ad(\lambda_{g})) \nonumber
\end{equation}
where $\beta_{g}: B_{s(g)} \rightarrow B_{t(g)}$ is an action of $\mathcal{G}$ on the field of factors $(B_{x})_{x \in X}$. We let the action $\beta_{g} \otimes \Ad(\lambda_{g})$ be denoted by $\sigma_{g}$. Consider now the two cocycle $u'(g,h) = V_{g}^{*}\alpha'_{g}(V_{h}^{*})\phi_{t(g)}(u_{g,h})V_{gh}$ and note that $u'(g,h) \in \mathcal{N}_{D_{t(g)}}(C_{t(g)})$. Thus $(\sigma,u')$ defines a cocycle action of $\mathcal{G}$ on the field of Cartan inclusions $(C_{x} \subset D_{x})_{x \in X}$. Once again we note that since $\lambda_{g}\circ \lambda_{h} = \lambda_{gh}$, there exists unitaries $u_{1}(g,h) \in \mathcal{N}_{D_{t(g)}}(C_{(t(g)})$ such that $u'(g,h) = u_{1}(g,h) \otimes 1$. For $r \in \mathcal{G}_{t(g)}$, we consider the matrix units $e^{t(g)}_{r,r}$ and define the unitary $W_{g} = \sum_{r \in \mathcal{G}_{t(g)}}u_{1}(g,g^{-1}r) \otimes e^{t(g)}_{r,r}$. Clearly $W_{g} \in \mathcal{N}_{D_{t(g)}}(C_{(t(g)})$, and $\sigma_{g}(W_{h})W_{g} = u'(g,h)W_{gh}$. Therefore the unitaries $Z_{g} \in \mathcal{N}_{D_{t(g)}}(C_{(t(g)})$ defined as $Z_{g} = \phi_{t(g)}^{-1}(V_{g}W_{g})$ satisfy $\alpha_{g}(Z_{h})Z_{g} = u(g,h)Z_{gh}$ as required.    

\end{proof}

The next step is to deduce the invariants for outer actions of amenable groupoids on fields of ergodic equivalence relations as in Theorem 5.7, \cite{MR4124434}. This was done for outer actions of countable amenable groups on the unique ergodic hyperfinite $\textrm{II}_{1}$ and $\textrm{II}_{\infty}$ equivalence relations in \cite{MR864169}. In the type $\textrm{III}$ setting as well, invariants were obtained for such actions on ergodic hyperfinite equivalence relations of type $\textrm{III}_{\lambda}$, $\lambda \in (0,1]$ in \cite{MR906080} and finally in the type $\textrm{III}_{0}$ case in \cite{MR1002120}. One should note here that these results can be considered as equivalence relation versions of Ocneanu's theorem (\cite{MR807949}) and the classification results due to Katayama, Sutherland and Takesaki ( \cite{MR990219} and \cite{MR1621416}). Recently, a unified approach to proving this was given by Masuda in \cite{https://doi.org/10.48550/arxiv.2210.15916}. 

We introduce the notion of the strong cocycle conjugacy in this context before stating the theorem. For $i \in \{1,2\}$, let $\mathcal{R}_{i}$ be countable ergodic equivalence relations on standard probability spaces $(X_{i}, \mu_{i})$. Now suppose $\Psi: (X_{1},\mu_{1}) \rightarrow (X_{2},\mu_{2})$ is a nonsingular measure space isomorphism. Letting $\rho = \log \frac{d(\Psi^{-1})_{*}\mu_{2}}{d\mu_{1}}$, we consider the map: 
\begin{equation}
    \tilde{\Psi}: (X_{1} \times \mathbb{R}, \mu_{1} \times \lambda) \rightarrow (X_{2} \times \mathbb{R}, \mu_{2} \times \lambda), \; \; \tilde{\Psi}(x,s) = (\Psi(x), \rho(x) + s) \nonumber  
\end{equation}
It can be checked that $\Tilde{\Psi}$ is a $\mathbb{R}$-equivariant measure preserving isomorphism and if $\Psi: \mathcal{R}_{1} \rightarrow \mathcal{R}_{2}$ is an orbit equivalence, then $\tilde{\Psi}: c(\mathcal{R}_{1}) \rightarrow c(\mathcal{R}_{2})$ is an orbit equivalence between the Maharam extensions. $\tilde{\Psi}$ induces a $\mathbb{R}$-equivariant isomorphism of the flow of weights given by $\tilde{\Psi}^{*}: L^{\infty}(X_{1} \times \mathbb{R})^{\mathcal{R}_{1}} \rightarrow L^{\infty}(X_{2} \times \mathbb{R})^{\mathcal{R}_{2}}$. Now we consider the measure spaces $(Y_{i}, \nu_{i})$ such that $\phi_{i}: L^{\infty}(Y_{i}) \rightarrow L^{\infty}(X_{i} \times \mathbb{R})^{\mathcal{R}_{i}}$ are the $\mathbb{R}$-equivariant *-isomorphisms that implements the flow of weights $\mathbb{R} \curvearrowright (Y_{i},\nu_{i})$. This gives us a $\mathbb{R}$-equivariant *-isomorphism: 
\begin{equation}
    \phi_{2}^{-1}\circ \tilde{\Psi}^{*} \circ \phi_{1}: L^{ \infty}(Y_{1},\nu_{1}) \rightarrow L^{\infty}(Y_{2},\nu_{2}) \nonumber
\end{equation}
and such a map is implemented by a $\mathbb{R}$-equivariant nonsingular isomorphism $\Psi_{0}: (Y_{1},\nu_{1}) \rightarrow (Y_{2},\nu_{2})$. We denote this isomorphism by $\modnew(\Psi)$. It can be shown that there is a nonsingular factor map $\pi_{i}:X_{i} \times \mathbb{R} \rightarrow Y_{i}$ implementing $\phi_{i}$ that satisfies $\modnew(\Psi) \circ \pi_{1} = \pi_{2} \circ \tilde{\Psi}$ almost everywhere. If $\Psi \in \Aut(\mathcal{R})$, we get the following formula: 
\begin{equation}
    \modnew(\Psi)(\pi_{1}(x,s)) = \rho(x) \cdot \pi_{2}(\Psi(x),s). \nonumber
\end{equation} 
For rigorous proofs of the facts discussed above, we refer the reader to \cite{MR0390172}. Now we are in a position to define the notion of strong cocycle conjugacy in this context.

\begin{definition}
\label{Defstrongcocycleconjugacycartans}
\begin{enumerate}[1.]
\item Let $G$ be a countable group and let $(\alpha,u)$ and $(\beta,v)$ be two cocycle actions of $G$ on Cartan inclusions $A \subset B$ and $C \subset D$ into injective factors. Let us denote the associated ergodic equivalence relations by $\mathcal{R}$ on $(X,\mu)$ and $\mathcal{S}$ on $(Y,\nu)$ respectively. We say that $(\alpha,u)$ and $(\beta,v)$ are \textit{strongly cocycle conjugate} if they are cocycle conjugate as in Definition \ref{DefCocycleConjugacyCartan} and $\theta: B \rightarrow D$ can be chosen such that the corresponding isomorphism $\theta: \mathcal{R} \rightarrow \mathcal{S}$ induces the trivial isomorphism on the flow of weights i.e., $\modnew(\theta) = 1$.

\item Let $(X, \mu)$ be a standard probability space and $\mathcal{G}$ be a discrete measured groupoid with unit space $\mathcal{G}^{0} = X$. Let $(\alpha,u)$ and $(\beta,v)$ be two cocycle actions of $\mathcal{G}$ on fields of injective  factors with Cartan subalgebras $(A_{x} \subset B_{x})_{x \in X}$ and $(C_{x} \subset D_{x})_{x \in X}$. Let $\mathcal{R}_{x}$ and $\mathcal{S}_{x}$ denote the corresponding fields of ergodic equivalence relations. We say that $(\alpha,u)$ and $(\beta,v)$ are \textit{strongly cocycle conjugate} if they are cocycle conjugate as in Definition \ref{DefCocycleConjugacyCartan} and moreover for almost every $x \in X$, $\theta_{x}: B_{x} \rightarrow D_{x}$ can be chosen such that the corresponding isomorphisms $\theta_{x}: \mathcal{R}_{x} \rightarrow \mathcal{S}_{x}$ induce trivial isomorphims on the flow of weights, i.e.,  $ \modnew(\theta_{x}) = 1$.
\end{enumerate}
\end{definition}

\begin{theorem}(See \cite[Theorem 2.4]{https://doi.org/10.48550/arxiv.2210.15916})
\label{conjugaceyOfGroupActionsOnEquivalenceRelation}
Let $B$ be an injective factor with a Cartan subalgebra $A \subset B$ and let the associated ergodic equivalence relation be $\mathcal{S}$. Let $G$ be a countable amenable group and $\alpha_{1}$, $\alpha_{2}$ actions of $G$ on $A \subset B$ such that the induced actions on $\mathcal{S}$ are outer. Then we have the following: 
\begin{enumerate}[(i)]
    \item $\alpha_{1}$ and $\alpha_{2}$ are cocycle conjugate, as in Definition \ref{DefCocycleConjugacyCartan} if and only if $ \modnew(\alpha_{1}) \sim  \modnew(\alpha_{2})$ 
    \item $\alpha_{1}$ and $\alpha_{2}$ are strongly cocycle conjugate, as in Definition \ref{Defstrongcocycleconjugacycartans} if and only if $\modnew(\alpha_{1}) =  \modnew(\alpha_{2})$.
\end{enumerate}
\end{theorem}

Now we briefly discuss ultrapowers of Lebesgue spaces. We first fix a free ultrafilter $\omega$ on $\mathbb{N}$. Let $A \subset B$ be a Cartan subalgebra of an injective factor. Let $\mathcal{R}$ be the corresponding ergodic hyperfinite equivalence relation on a probability space $(X, \mathcal{B}, \mu)$. We consider sequences $(A_{n})_{n \in \mathbb{N}}$ of subsets in $\mathcal{B}$ and define the equivalence relation $\sim$ by putting $(A_{n})_{n \in \mathbb{N}} \sim (B_{n})_{n \in \mathbb{N}}$ if $\mu(A_{n} \Delta B_{n}) \rightarrow 0$ as $n \rightarrow \omega$. Now we define, as in \cite{MR0444900} the ultraproduct $\mathcal{B}^{\omega} = \{(A_{n})_{n \in \mathbb{N}}, A_{n} \in \mathcal{B}\} / \sim$. This forms a Boolean algebra and we define the measure $\mu^{\omega}$ by putting $\mu^{\omega}((A_{n})_{n \in \mathbb{N}}) = \lim_{n \rightarrow \omega} \mu(A_{n})$. Thus $(\mathcal{B}^{\omega}, \mu^{\omega})$ is a measure algebra and every automorphism $T$ of $(X, \mu)$ induces an automorphism $T^{\omega}$ of $\mathcal{B}^{\omega}$ given by $T^{\omega}((A_{n})_{n \in \mathbb{N}}) = (T(A_{n}))_{n \in \mathbb{N}}$. 

We now denote by $\mathcal{B}_{\omega}$ the fixed point algebra of $\{S^{\omega}, \; S \in [\mathcal{R}]\}$. We also denote by $\mu_{\omega}$ the restriction of $\mu^{\omega}$ to $\mathcal{B}_{\omega}$. In other words, $\mathcal{B}_{\omega}$ can be represented by sequences $(A_{n})_{n \in \mathbb{N}}$ such that $\mu(A_{n} \Delta S(A_{n})) \rightarrow 0$ as $n \rightarrow \omega$ for all $S \in [\mathcal{R}]$. If $T \in \Aut(\mathcal{R})$, we have that $T^{\omega}$ leaves $\mathcal{B}_{\omega}$ invariant and we denote by $T_{\omega}$ the restriction of $T^{\omega}$ to $\mathcal{B}_{\omega}$. It can be checked that $T_{\omega}$ is a measure-preserving automorphism of $\mathcal{B}_{\omega}$. By \cite[Lemma 2.4]{MR0444900}, we have that for all $T \in \Aut(\mathcal{R})$, $T_{\omega} = 1$ if and only if $T \in [\mathcal{R}]$. We note that \cite[Lemma 2.3]{MR0444900} in the context of von Neumann algebras, gives the following fact: Let $B$ be an injective factor with a Cartan subalgebra $A\subset B$ and let $\mathcal{S}$ be the associated ergodic equivalence relation. If $\alpha \in \Aut(\mathcal{S})$ is outer, then the induced automorphism on $A^{\omega} \cap B_{\omega}$ is properly outer.

We now prove the following lemma on the approximate innerness of cocycle self conjugacies, i.e. an equivalence relation version of Lemma \ref{CocycleSelfConjugacyNonCartan} 

\begin{lemma}
\label{approximateInnernessOfSelfConjugacyCartan}
Let $G$ be a countable amenable group and $\alpha$ be an action of $G$ on an injective factor $B$ with a Cartan subalgebra $A \subset B$ such that the action is Cartan preserving, i.e., for all $g \in G$, $\alpha_{g}(A) = A$. Let $\mathcal{S}$ be the ergodic equivalence relation associated to the inclusion $A \subset B$ on a standard probability space $(X, \mathcal{B}, \mu)$ and assume that the action of $G$ induced on $\mathcal{S}$ is outer. If $\theta \in \Aut(A \subset B)$ such that $\theta \in \ker(\modnew)$ and $v_{g} \in \mathcal{N}_{B}(A)$ is a 1-cocycle for the action $\alpha$ satisfying $\theta\circ\alpha_{g}\circ \theta^{-1} = \Ad(v_{g}) \circ \alpha_{g}$ for all $g \in G$, there exists a sequence $u_{n} \in \mathcal{N}_{B}(A)$ such that:
 \begin{equation}
     \Ad(u_{n}) \rightarrow \theta \in \Aut(B) \; \; \; \text{and} \; \; \; u_{n}\alpha_{g}(u_{n}^{*}) - v_{g} \rightarrow 0 \nonumber
 \end{equation}
in the u-topology and the *-strong topology for all $g \in G$ respectively.  

\end{lemma}

\begin{proof}
We prove this in a similar way as Lemma \ref{CocycleSelfConjugacyNonCartan}. Let us denote the group of automorphisms of $B$ that preserve $A$ by $\Aut(B,A)$. By \cite[Theorem 3, Theorem 4]{MR578730} and the fact that the equivalence relation is hyperfinite, we have that: 
\begin{equation}
    \Aut(B,A) \cong  Z^{1}(\mathcal{R}, \mathbb{T}) \rtimes \Aut(\mathcal{R}) \nonumber
\end{equation}
where $Z^{1}(\mathcal{R}, \mathbb{T})$ denotes the group of 1 cocycles of $\mathcal{R}$ with values in the circle. The usual topology in $\Aut(B,A)$ corresponds to the u-topology in $\Aut(\mathcal{S})$ defined in Section 2.3 and the standard topology in $Z^{1}(\mathcal{R},\mathbb{T})$ of convergence
in measure as introduced in \cite{MR414775}.

Hence $\theta$ corresponds to a pair $(c,\theta_{0})$ where $\theta_{0} \in \Out(\mathcal{R})$ and $\modnew(\theta_{0}) = 1$. By the main result in \cite{MR550409}, we have that $\theta_{0} \in \overline{[\mathcal{S}]}$. Hence there exists a sequence $\phi_{n} \in [\mathcal{S}]$ such that $\phi_{n} \rightarrow \theta_{0}$ in the u-topology in $\Aut(\mathcal{S})$. Now we write $\mathcal{R} = \cup_{n}\mathcal{R}_{n}$ where each $\mathcal{R}_{n}$ is of type $\textrm{I}_{n}$ and note that the cocycle $c_{n} \coloneqq c|_{\mathcal{R}_{n}}$ is a coboundary by \cite[Proposition 7.4]{MR578656}. Thus there exists for each $n$, a Borel function $\delta_{n}: X \rightarrow \mathbb{T}$ such that $c_{n}(x,y) = \delta_{n}(x)\delta_{n}(y)^{-1}$ for a.e. $(x,y) \in \mathcal{R}_{n}$. Clearly, we can extend $c_{n}$ to all of $\mathcal{R}$ and define a cocycle on $\mathcal{S}$ that is a coboundary. Now we can check that $c_{n} \rightarrow c$ in $Z^{1}(\mathcal{R},\mathbb{T})$, i.e., $c$ is approximately a coboundary.

Since the topologies coincide, we have that the automorphisms in $\Aut(B,A)$ corresponding to $(c_{n},\phi_{n})$ are inner and converge to $\theta$ in the usual topology in $\Aut(B)$. Denoting by $w_{n} \in \mathcal{N}_{B}(A)$ the unitaries such that $\Ad(w_{n})$ corresponds to $(c_{n},\phi_{n})$, we note that $\Ad(w_{n}) \in \Aut(A \subset B)$ and $\Ad (w_{n}) \rightarrow \theta$. We fix a free ultrafilter $\omega$ on $\mathbb{N}$ and consider now the sequence $W \coloneqq (w_{n})_{n}$ which lies in $B^{\omega}$ by the same argument as in Lemma \ref{CocycleSelfConjugacyNonCartan}. Since $\Ad(w_{n}) \circ \alpha_{g} \circ \Ad(w_{n}^{*}) \rightarrow \Ad(v_{g}) \circ \alpha_{g}$ in $\Aut(A \subset B)$, we can check once again that letting $v_{g,n} = w_{n}v_{g}\alpha_{g}(w_{n}^{*})$, the sequence $(v_{g,n})_{n \in \mathbb{N}}$ is a $\omega$-centralizing sequence. Clearly for all $n \in \mathbb{N}$ and $g \in G$, $v_{g,n} \in \mathcal{N}_{B}(A)$.

Now we consider the action $\alpha_{\omega}$ of $G$ on $A^{\omega} \cap B_{\omega}$ and look at the corresponding $\mu_{\omega}$ preserving action on $\mathcal{B}_{\omega}$. By \cite[Lemma 2.3]{MR0444900} (see the previous discussion), this action is free. We claim now that there is a separable von Neumann subalgebra $A_{0} \subset A^{\omega} \cap B_{\omega}$ such that $\alpha_{\omega}$ restricts to a free measure preserving action of $G$ on $A_{0}$. To prove the claim, we first note that since $\alpha_{\omega}$ is free, for a fixed $g \in G$, there exists a set $I$ and orthogonal projections $p_{g,i} \in A^{\omega} \cap B_{\omega}$ such that for each $i \in I$, $\alpha_{\omega, g}(p_{g,i}) \perp p_{g,i}$ and $\vee_{i \in I}p_{g,i} = 1$. Since $A^{\omega} \cap B_{\omega}$ is a countably decomposable von Neumann algebra, we can thus find projections $p_{g,n}$ such that $\alpha_{\omega,g}(p_{g,n}) \perp p_{g,n}$ for all $n \in \mathbb{N}$ and $\vee_{n}p_{g,n} = 1$. Now we can define $A_{0}$ to be the von Neumann algebra generated by the countably many elements $\{\alpha_{\omega,g}(p_{h,n}) \; | \; g,h \in G\}$. This is clearly separable and by construction, the action $\alpha_{\omega}$ on $A_{0}$ remains free.  

Let the underlying measure space be $(Y,\nu)$ and let $K$ be a finite subset of $G$ and $\epsilon > 0$. We shall now use a version of the Ornstein-Weiss Rohlin Theorem (\cite{MR910005}) presented in \cite[Theorem 4.46]{MR3616077}. For some $m \in \mathbb{N}$, there exists finite subsets $T_{1},...,T_{m}$ of $G$ and projections $P_{1},...,P_{m}$ in $A_{0}$ such that: 
\begin{itemize}
    \item The projections $\{\alpha_{\omega,h}(P_{i}) \; | \; 1 \leq i \leq m, h \in T_{i}\}$ are orthogonal and their sum $P = \sum_{i=1}^{m}\sum_{s \in T_{i}}\alpha_{\omega,s}(P_{i})$ has trace at least $1 - \epsilon$.
    \item For every $g \in K$ and $1 \leq i \leq m$, we have that $|gT_{i} \cap T_{i}| \geq (1 - \epsilon)|T_{i}|$.
\end{itemize}

Let $(p_{i,k})_{k \in \mathbb{N}} \in A$ be a representative sequence for the element $P_{i} \in A_{0}$. Since $v_{g,n}$ is $\omega$-centralizing, for each $k$, we can find $n_{k} \in \mathbb{N}$ such that for all $1 \leq i\leq m$ and for all $g \in T_{i}$ we have that
\begin{equation}
   \| v_{g,n_{k}}\alpha_{g}(p_{i,k}) - \alpha_{g}(p_{i,k})v_{g,n_{k}} \|_{\phi}^{\#} \leq \frac{1}{k} \nonumber
\end{equation}

Now for all $g \in G$, we define $V_{g} \coloneqq (v_{g,n_{k}})_{k\in \mathbb{N}}$ and note that $V_{g}$ is still a $\omega$-centralizing sequence, i.e., $V_{g} \in B_{\omega}$. From the above equation we have that for each $1 \leq i \leq m$, $V_{g}$ commutes in $B_{\omega}$ with $\alpha_{g}(P_{i})$ whenever $g \in T_{i}$. It is easy to check, as in Lemma \ref{CocycleSelfConjugacyNonCartan} that $V_{g}$ is a 1-cocycle for the action of $G$ on $B_{\omega}$. Now we consider the element $V \in B_{\omega}$ given by $V = (1- P) + \sum_{i = 1}^{m}\sum_{g \in T_{i}} V_{g}\alpha_{g}(P_{i})$. The summands in $V$ have orthogonal domains and ranges, and hence $V$ is a unitary operator. Moreover $V$ clearly normalizes the subalgebra $A_{0}$. Now we fix $g \in K$ and look at the element $Q = \sum_{i=1}^{m}\sum_{s \in T_{i} \cap gT_{i}}\alpha_{s}(P_{i})$. We have that: 
\begin{equation}
    \begin{split}
        V_{g}\alpha_{\omega,g}(V)Q &= \sum_{i=1}^{m}\sum_{s \in T_{i} \cap gT_{i}}V_{g}\alpha_{\omega,g}(V)\alpha_{\omega,s}(P_{i}) = \sum_{i=1}^{m}\sum_{s \in T_{i} \cap gT_{i}}V_{g}\alpha_{\omega,g}(V\alpha_{\omega,g^{-1}s}(P_{i})) \\
        &= \sum_{i,j=1}^{m}\sum_{s,t \in T_{i} \cap gT_{i}} V_{g}\alpha_{\omega,g}(V_{t}\alpha_{\omega,t}(P_{j})\alpha_{\omega,g^{-1}s}(P_{i})) \\
        &= \sum_{i=1}^{m}\sum_{s \in T_{i} \cap gT_{i}}V_{g}\alpha_{\omega,g}(V_{g^{-1}s})\alpha_{\omega,s}(P_{i}) \\
        &= \sum_{i=1}^{m}\sum_{s \in T_{i} \cap gT_{i}}V_{s}\alpha_{\omega,s}(P_{i}) = VQ \nonumber
    \end{split}
\end{equation}
We note here that by our choice of $V_{g}$, it is clear that $Q$ commutes with $V$, and we similarly get $QV = QV_{g}\alpha_{g}(V)$. As in \cite[Theorem 4.7]{MR4124434}, we have that: 
\begin{equation}
     \| V_{g}\alpha_{\omega,g}(V) - V \|_{2}^{2} = \| (V_{g}\alpha_{\omega,g}(V) - V)(1 - Q) \|_{2}^{2} \leq 4\tau(1-Q) \nonumber
\end{equation}
where $\tau$ denotes the trace in $\mathcal{B}_{\omega}$. Now we have
\begin{equation}
    \begin{split}
        \tau(1-Q) &= \tau(1-P) + \sum_{i=1}^{n}\sum_{s \in T_{i} \backslash gT_{i}} \alpha_{\omega,s}(P_{i}) \\
        &= \tau(1-P) + \sum_{i=1}^{n}|T_{i} \backslash gT_{i}|\tau(P_{i}) \\
        & \leq \epsilon + \epsilon\sum_{i=1}^{n}|T_{i}|\tau(P_{i}) \leq 2\epsilon \nonumber
    \end{split}
\end{equation}
Thus we have that $\|V_{g} - V\alpha_{\omega,g}(V^{*})\|_{2}^{2} \leq 8\epsilon$ and that the 1-cocycle $V_{g}$ for the action $\alpha_{\omega}$ of $G$ on $B_{\omega}$ approximately vanishes. Now applying Ocneanu's index selection trick (\cite[Lemma 5.5]{MR807949}), we can find a unitary $Z \in \mathcal{U}(B_{\omega})$ which normalizes $A^{\omega} \cap B_{\omega}$ and such that $V_{g} = Z\alpha_{\omega,g}(Z^{*})$. Now we consider the unitary $W \in B^{\omega}$ given by $W \coloneqq (w_{n_{k}})_{k \in \mathbb{N}}$ and we choose a representative sequence $(z_{k})_{k \in \mathbb{N}}$ for the unitary $Z \in \mathcal{U}(B_{\omega})$ such that $z_{k} \in \mathcal{N}_{B}(A)$ for all $k \in \mathbb{N}$. We define $u_{k} \coloneqq w_{n_{k}}z_{k}$ and note that $\Ad(u_{k}) \rightarrow \theta$ and $u_{k}v_{g}\alpha_{g}(u_{k}^{*}) \rightarrow 1$ as $k \rightarrow \omega$. As in Lemma \ref{CocycleSelfConjugacyNonCartan}, we note that both the topologies are metrizable and hence we can choose a subsequence such that the conclusion holds as $k \rightarrow \infty$, as required.     
\end{proof}

\begin{theorem}
\label{CocycleCOnjugacyForActionsOnEquivalenceRelations}
Let $(Y, \nu)$ be a standard probability space and $(B_{y})_{y \in Y}$ be a measurable field of injective factors, with Cartan subalgebras $A_{y} \subset B_{y}$. Let $\alpha$ and $\beta$ be free actions of a discrete measured amenable groupoid $\mathcal{G}$ with unit space $Y$ on the field $(A_{y} \subset B_{y})_{y \in Y}$. Then the following holds:
 \begin{enumerate}[(i)]
     \item $\alpha$ and $\beta$ are cocycle conjugate if and only if $ \modnew(\alpha) \sim  \modnew(\beta)$.
    \item $\alpha$ and $\beta$ are strongly cocycle conjugate if and only if $ \modnew(\alpha) =  \modnew(\beta)$. 
\end{enumerate}
\end{theorem}

\begin{proof}
As in Theorem \ref{GroupoidActionsClassification}, the `if' directions are the nontrivial ones. By the same argument as in Theorem \ref{GroupoidActionsClassification} using surjectivity of the Connes-Takesaki module map at the measure space level (c.f. \cite[Theorem 2.2]{MR1643188}, we can deduce (i) from (ii), so we prove (ii) here. Let $A$ and $B$ denote the direct integrals of the measurable fields $(A_{y})_{y \in Y}$ and $(B_{y})_{y \in Y}$ respectively. Since $\alpha$ and $\beta$ preserve globally the Cartan subalgebras, it can be checked as in the discussion preceding Theorem \ref{CohomologyVanishingGroupoidsCartanAlternate}, that $\mathcal{Z}(\widetilde{B}) = (\widetilde{B})'\cap \widetilde{M}_{\alpha}$ and $\mathcal{Z}(\widetilde{B}) = (\widetilde{B})'\cap \widetilde{M}_{\beta}$  where $M_{\alpha}$ and $M_{\beta}$ denotes the respective crossed products. Now by Lemma \ref{CentrallyFreeCondition}, $\alpha$ and $\beta$ are centrally free. With the same notation as Theorem \ref{GroupoidActionsClassification}, we let $\mathcal{R}$ be the corresponding amenable nonsingular equivalence relation on $Y$ and we let $\delta_{(y,z)}: \Gamma_{z} \rightarrow \Gamma_{y}$ be the measurable family of group isomorphisms for all $(y,z) \in \mathcal{R}$. We let $P_{y}$ be the Polish space of all strong cocycle conjugacies between the actions $(\alpha_{g})_{g \in \Gamma_{y}}$ and $(\beta_{g})_{g \in \Gamma_{y}}$ on the inclusion $A_{y} \subset B_{y}$. More precisely we define $P_{y} = \{(\theta,c_{g})_{g \in \Gamma_{y}} \; | \; \theta \in \ker(\modnew) \subset \Aut(B_{y}), \; \theta(A_{y}) = A_{y},  \; c_{g} \in \mathcal{N}_{B_{y}}(A_{y}), \; c_{gh} =c_{g}\alpha_{g}(c_{h}), \; \theta\beta_{g}\theta^{-1} = \Ad(c_{g})\alpha_{g} \; \forall \; g,h \in \Gamma_{y}\}$. We endorse $P_{y}$ with a topology in the same way as in Theorem \ref{GroupoidActionsClassification}. By Theorem \ref{conjugaceyOfGroupActionsOnEquivalenceRelation}, $P_{y}$ is nonempty for all $y \in Y$. Letting $G_{y} = \mathcal{N}_{B_{y}}(A_{y})$, we have an action $G_{y} \curvearrowright P_{y}$, as in Theorem \ref{GroupoidActionsClassification}. 
By the same argument as in Lemma \ref{CocycleSelfConjugacyNonCartan}, we can conclude that this action has dense orbits if every cocycle self conjugacy is approximately inner, which is guaranteed by Lemma \ref{approximateInnernessOfSelfConjugacyCartan}. 

We now consider the measurable field of isomorphisms $\mathcal{R} \ni (y,z) \mapsto \gamma_{(y,z)}: P_{z} \rightarrow P_{y}$ given by mapping $(\theta, (c_{g})_{g \in \Gamma_{z}}) \in P_{z}$ to the automorphism $\alpha_{(z,y)} \circ \theta \circ \beta_{(y,z)}$ of $B_{z}$ and the $1$-cocycle $g \mapsto \alpha_{(z,y)}(c_{\delta(y,z)}(g))$. This is well defined because $ \modnew(\alpha_{(z,y)} \circ \theta \circ \beta_{(y,z)}) = 1$ as $ \modnew(\alpha_{g}) = \ \modnew(\beta_{g})$ for all $g \in \mathcal{G}$.  We use once again \cite[Theorem 3.5]{MR4124434} to find a measurable field  $y \mapsto \pi_{y} = (\theta_{y},(c_{g})_{g \in \Gamma_{y}})$ of cocycle conjugacies together with a 1-cocycle $\mathcal{R} \ni (y,z) \mapsto c_{(y,z)} \in G_{y}$ with $\pi_{y} = c_{(y,z)} \in G_{y}$ such that $\pi_{y} = c_{(y,z)} \cdot \gamma_{(y,z)}(\pi_{z})$. Hence the field $(c_{g,y})_{g \in \Gamma_{y}}$ and $(c_{(y,z)})_{(y,z)\in \mathcal{R}}$ forms a 1-cocycle for the groupoid action $\alpha$ as shown in the proof of Theorem \ref{GroupoidActionsClassification}, and along with the field of isomorphisms $(\theta_{y})_{y\in Y}$, we have a strong cocycle conjugacy, as required.
\end{proof}

We have now all the ingredients to prove the main classification result for regular subalgebras with Cartan inclusions. 

\begin{theorem}
\label{classification of inclusionsCartan}
Let $A_{i} \subset B_{i} \subset M_{i}$, $i \in \{1,2\}$ be two triples of von Neumann algebras with separable preduals such that $M_{i}$'s are injective factors, $B_{i}\subset M_{i}$ are regular and $A_{i} \subset M_{i}$ are Cartan subalgebras. Let $\mathcal{G}_{i}$ be the associated amenable discrete measured groupoids and $(\alpha_{i},u_{i})$ be the associated free cocycle action of $\mathcal{G}_{i}$ on $A_{i} \subset B_{i}$. Then there exists an isomorphism $\theta: M_{1} \rightarrow M_{2}$ satisfying $\theta(B_{1}) = B_{2}$ and $\theta(A_{1}) = A_{2}$ if and only if $\type(B_{1}) = \type(B_{2})$ and there is a nonsingular isomorphism $\sigma: \mathcal{G}_{1} \rightarrow \mathcal{G}_{2}$ with $\sigma(X_{1}) = X_{2}$ such that $ \modnew(\alpha_{1}) \sim_{\sigma} \modnew(\alpha_{2})$. \end{theorem}

\begin{proof}
Once again, we only show the isomorphism of inclusions if the conditions are satisfied. By Theorem B in \cite{MR4124434}, the only cases to consider are type $\textrm{I}_{\infty}$,  $\textrm{II}_{\infty}$ and type $\textrm{III}$. By Lemma \ref{CohomologyVanishingGroupoidsCartanAlternate}, we can assume that the 2-cocycles $u_{1}$ and $u_{2}$ are trivial. The genuine free actions are now cocycle conjugate by Theorem \ref{CocycleCOnjugacyForActionsOnEquivalenceRelations} and consecutively the crossed product inclusions are isomorphic.     
\end{proof}

\section{Model actions}
\label{sectionmodelactions}

In \cite{vaes_verjans_2022}, the authors provide a canonical construction of ergodic equivalence relations with prescribed flows. More precisely, given any ergodic flow $F: \mathbb{R} \curvearrowright (Y, \nu)$, they construct an adjoint flow $\hat{F}$ and a canonical ergodic equivalence relation $\mathcal{R}_{F}$ such that the associated flow of $\mathcal{R}_{F}$ is isomorphic to $\hat{F}$. The construction of the adjoint flow is involutive, i.e., the adjoint of the flow $\hat{F}$ is isomorphic to $F$. In this section, we exploit this construction to provide model actions on fields of ergodic equivalence relations with prescribed actions on the associated fields of flows.  

We recall the notion of adjoint flows introduced in \cite{vaes_verjans_2022}. For an ergodic flow $F: \mathbb{R} \curvearrowright (Y, \nu)$, by \cite[Proposition 3.1]{vaes_verjans_2022}, there is a unique (up to measure-preserving isomorphism) non-singular ergodic action of $\mathbb{R}^{2}$ on a standard $\sigma$-finite measure space $(Z, \zeta)$ such that the action of both $\mathbb{R} \times \{0\}$ and $\{0\} \times \mathbb{R}$ scale the measure $\zeta$ and such that $\mathbb{R} \curvearrowright Z /(\{0\} \times \mathbb{R})$ is isomorphic to $\mathbb{R} \curvearrowright Y$. This unique $\mathbb{R}^{2}$ action can be realized concretely by considering $(Z, \zeta) = (Y \times \mathbb{R}, \nu \times \lambda)$ where $d\lambda = e^{-t}dt$ and the action is given by
\begin{equation}
    (t,r) \cdot (y,s) = (t \cdot y,\omega(t,y)+t+r+s) \nonumber
\end{equation}
where $\omega: \mathbb{R} \times Y \rightarrow \mathbb{R}$ is the logarithm of the Radon-Nikodym cocycle. The ergodic flow $\mathbb{R} \curvearrowright Z /(\mathbb{R} \times \{0\})$ is called the adjoint flow of $F$ and denoted by $\hat{F}: \mathbb{R} \curvearrowright (\hat{Y},\hat{\nu})$. In the concrete realization of the $\mathbb{R}^{2}$ action as above, we note that the adjoint flow can be described as the action of $(\{0\} \times \mathbb{R})$ on $(Y \times \mathbb{R})/ \mathbb{R}$ where $\mathbb{R}$ acts by translation on the second variable.

Now let $\mathcal{S}$ be the unique countable ergodic hyperfinite equivalence relation of type $\textrm{III}_{1}$ on a standard probability space $(X, \mu)$ with the Radon-Nikodym one cocycle $\Omega: \mathcal{S} \rightarrow \mathbb{R}$. We consider the equivalence relation $\mathcal{R}_{F}$ on $(X \times Y)$ given by $(x,y) \sim (x',y')$ if and only if $(x,x') \in \mathcal{S}$ and $y' = \Omega(x,x') \cdot y$. By \cite[Proposition 3.4]{vaes_verjans_2022}, we have that $\mathcal{R}_{F}$ is ergodic, hyperfinite and has associated flow $\hat{F}$.

Let us consider two ergodic flows $F_{1}:\mathbb{R}  \curvearrowright (Y_{1}, \nu_{1})$ and $F_{2}:\mathbb{R} \curvearrowright(Y_{2},\nu_{2})$ and let $\psi$ be a $\mathbb{R}$-equivariant nonsingular isomorphism between them. Let $\omega_{i}$ denote the corresponding Radon-Nikodym cocycle. Now we consider the map: 
\begin{equation}
    \hat{\psi}: (\hat{Y}_{1}, \hat{\nu}_{1}) \rightarrow (\hat{Y}_{2}, \hat{\nu}_{2}), \; \; \hat{\psi}(y,s) = (\psi(y), \delta(y) + s) \nonumber
\end{equation}
where $\delta = \log \frac{d(\psi^{-1}_{*})\nu_{2}}{d\nu_{1}}$. Note that here we use the concrete realization of the adjoint flow to define the map. We now calculate: 
\begin{equation}
    \begin{split}
        (t,0) \cdot (\hat{\psi}(y,s)) &= (t,0) \cdot (\psi(y),\delta(y) +s) = (\alpha_{t}(\psi(y)), \omega_{2}(t,\psi(y)) + t + \delta(y) + s) \\
        &= (\psi(\alpha_{t}(y)), \delta(\alpha_{t}(y)) + \omega_{1}(t,y) + t + s) = \hat{\psi}(\alpha_{t}(y), \omega_{1}(t,y) + t + s) \\
        &= \hat{\psi}((t,0) \cdot(y,s)) \nonumber
    \end{split}
\end{equation}
where the third equality follows from the fact that $\psi$ is $\mathbb{R}$-equivariant. Thus $\hat{\psi}$ commutes with the $(\mathbb{R} \times \{0\})$ action and hence is well defined. Since the flow is given by translation, it is now easy to see that $\hat{\psi}$ is a $\mathbb{R}$-equivariant nonsingular isomorphism between the adjoint flows. Now we define the nonsingular isomorphism
\begin{equation}
    \Psi:(X \times Y_{1}, \mu \times \nu_{1}) \rightarrow (X \times Y_{2}, \mu \times \nu_{2}), \; \Psi(x,y) = (x, \psi(y)) \nonumber
\end{equation}
We consider the Maharam extension $c(\mathcal{R}_{i})$ of $\mathcal{R}_{i}$ together with the measure scaling action of $\mathbb{R}$. We have then an isomorphism $\tilde{\Psi}$ between $c(\mathcal{R}_{1})$ and $c(\mathcal{R}_{2})$ given by: 
\begin{equation}
    \tilde{\Psi}(x,y,r) = (x,\psi(y), r + \delta(y)) \nonumber
\end{equation}

Note that the space of $\mathcal{R}_{F_{i}}$-invariant functions $L^{\infty}(X \times Y_{i} \times \mathbb{R})^{\mathcal{R}_{F_{i}}}$ and $L^{\infty}(\hat{Y}_{i})$ are isomorphic. It can be checked that this isomorphism is induced by the set map $\pi_{i}: X \times Y_{i} \times \mathbb{R}_{i} \rightarrow \hat{Y}_{i}$ given by: 
\begin{equation}
    \pi_{i}(x,y,t) = [(y,t)] \nonumber
\end{equation}
Now we have that: 
\begin{equation}
    \begin{split}
        \modnew(\Psi) ([y,t]) &= \modnew(\Psi)(\pi_{1}(x,y,t)) = \delta(y) \cdot \pi_{2}(x, \psi(y), t) \\
        &= \delta(y) \cdot [(\psi(y),t)] = [(\psi(y), \delta(y) + t)] \\ &= \hat{\psi}([(y,t)])
    \end{split} \nonumber
\end{equation}
Thus we have that $\modnew(\Psi) = \hat{\psi}$. We summarize the above discussion in the following proposition. 

\begin{proposition}
\label{modelAutomorphism}
    For $i \in \{1,2\}$, let $F_{i}: \mathbb{R} \curvearrowright (Y_{i},\nu_{i})$ be two ergodic flows and $\psi$ be a nonsingular $\mathbb{R}$-equivariant isomorphism between them. Then $\psi$ induces an isomorphism $\hat{\psi}$ between the adjoint flows $\hat{F}_{i}$ and a nonsingular isomorphism $\Psi: \mathcal{R}_{F_{1}} \rightarrow \mathcal{R}_{F_{2}}$ such that $\modnew(\Psi) = \hat{\psi}$. 
\end{proposition}

Since the construction of the adjoint flow is involutive, i.e, $\hat{\hat{F}}$ is isomorphic to $F$, Proposition \ref{modelAutomorphism} gives a way to construct an isomorphism between ergodic equivalence relations inducing a given isomorphism between flows.  Now, from this construction, we have almost immediately the following theorem providing a model action of a groupoid realizing the invariant. 

\begin{theorem}
\label{modelAction}
    Let $(Z,\Omega)$ be a standard probability space and $\mathcal{G}$ be a discrete measured groupoid with $\mathcal{G}^{(0)} = Z$. Let $(F_{z}:\mathbb{R}\curvearrowright (Y_{z},\nu_{z}))_{z \in Z}$ be a measurable field of ergodic flows. Let $\mathcal{\psi}$ be an action of the groupoid $\mathcal{G}$ on the field of flows $(F_{z})_{z \in Z}$, i.e., for each $g \in \mathcal{G}$ there is a nonsingular $\mathbb{R}$-equivariant isomorphism $\psi_{g}: F_{s(g)} \rightarrow F_{t(g)}$ such that $\psi_{h} \circ \psi_{g} = \psi_{hg}$ for a.e. $(g,h) \in \mathcal{G}^{(2)}$ . Then $\psi$ canonically lifts to an outer action $\Psi$ of $\mathcal{G}$ on the field of ergodic equivalence relations $(\mathcal{\mathcal{R}}_{F_{z}})_{z \in Z}$ such that $\modnew(\Psi) \sim \hat{\psi}$.   
\end{theorem}
\begin{proof}
    Let $\mathcal{S}_{1}$ on $(X_{1},\mu_{1})$ and $\mathcal{S}_{0}$ on $(X_{0},\mu_{0})$ be the unique countable ergodic hyperfinite equivalence relations of type $\textrm{III}_{1}$ and $\textrm{II}_{1}$ respectively and let $\mathcal{S} = \mathcal{S}_{1} \times \mathcal{S}_{0}$ be the direct product on a standard probability space $(X, \mu) = (X_{1} \times X_{0}, \mu_{1} \times \mu_{0})$. It is easy to see that $\mathcal{S}$ is still ergodic, hyperfinite and of type $\textrm{III}_{1}$. Recall that we define $\mathcal{R}_{F_{z}}$ on the space $(X \times Y_{z})$ as $(x,y) \sim (x',y')$ if and only if $(x,x') \in \mathcal{S}$ and $y'= \omega(x,x')\cdot y$. Clearly $\mathcal{R}_{F_{z}}$ is a direct product of $\mathcal{S}_{0}$ and $\mathcal{R}^{1}_{F_{z}}$ where $\mathcal{R}^{1}_{F_{z}}$ is defined on $X_{1} \times Y_{z}$ in the same way as above.   
    
    Claim 1: \textit{Every discrete measured groupoid $\mathcal{G}$ admits an outer action on the unique ergodic hyperfinite $\textrm{II}_{1}$ equivalence relation}. Assume for the moment that Claim 1 is true, and let $\Phi$ be such an outer action on $\mathcal{S}_{0}$. We define the isomorphisms $\Psi^{1}_{g}: \mathcal{R}^{1}_{F_{s(g)}} \rightarrow \mathcal{R}^{1 }_{F_{t(g)}}$ such that $\modnew(\Psi^{1}_{g}) = \hat{\psi}_{g}$ as in Proposition \ref{modelAutomorphism}. We now consider the diagonal action $\Psi_{g} = \Phi_{g} \times \Psi^{1}_{g}$ of $\mathcal{G}$ on the field $(\mathcal{R}_{F_{z}})_{z \in Z}$. Clearly, $\modnew(\Psi) = \hat{\psi}$ and the only thing to check is that this is indeed an action. This is clear as for $y \in Y_{s(g)}$, we have that 
    \begin{equation}
        \Psi^{1}_{h} \circ \Psi^{1}_{g}(x,y) = \Psi^{1}_{h}(x,\psi_{g}(y)) = (x, \psi_{h} \circ \psi_{g}(y)) = (x, \psi_{hg}(x)) = \Psi^{1}_{hg}(x,y) \nonumber
    \end{equation}
Since $\Phi$ is outer, $\Psi$ is a well-defined outer action on the field of equivalence relations $\mathcal{R}_{F_{z}}$ satisfying the conditions of the theorem. 

\textit{Proof of Claim 1:} Let $\mu_{0}$ be a probability measure on $[0,1]$ and let $X$ denote the unit space of $\mathcal{G}$. Let $G_{x}$ denote the set $\{g \in \mathcal{G} \; | \; t(g) = x\}$. Let $(K_{x}, \mu_{x})$ denote the standard probability space $([0,1], \mu_{0})^{G_{x}}$ for all $x \in X$. Consider a fixed free weakly mixing $\mathbb{Z}$-action on $([0,1], \mu_{0})$ and let the orbit equivalence relation of the diagonal $\mathbb{Z}$ action on $(K_{x}, \mu_{x})$ be $\mathcal{R}_{x}$. By definition $\mathcal{R}_{x}$ is p.m.p., hyperfinite and since $\mathbb{Z} \curvearrowright [0,1]$ is weakly mixing, $\mathcal{R}_{x}$ is ergodic. We construct an action $\mathcal{G} \curvearrowright (\mathcal{R}_{x})_{x \in X}$ as follows: for $g \in \mathcal{G}$, we define $\alpha_{g}: K_{s(g)} \rightarrow K_{t(g)}$ by $\alpha_{g}(b)_{h} = b_{g^{-1}h}$ for all $b \in K_{s(g)}$. For a.e. $x \in X$, any $g \in \Gamma_{x}$ 
and a fixed $n \in \mathbb{Z}$, the set $\{k \in K_{x} \; | \; \alpha_{g}(k)  = n \cdot k \}$ has measure zero, and this can be proved in the same way as in \cite[Lemma 2.2]{MR4243020}. Taking intersections, it is clear then that for a.e. $x \in X$ and any $g \in \Gamma_{x}$, $\alpha_{g}$ is outer as an automorphism on $\mathcal{R}_{x}$, and hence the action $\mathcal{G}\curvearrowright (\mathcal{R}_{x})_{x \in X}$ is outer. Now we can fix a copy of the unique hyperfinite ergodic type $\textrm{II}_{1}$ equivalence relation $\mathcal{S}_{0}$ and let $\phi_{x}: \mathcal{R}_{x} \rightarrow \mathcal{S}_{0}$ be a field of isomorphisms. Defining $\beta_{g}$ to be $\phi_{t(g)} \circ \alpha_{g} \circ \phi_{s(g)}^{-1}$, we have that $\beta$ is the required outer action of $\mathcal{G}$ on $\mathcal{S}_{0}$.  
\end{proof}

\begin{proof}[Proof of Theorem C]
    We consider the field of adjoint flows $(\hat{F}_{x})_{x \in X}$ and the induced action $\hat{\psi}$ of $\mathcal{G}$ on $(\hat{F}_{x})$. By Theorem \ref{modelAction}, $\hat{\psi}$ lifts to an outer action $\hat{\Psi}$ of $\mathcal{G}$ on a field of ergodic hyperfinite equivalence relations $\mathcal{R}_{F_{x}}$ such that $\modnew(\hat{\Psi}) = \hat{\hat{\psi}} = \psi$. Since outside a null set, $F_{x}$ is not the translation action $\mathbb{R} \curvearrowright \mathbb{R}$, we have that $R_{F_{x}}$ is of type $\textrm{III}$ almost everywhere. We now consider the corresponding injective factors $B_{x} = L(\mathcal{R}_{F_{x}})$ with the obvious Cartan subalgebras $A_{x}$. Let $M = (B_{x})_{x \in X} \rtimes_{\hat{\Psi}} \mathcal{G}$ be the crossed product, and $B$ and $A$ be the direct integrals of the fields $(B_{x})_{x \in X}$ and $(A_{x})_{x \in X}$ respectively. Since $\mathcal{G}$ is ergodic and $\hat{\Psi}$ is an outer action, we have that by construction $M$ is a factor, $B \subset M$ is a regular subalgebra, and $A \subset M$ is a Cartan subalgebra. Since such inclusions are classified by the associated discrete measured groupoid and the associated actions on the field of flows as in Theorem B, the desired conclusion follows.     
\end{proof}

We conclude this section with the following corollary to Theorem \ref{modelAction}. The tracial version of this was proved in \cite[Corollary 3.4]{MR4124434}.  

\begin{corollary}
\label{CartanIffRelativeCommutant}
Let $M$ be an injective factor and $B$ be a regular subalgebra of $M$ of type $\textrm{III}$ with a conditional expectation $E: M \rightarrow B$. Then $\mathcal{Z}(\widetilde{B}) = \widetilde{B}' \cap \widetilde{M}$ if and only if there exists a Cartan subalgebra $A$ of $M$ such that $A \subset B \subset M$.      
\end{corollary}

\begin{proof}
If $A \subset M$ is a Cartan subalgebra such that $A \subset B$, then by the discussion preceding Theorem \ref{CohomologyVanishingGroupoidsCartanAlternate}, it is clear that $\mathcal{Z}(\widetilde{B}) = \widetilde{B}' \cap \widetilde{M}$. Conversely, suppose the relative commutant condition is satisfied. We can write $M$ as a crossed product $B \rtimes_{\alpha} \mathcal{G}$ of a free action of a discrete measured amenable groupoid $\mathcal{G}$ on $B$ by Theorem \ref{2CohomologyVanishingGroupoidsAlternate}. Let $\mathcal{G}^{(0)} = X$ and choose a field of ergodic equivalence relations $\mathcal{R}_{x}$ and an action $\beta$ of $\mathcal{G}$ on $(\mathcal{R}_{x})_{x \in X}$ such that $\modnew(\beta) = \modnew(\alpha)$ by Theorem \ref{modelAction}. Let $(C_{x} \subset D_{x})_{x \in X}$ denote the corresponding field of Cartan inclusions with each $D_{x}$ being an injective factor. Let $C$ and $D$ denote the direct integrals of $(C_{x})_{x \in X}$ and $(D_{x})_{x \in X}$ respectively and let $N$ denote the corresponding crossed product $(D_{x})_{x \in X} \rtimes_{\beta} \mathcal{G}$. By Theorem \ref{GroupoidActionsClassification}, there exists a *-isomorphism $\theta: N \rightarrow M$ satisfying $\theta(D) = B$. It can be easily checked now that $A = \theta(C)$ is a Cartan subalgebra of $M$ and is obviously contained in $B$. 
\end{proof}

\section*{Acknowledgments}
The author is supported by FWO research project G090420N of the Research Foundation Flanders. The author would like to thank his supervisor Stefaan Vaes for his immense support throughout the process of writing this article. The author would also like to thank the reviewer for their insightful comments on the article.  

\printbibliography

@article {MR4124434,
    AUTHOR = {Popa, Sorin and Shlyakhtenko, Dimitri and Vaes, Stefaan},
     TITLE = {Classification of regular subalgebras of the hyperfinite {$\rm
              II_1$} factor},
   JOURNAL = {J. Math. Pures Appl. (9)},
  FJOURNAL = {Journal de Math\'{e}matiques Pures et Appliqu\'{e}es. Neuvi\`eme S\'{e}rie},
    VOLUME = {140},
      YEAR = {2020},
     PAGES = {280--308},
      ISSN = {0021-7824},
   MRCLASS = {46L36 (20L05 37A20 46L10 46L55)},
  MRNUMBER = {4124434},
MRREVIEWER = {Qihui Li},
       DOI = {10.1016/j.matpur.2020.02.009},
       URL = {https://doi.org/10.1016/j.matpur.2020.02.009},
}

@article {MR454659,
    AUTHOR = {Connes, Alain},
     TITLE = {Classification of injective factors. {C}ases {$\rm{II}_{1},$}
              {$\rm{II}_{\infty },$} {$\rm{III}_{\lambda },$} {$\lambda \not=1$}},
   JOURNAL = {Ann. of Math. (2)},
  FJOURNAL = {Annals of Mathematics. Second Series},
    VOLUME = {104},
      YEAR = {1976},
    NUMBER = {1},
     PAGES = {73--115},
      ISSN = {0003-486X},
   MRCLASS = {46L10},
  MRNUMBER = {454659},
MRREVIEWER = {Fran\c{c}ois Combes},
       DOI = {10.2307/1971057},
       URL = {https://doi.org/10.2307/1971057},
}

@article {MR9096,
    AUTHOR = {Murray, Francis J. and von Neumann, John},
     TITLE = {On rings of operators. {IV}},
   JOURNAL = {Ann. of Math. (2)},
  FJOURNAL = {Annals of Mathematics. Second Series},
    VOLUME = {44},
      YEAR = {1943},
     PAGES = {716--808},
      ISSN = {0003-486X},
   MRCLASS = {46.0X},
  MRNUMBER = {9096},
MRREVIEWER = {E. R. Lorch},
       DOI = {10.2307/1969107},
       URL = {https://doi.org/10.2307/1969107},
}

@article {MR218905,
    AUTHOR = {Powers, Robert T.},
     TITLE = {Representations of uniformly hyperfinite algebras and their
              associated von {N}eumann rings},
   JOURNAL = {Ann. of Math. (2)},
  FJOURNAL = {Annals of Mathematics. Second Series},
    VOLUME = {86},
      YEAR = {1967},
     PAGES = {138--171},
      ISSN = {0003-486X},
   MRCLASS = {46.65},
  MRNUMBER = {218905},
MRREVIEWER = {D. M. Topping},
       DOI = {10.2307/1970364},
       URL = {https://doi.org/10.2307/1970364},
}

@article {MR415341,
    AUTHOR = {Krieger, Wolfgang},
     TITLE = {On ergodic flows and the isomorphism of factors},
   JOURNAL = {Math. Ann.},
  FJOURNAL = {Mathematische Annalen},
    VOLUME = {223},
      YEAR = {1976},
    NUMBER = {1},
     PAGES = {19--70},
      ISSN = {0025-5831},
   MRCLASS = {46L10 (28A65)},
  MRNUMBER = {415341},
MRREVIEWER = {M. Takesaki},
       DOI = {10.1007/BF01360278},
       URL = {https://doi.org/10.1007/BF01360278},
}

@article {MR880070,
    AUTHOR = {Haagerup, Uffe},
     TITLE = {Connes' bicentralizer problem and uniqueness of the injective
              factor of type {${\rm III}_1$}},
   JOURNAL = {Acta Math.},
  FJOURNAL = {Acta Mathematica},
    VOLUME = {158},
      YEAR = {1987},
    NUMBER = {1-2},
     PAGES = {95--148},
      ISSN = {0001-5962},
   MRCLASS = {46L35},
  MRNUMBER = {880070},
MRREVIEWER = {Steve Wright},
       DOI = {10.1007/BF02392257},
       URL = {https://doi.org/10.1007/BF02392257},
}

@article {MR0244773,
    AUTHOR = {Araki, Huzihiro and Woods, Edward J.},
     TITLE = {A classification of factors},
   JOURNAL = {Publ. Res. Inst. Math. Sci. Ser. A},
  FJOURNAL = {Kyoto University. Research Institute for Mathematical
              Sciences. Publications},
    VOLUME = {4},
      YEAR = {1969},
     PAGES = {51--130},
      ISSN = {0034-5318},
   MRCLASS = {46.65},
  MRNUMBER = {0244773},
MRREVIEWER = {N. Suzuki},
       DOI = {10.2977/prims/1195195263},
       URL = {https://doi.org/10.2977/prims/1195195263},
}

@article {MR480760,
    AUTHOR = {Connes, Alain and Takesaki, Masamichi},
     TITLE = {The flow of weights on factors of type {${\rm III}$}},
   JOURNAL = {Tohoku Math. J. (2)},
  FJOURNAL = {The Tohoku Mathematical Journal. Second Series},
    VOLUME = {29},
      YEAR = {1977},
    NUMBER = {4},
     PAGES = {473--575},
      ISSN = {0040-8735},
   MRCLASS = {46L10 (46L35 46L55)},
  MRNUMBER = {480760},
MRREVIEWER = {M. Walter},
       DOI = {10.2748/tmj/1178240493},
       URL = {https://doi.org/10.2748/tmj/1178240493},
}

@article {MR341115,
    AUTHOR = {Connes, Alain},
     TITLE = {Une classification des facteurs de type {${\rm III}$}},
   JOURNAL = {Ann. Sci. \'{E}cole Norm. Sup. (4)},
  FJOURNAL = {Annales Scientifiques de l'\'{E}cole Normale Sup\'{e}rieure. Quatri\`eme
              S\'{e}rie},
    VOLUME = {6},
      YEAR = {1973},
     PAGES = {133--252},
      ISSN = {0012-9593},
   MRCLASS = {46L10},
  MRNUMBER = {341115},
MRREVIEWER = {E. St\o rmer},
       URL = {http://www.numdam.org/item?id=ASENS_1973_4_6_2_133_0},
}

@article {MR192360,
    AUTHOR = {Effros, Edward G.},
     TITLE = {Global structure in von {N}eumann algebras},
   JOURNAL = {Trans. Amer. Math. Soc.},
  FJOURNAL = {Transactions of the American Mathematical Society},
    VOLUME = {121},
      YEAR = {1966},
     PAGES = {434--454},
      ISSN = {0002-9947},
   MRCLASS = {46.65},
  MRNUMBER = {192360},
MRREVIEWER = {P. C. Deliyannis},
       DOI = {10.2307/1994489},
       URL = {https://doi.org/10.2307/1994489},
}

@article {MR317063,
    AUTHOR = {Mar\'{e}chal, Odile},
     TITLE = {Topologie et structure bor\'{e}lienne sur l'ensemble des alg\`ebres
              de von {N}eumann},
   JOURNAL = {C. R. Acad. Sci. Paris S\'{e}r. A-B},
  FJOURNAL = {Comptes Rendus Hebdomadaires des S\'{e}ances de l'Acad\'{e}mie des
              Sciences. S\'{e}ries A et B},
    VOLUME = {276},
      YEAR = {1973},
     PAGES = {A847--A850},
      ISSN = {0151-0509},
   MRCLASS = {46L10},
  MRNUMBER = {317063},
MRREVIEWER = {S. Sakai},
}

@article {MR185456,
    AUTHOR = {Effros, Edward G.},
     TITLE = {The {B}orel space of von {N}eumann algebras on a separable
              {H}ilbert space},
   JOURNAL = {Pacific J. Math.},
  FJOURNAL = {Pacific Journal of Mathematics},
    VOLUME = {15},
      YEAR = {1965},
     PAGES = {1153--1164},
      ISSN = {0030-8730},
   MRCLASS = {46.65},
  MRNUMBER = {185456},
MRREVIEWER = {J. Ernest},
       URL = {http://projecteuclid.org/euclid.pjm/1102995273},
}

@article {MR333750,
    AUTHOR = {Nielsen, Ole A.},
     TITLE = {Borel sets of von Neumann algebras},
   JOURNAL = {Amer. J. Math.},
  FJOURNAL = {American Journal of Mathematics},
    VOLUME = {95},
      YEAR = {1973},
     PAGES = {145--164},
      ISSN = {0002-9327},
   MRCLASS = {46L10},
  MRNUMBER = {333750},
MRREVIEWER = {E. St\o rmer},
       DOI = {10.2307/2373648},
       URL = {https://doi.org/10.2307/2373648},
}

@article {MR461161,
    AUTHOR = {Sutherland, Colin},
     TITLE = {Crossed products, direct integrals and {C}onnes'
              classification of type {III} factors},
   JOURNAL = {Math. Scand.},
  FJOURNAL = {Mathematica Scandinavica},
    VOLUME = {40},
      YEAR = {1977},
    NUMBER = {2},
     PAGES = {209--214},
      ISSN = {0025-5521},
   MRCLASS = {46L10},
  MRNUMBER = {461161},
MRREVIEWER = {Ole A. Nielsen},
       DOI = {10.7146/math.scand.a-11690},
       URL = {https://doi.org/10.7146/math.scand.a-11690},
}

@article {MR361817,
    AUTHOR = {Woods, Edward J.},
     TITLE = {The classification of factors is not smooth},
   JOURNAL = {Canadian J. Math.},
  FJOURNAL = {Canadian Journal of Mathematics. Journal Canadien de
              Math\'{e}matiques},
    VOLUME = {25},
      YEAR = {1973},
     PAGES = {96--102},
      ISSN = {0008-414X},
   MRCLASS = {46L10},
  MRNUMBER = {361817},
MRREVIEWER = {J. Ernest},
       DOI = {10.4153/CJM-1973-008-7},
       URL = {https://doi.org/10.4153/CJM-1973-008-7},
}

@article {MR96140,
    AUTHOR = {Tomiyama, Jun},
     TITLE = {On the projection of norm one in {$W^{\ast} $}-algebras},
   JOURNAL = {Proc. Japan Acad.},
  FJOURNAL = {Proceedings of the Japan Academy},
    VOLUME = {33},
      YEAR = {1957},
     PAGES = {608--612},
      ISSN = {0021-4280},
   MRCLASS = {46.00},
  MRNUMBER = {96140},
MRREVIEWER = {S. Sherman},
       URL = {http://projecteuclid.org/euclid.pja/1195524885},
}

@article {MR165034,
    AUTHOR = {Mackey, George W.},
     TITLE = {Ergodic theory, group theory, and differential geometry},
   JOURNAL = {Proc. Nat. Acad. Sci. U.S.A.},
  FJOURNAL = {Proceedings of the National Academy of Sciences of the United
              States of America},
    VOLUME = {50},
      YEAR = {1963},
     PAGES = {1184--1191},
      ISSN = {0027-8424},
   MRCLASS = {22.60 (28.70)},
  MRNUMBER = {165034},
MRREVIEWER = {R. J. Blattner},
       DOI = {10.1073/pnas.50.6.1184},
       URL = {https://doi.org/10.1073/pnas.50.6.1184},
}

@incollection {MR1855241,
    AUTHOR = {Anantharaman-Delaroche, Claire and Renault, Jean},
     TITLE = {Amenable groupoids},
 BOOKTITLE = {Groupoids in analysis, geometry, and physics ({B}oulder, {CO},
              1999)},
    SERIES = {Contemp. Math.},
    VOLUME = {282},
     PAGES = {35--46},
 PUBLISHER = {Amer. Math. Soc., Providence, RI},
      YEAR = {2001},
   MRCLASS = {46L55 (22A22 22D25 58H05)},
  MRNUMBER = {1855241},
       DOI = {10.1090/conm/282/04677},
       URL = {https://doi.org/10.1090/conm/282/04677},
}

@article {MR4266136,
    AUTHOR = {Donsig, Allan P. and Fuller, Adam H. and Pitts, David R.},
     TITLE = {Cartan triples},
   JOURNAL = {Int. Math. Res. Not. IMRN},
  FJOURNAL = {International Mathematics Research Notices. IMRN},
      YEAR = {2021},
    NUMBER = {11},
     PAGES = {8007--8070},
      ISSN = {1073-7928},
   MRCLASS = {46L10 (16E50 46L51)},
  MRNUMBER = {4266136},
MRREVIEWER = {Vaggelis Felouzis},
       DOI = {10.1093/imrn/rnz340},
       URL = {https://doi.org/10.1093/imrn/rnz340},
}

@book {MR1724106,
    AUTHOR = {Paterson, Alan L. T.},
     TITLE = {Groupoids, inverse semigroups, and their operator algebras},
    SERIES = {Progress in Mathematics},
    VOLUME = {170},
 PUBLISHER = {Birkh\"{a}user Boston, Inc., Boston, MA},
      YEAR = {1999},
     PAGES = {xvi+274},
      ISBN = {0-8176-4051-7},
   MRCLASS = {22A22 (20M18 22-02 22D25 46L05)},
  MRNUMBER = {1724106},
MRREVIEWER = {Jean N. Renault},
       DOI = {10.1007/978-1-4612-1774-9},
       URL = {https://doi.org/10.1007/978-1-4612-1774-9},
}

@article {MR578730,
    AUTHOR = {Feldman, Jacob and Moore, Calvin C.},
     TITLE = {Ergodic equivalence relations, cohomology, and von {N}eumann
              algebras. {II}},
   JOURNAL = {Trans. Amer. Math. Soc.},
  FJOURNAL = {Transactions of the American Mathematical Society},
    VOLUME = {234},
      YEAR = {1977},
    NUMBER = {2},
     PAGES = {325--359},
      ISSN = {0002-9947},
   MRCLASS = {22D40 (28A65 46L10)},
  MRNUMBER = {578730},
       DOI = {10.2307/1997925},
       URL = {https://doi.org/10.2307/1997925},
}

@article {MR578656,
    AUTHOR = {Feldman, Jacob and Moore, Calvin C.},
     TITLE = {Ergodic equivalence relations, cohomology, and von {N}eumann
              algebras. {I}},
   JOURNAL = {Trans. Amer. Math. Soc.},
  FJOURNAL = {Transactions of the American Mathematical Society},
    VOLUME = {234},
      YEAR = {1977},
    NUMBER = {2},
     PAGES = {289--324},
      ISSN = {0002-9947},
   MRCLASS = {22D40 (28A65 46L10)},
  MRNUMBER = {578656},
       DOI = {10.2307/1997924},
       URL = {https://doi.org/10.2307/1997924},
}

@article {MR1007409,
    AUTHOR = {Feldman, Jacob and Sutherland, Colin E. and Zimmer, Robert J.},
     TITLE = {Subrelations of ergodic equivalence relations},
   JOURNAL = {Ergodic Theory Dynam. Systems},
  FJOURNAL = {Ergodic Theory and Dynamical Systems},
    VOLUME = {9},
      YEAR = {1989},
    NUMBER = {2},
     PAGES = {239--269},
      ISSN = {0143-3857},
   MRCLASS = {28D15 (22D40)},
  MRNUMBER = {1007409},
MRREVIEWER = {Arlan Ramsay},
       DOI = {10.1017/S0143385700004958},
       URL = {https://doi.org/10.1017/S0143385700004958},
}

@article {MR662736,
    AUTHOR = {Connes, Alain and Feldman, Jacob and Weiss, Benjamin},
     TITLE = {An amenable equivalence relation is generated by a single
              transformation},
   JOURNAL = {Ergodic Theory Dynam. Systems},
  FJOURNAL = {Ergodic Theory and Dynamical Systems},
    VOLUME = {1},
      YEAR = {1981},
    NUMBER = {4},
     PAGES = {431--450 (1982)},
      ISSN = {0143-3857},
   MRCLASS = {46L55 (22D40)},
  MRNUMBER = {662736},
       DOI = {10.1017/s014338570000136x},
       URL = {https://doi.org/10.1017/s014338570000136x},
}

@article {MR4258165,
    AUTHOR = {Popa, Sorin},
     TITLE = {On the vanishing cohomology problem for cocycle actions of
              groups on {$\rm II_1$} factors},
   JOURNAL = {Ann. Sci. \'{E}c. Norm. Sup\'{e}r. (4)},
  FJOURNAL = {Annales Scientifiques de l'\'{E}cole Normale Sup\'{e}rieure. Quatri\`eme
              S\'{e}rie},
    VOLUME = {54},
      YEAR = {2021},
    NUMBER = {2},
     PAGES = {407--443},
      ISSN = {0012-9593},
   MRCLASS = {22D25 (37A20 37A55 46L10 46L55)},
  MRNUMBER = {4258165},
MRREVIEWER = {Robert S. Doran},
       DOI = {10.24033/asens.2461},
       URL = {https://doi.org/10.24033/asens.2461},
}

@article {MR1066817,
    AUTHOR = {Packer, Judith A. and Raeburn, Iain},
     TITLE = {Twisted crossed products of {$C^*$}-algebras. {II}},
   JOURNAL = {Math. Ann.},
  FJOURNAL = {Mathematische Annalen},
    VOLUME = {287},
      YEAR = {1990},
    NUMBER = {4},
     PAGES = {595--612},
      ISSN = {0025-5831},
   MRCLASS = {46L55},
  MRNUMBER = {1066817},
MRREVIEWER = {John Quigg},
       DOI = {10.1007/BF01446916},
       URL = {https://doi.org/10.1007/BF01446916},
}

@book {MR807949,
    AUTHOR = {Ocneanu, Adrian},
     TITLE = {Actions of discrete amenable groups on von {N}eumann algebras},
    SERIES = {Lecture Notes in Mathematics},
    VOLUME = {1138},
 PUBLISHER = {Springer-Verlag, Berlin},
      YEAR = {1985},
     PAGES = {iv+115},
      ISBN = {3-540-15663-1},
   MRCLASS = {46L55 (46L35 46L40)},
  MRNUMBER = {807949},
MRREVIEWER = {Hisashi Choda},
       DOI = {10.1007/BFb0098579},
       URL = {https://doi.org/10.1007/BFb0098579},
}

@article {MR587749,
    AUTHOR = {Jones, Vaughan F. R.},
     TITLE = {Actions of finite groups on the hyperfinite type {${\rm
              II}_{1}$} factor},
   JOURNAL = {Mem. Amer. Math. Soc.},
  FJOURNAL = {Memoirs of the American Mathematical Society},
    VOLUME = {28},
      YEAR = {1980},
    NUMBER = {237},
     PAGES = {v+70},
      ISSN = {0065-9266},
   MRCLASS = {46L55 (20C99 46L40)},
  MRNUMBER = {587749},
MRREVIEWER = {Marie Choda},
       DOI = {10.1090/memo/0237},
       URL = {https://doi.org/10.1090/memo/0237},
}

@article {MR842413,
    AUTHOR = {Sutherland, Colin E. and Takesaki, Masamichi},
     TITLE = {Actions of discrete amenable groups and groupoids on von
              {N}eumann algebras},
   JOURNAL = {Publ. Res. Inst. Math. Sci.},
  FJOURNAL = {Kyoto University. Research Institute for Mathematical
              Sciences. Publications},
    VOLUME = {21},
      YEAR = {1985},
    NUMBER = {6},
     PAGES = {1087--1120},
      ISSN = {0034-5318},
   MRCLASS = {46L55 (22D40)},
  MRNUMBER = {842413},
       DOI = {10.2977/prims/1195178510},
       URL = {https://doi.org/10.2977/prims/1195178510},
}

@article {MR990219,
    AUTHOR = {Sutherland, Colin E. and Takesaki, Masamichi},
     TITLE = {Actions of discrete amenable groups on injective factors of
              type {${\rm III}_\lambda,\;\lambda\neq 1$}},
   JOURNAL = {Pacific J. Math.},
  FJOURNAL = {Pacific Journal of Mathematics},
    VOLUME = {137},
      YEAR = {1989},
    NUMBER = {2},
     PAGES = {405--444},
      ISSN = {0030-8730},
   MRCLASS = {46L40 (22D25 46L10 46L55)},
  MRNUMBER = {990219},
MRREVIEWER = {Sze-Kai Tsui},
       URL = {http://projecteuclid.org/euclid.pjm/1102650391},
}

@article {MR1179014,
    AUTHOR = {Kawahigashi, Yasuyuki and Sutherland, Colin E. and Takesaki, Masamichi},
     TITLE = {The structure of the automorphism group of an injective factor
              and the cocycle conjugacy of discrete abelian group actions},
   JOURNAL = {Acta Math.},
  FJOURNAL = {Acta Mathematica},
    VOLUME = {169},
      YEAR = {1992},
    NUMBER = {1-2},
     PAGES = {105--130},
      ISSN = {0001-5962},
   MRCLASS = {46L40},
  MRNUMBER = {1179014},
MRREVIEWER = {Yoshikazu Katayama},
       DOI = {10.1007/BF02392758},
       URL = {https://doi.org/10.1007/BF02392758},
}

@article {MR1621416,
    AUTHOR = {Katayama, Yoshikazu and Sutherland, Colin E. and Takesaki,
              Masamichi},
     TITLE = {The characteristic square of a factor and the cocycle
              conjugacy of discrete group actions on factors},
   JOURNAL = {Invent. Math.},
  FJOURNAL = {Inventiones Mathematicae},
    VOLUME = {132},
      YEAR = {1998},
    NUMBER = {2},
     PAGES = {331--380},
      ISSN = {0020-9910},
   MRCLASS = {46L55 (46L40)},
  MRNUMBER = {1621416},
MRREVIEWER = {Claire Anantharaman-Delaroche},
       DOI = {10.1007/s002220050226},
       URL = {https://doi.org/10.1007/s002220050226},
}

@article {MR1064693,
    AUTHOR = {Haagerup, Uffe and St\o rmer, Erling},
     TITLE = {Pointwise inner automorphisms of von {N}eumann algebras},
      NOTE = {With an appendix by Colin Sutherland},
   JOURNAL = {J. Funct. Anal.},
  FJOURNAL = {Journal of Functional Analysis},
    VOLUME = {92},
      YEAR = {1990},
    NUMBER = {1},
     PAGES = {177--201},
      ISSN = {0022-1236},
   MRCLASS = {46L40 (46L10)},
  MRNUMBER = {1064693},
       DOI = {10.1016/0022-1236(90)90074-U},
       URL = {https://doi.org/10.1016/0022-1236(90)90074-U},
}

@article {MR3181546,
    AUTHOR = {Masuda, Toshihiko},
     TITLE = {Unified approach to the classification of actions of discrete
              amenable groups on injective factors},
   JOURNAL = {J. Reine Angew. Math.},
  FJOURNAL = {Journal f\"{u}r die Reine und Angewandte Mathematik. [Crelle's
              Journal]},
    VOLUME = {683},
      YEAR = {2013},
     PAGES = {1--47},
      ISSN = {0075-4102},
   MRCLASS = {46L10},
  MRNUMBER = {3181546},
       DOI = {10.1515/crelle-2011-0011},
       URL = {https://doi.org/10.1515/crelle-2011-0011},
}

@article {MR2353241,
    AUTHOR = {Masuda, Toshihiko},
     TITLE = {Evans-{K}ishimoto type argument for actions of discrete
              amenable groups on {M}c{D}uff factors},
   JOURNAL = {Math. Scand.},
  FJOURNAL = {Mathematica Scandinavica},
    VOLUME = {101},
      YEAR = {2007},
    NUMBER = {1},
     PAGES = {48--64},
      ISSN = {0025-5521},
   MRCLASS = {46L40 (46L55)},
  MRNUMBER = {2353241},
MRREVIEWER = {Yoshikazu Katayama},
       DOI = {10.7146/math.scand.a-15031},
       URL = {https://doi.org/10.7146/math.scand.a-15031},
}

@article {MR4484234,
    AUTHOR = {Masuda, Toshihiko},
     TITLE = {Classification of outer actions of discrete amenable groupoids
              on injective factors},
   JOURNAL = {J. Math. Soc. Japan},
  FJOURNAL = {Journal of the Mathematical Society of Japan},
    VOLUME = {74},
      YEAR = {2022},
    NUMBER = {3},
     PAGES = {873--901},
      ISSN = {0025-5645},
   MRCLASS = {46L40 (46L55)},
  MRNUMBER = {4484234},
       DOI = {10.2969/jmsj/86328632},
       URL = {https://doi.org/10.2969/jmsj/86328632},
}

@article {MR169988,
    AUTHOR = {Maharam, Dorothy},
     TITLE = {Incompressible transformations},
   JOURNAL = {Fund. Math.},
  FJOURNAL = {Polska Akademia Nauk. Fundamenta Mathematicae},
    VOLUME = {56},
      YEAR = {1964},
     PAGES = {35--50},
      ISSN = {0016-2736},
   MRCLASS = {28.70},
  MRNUMBER = {169988},
MRREVIEWER = {P. R. Halmos},
       DOI = {10.4064/fm-56-1-35-50},
       URL = {https://doi.org/10.4064/fm-56-1-35-50},
}

@article {MR910005,
    AUTHOR = {Ornstein, Donald S. and Weiss, Benjamin},
     TITLE = {Entropy and isomorphism theorems for actions of amenable
              groups},
   JOURNAL = {J. Analyse Math.},
    VOLUME = {48},
      YEAR = {1987},
     PAGES = {1--141},
      ISSN = {0021-7670},
   MRCLASS = {28D20 (22D40)},
  MRNUMBER = {910005},
MRREVIEWER = {John C. Kieffer},
       DOI = {10.1007/BF02790325},
       URL = {https://doi.org/10.1007/BF02790325},
}

@article {MR1002120,
    AUTHOR = {Bezuglyi, Sergey I. and Golodets, Valentin Ya.},
     TITLE = {Type {${\rm III}_0$} transformations of measure space and
              outer conjugacy of countable amenable groups of automorphisms},
   JOURNAL = {J. Operator Theory},
  FJOURNAL = {Journal of Operator Theory},
    VOLUME = {21},
      YEAR = {1989},
    NUMBER = {1},
     PAGES = {3--40},
      ISSN = {0379-4024},
   MRCLASS = {46L55 (22D40 28D99)},
  MRNUMBER = {1002120},
MRREVIEWER = {Colin E. Sutherland},
}

@incollection {MR906080,
    AUTHOR = {Bezuglyi, Sergey I.},
     TITLE = {Outer conjugation of the actions of countable amenable groups},
 BOOKTITLE = {Mathematical physics, functional analysis ({R}ussian)},
     PAGES = {59--63, 145},
 PUBLISHER = {``Naukova Dumka'', Kiev},
      YEAR = {1986},
   MRCLASS = {46L55 (22D40 28D15)},
  MRNUMBER = {906080},
MRREVIEWER = {N. N. Ganikhodjaev},
}

@article {MR864169,
    AUTHOR = {Bezuglyi, Sergey I. and Golodets, Valentin Ya.},
     TITLE = {Outer conjugacy of actions of countable amenable groups on a
              space with measure},
   JOURNAL = {Izv. Akad. Nauk SSSR Ser. Mat.},
  FJOURNAL = {Izvestiya Akademii Nauk SSSR. Seriya Matematicheskaya},
    VOLUME = {50},
      YEAR = {1986},
    NUMBER = {4},
     PAGES = {643--660, 877},
      ISSN = {0373-2436},
   MRCLASS = {46L40 (22D40 28D15 46L55)},
  MRNUMBER = {864169},
MRREVIEWER = {K. A. Ross},
}

@book {MR3616077,
    AUTHOR = {Kerr, David and Li, Hanfeng},
     TITLE = {Ergodic theory},
    SERIES = {Springer Monographs in Mathematics},
      NOTE = {Independence and dichotomies},
 PUBLISHER = {Springer, Cham},
      YEAR = {2016},
     PAGES = {xxxiv+431},
   MRCLASS = {37-02 (03E15 05D10 05D40 37Axx 37B05 37B40)},
  MRNUMBER = {3616077},
MRREVIEWER = {M. L. Blank},
       DOI = {10.1007/978-3-319-49847-8},
       URL = {https://doi.org/10.1007/978-3-319-49847-8},
}

@article{https://doi.org/10.48550/arxiv.2210.15916,
  DOI = {10.48550/ARXIV.2210.15916},
  
  URL = {https://arxiv.org/abs/2210.15916},
  
  AUTHOR = {Masuda, Toshihiko},
  TITLE = {Actions of discrete amenable groups into the normalizers of full groups of ergodic transformations},
  
  PUBLISHER = {arXiv},
  JOURNAL = {arXiv},
  YEAR = {2022},
  
  copyright = {Creative Commons Attribution 4.0 International}
}

@article {MR438149,
    AUTHOR = {Takesaki, Masamichi},
     TITLE = {Duality for crossed products and the structure of von
              {N}eumann algebras of type {III}},
   JOURNAL = {Acta Math.},
  FJOURNAL = {Acta Mathematica},
    VOLUME = {131},
      YEAR = {1973},
     PAGES = {249--310},
      ISSN = {0001-5962},
   MRCLASS = {46L10},
  MRNUMBER = {438149},
MRREVIEWER = {S. Sakai},
       DOI = {10.1007/BF02392041},
       URL = {https://doi.org/10.1007/BF02392041},
}

@inproceedings {MR550409,
    AUTHOR = {Hamachi, Toshihiro and Osikawa, Motosige},
     TITLE = {Fundamental homomorphism of normalizer group of ergodic
              transformation},
 BOOKTITLE = {Ergodic theory ({P}roc. {C}onf., {M}ath. {F}orschungsinst.,
              {O}berwolfach, 1978)},
    SERIES = {Lecture Notes in Math.},
    VOLUME = {729},
     PAGES = {43--57},
 PUBLISHER = {Springer, Berlin},
      YEAR = {1979},
   MRCLASS = {46L10 (28D99)},
  MRNUMBER = {550409},
MRREVIEWER = {A. de Korvin},
}

@article {MR0444900,
    AUTHOR = {Connes, Alain and Krieger, Wolfgang},
     TITLE = {Measure space automorphisms, the normalizers of their full
              groups, and approximate finiteness},
   JOURNAL = {J. Functional Analysis},
    VOLUME = {24},
      YEAR = {1977},
    NUMBER = {4},
     PAGES = {336--352},
   MRCLASS = {28A65},
  MRNUMBER = {0444900},
MRREVIEWER = {Douglas Lind},
       DOI = {10.1016/0022-1236(77)90062-3},
       URL = {https://doi.org/10.1016/0022-1236(77)90062-3},
}

@article {MR842412,
    AUTHOR = {Sutherland, Colin E.},
     TITLE = {A {B}orel parametrization of {P}olish groups},
   JOURNAL = {Publ. Res. Inst. Math. Sci.},
  FJOURNAL = {Kyoto University. Research Institute for Mathematical
              Sciences. Publications},
    VOLUME = {21},
      YEAR = {1985},
    NUMBER = {6},
     PAGES = {1067--1086},
      ISSN = {0034-5318},
   MRCLASS = {22A05 (46L10)},
  MRNUMBER = {842412},
MRREVIEWER = {G. A. Reid},
       DOI = {10.2977/prims/1195178509},
       URL = {https://doi.org/10.2977/prims/1195178509},
}

@article {MR766264,
    AUTHOR = {Jones, Vaughan F. R. and Takesaki, Masamichi},
     TITLE = {Actions of compact abelian groups on semifinite injective
              factors},
   JOURNAL = {Acta Math.},
  FJOURNAL = {Acta Mathematica},
    VOLUME = {153},
      YEAR = {1984},
    NUMBER = {3-4},
     PAGES = {213--258},
      ISSN = {0001-5962},
   MRCLASS = {46L35 (46L10)},
  MRNUMBER = {766264},
       DOI = {10.1007/BF02392378},
       URL = {https://doi.org/10.1007/BF02392378},
}

@article {ADP_An_introduction,
    AUTHOR = {Anantharaman-Delaroche, Claire and Popa, Sorin},
     TITLE = {An introduction to $\rm{II}_{1}$ factors},
   JOURNAL = {Preliminary version},
      YEAR = {2010},
}

@article {MR0390172,
    AUTHOR = {Hamachi, Toshihiro and Oka, Yukimasa and Osikawa, Motosige},
     TITLE = {Flows associated with ergodic non-singular transformation
              groups},
   JOURNAL = {Publ. Res. Inst. Math. Sci.},
  FJOURNAL = {Kyoto University. Research Institute for Mathematical
              Sciences. Publications},
    VOLUME = {11},
      YEAR = {1976},
    NUMBER = {1},
     PAGES = {31--50},
      ISSN = {0034-5318},
   MRCLASS = {28A65 (46L10)},
  MRNUMBER = {0390172},
MRREVIEWER = {Alain Connes},
       DOI = {10.2977/prims/1195191686},
       URL = {https://doi.org/10.2977/prims/1195191686},
}

@incollection {MR1643188,
    AUTHOR = {Sutherland, Colin E. and Takesaki, Masamichi},
     TITLE = {Right inverse of the module of approximately
              finite-dimensional factors of type {III} and approximately
              finite ergodic principal measured groupoids},
 BOOKTITLE = {Operator algebras and their applications, {II} ({W}aterloo,
              {ON}, 1994/1995)},
    SERIES = {Fields Inst. Commun.},
    VOLUME = {20},
     PAGES = {149--159},
 PUBLISHER = {Amer. Math. Soc., Providence, RI},
      YEAR = {1998},
   MRCLASS = {46L35 (46L40 46L55)},
  MRNUMBER = {1643188},
MRREVIEWER = {Jean N. Renault},
       DOI = {10.1007/s002220050226},
       URL = {https://doi.org/10.1007/s002220050226},
}

@article {MR394228,
    AUTHOR = {Connes, Alain},
     TITLE = {Outer conjugacy classes of automorphisms of factors},
   JOURNAL = {Ann. Sci. \'{E}cole Norm. Sup. (4)},
  FJOURNAL = {Annales Scientifiques de l'\'{E}cole Normale Sup\'{e}rieure. Quatri\`eme
              S\'{e}rie},
    VOLUME = {8},
      YEAR = {1975},
    NUMBER = {3},
     PAGES = {383--419},
      ISSN = {0012-9593},
   MRCLASS = {46L10},
  MRNUMBER = {394228},
MRREVIEWER = {Hisashi Choda},
       URL = {http://www.numdam.org/item?id=ASENS_1975_4_8_3_383_0},
}

@book {MR617740,
    AUTHOR = {Hamachi, Toshihiro and Osikawa, Motosige},
     TITLE = {Ergodic groups of automorphisms and {K}rieger's theorems},
    SERIES = {Seminar on Mathematical Sciences},
    VOLUME = {3},
 PUBLISHER = {Keio University, Department of Mathematics, Yokohama},
      YEAR = {1981},
     PAGES = {113},
   MRCLASS = {46L40 (28Dxx 46L55)},
  MRNUMBER = {617740},
}

@article{vaes_verjans_2022, title={Orbit equivalence superrigidity for type $\textrm{III}_{0}$ actions}, DOI={10.1017/etds.2022.112}, journal={Ergodic Theory and Dynamical Systems}, publisher={Cambridge University Press}, author={Vaes, Stefaan and Verjans, Bram}, year={2022}, pages={1–33}}

@article {MR414775,
    AUTHOR = {Moore, Calvin C.},
     TITLE = {Group extensions and cohomology for locally compact groups.
              {III}},
   JOURNAL = {Trans. Amer. Math. Soc.},
  FJOURNAL = {Transactions of the American Mathematical Society},
    VOLUME = {221},
      YEAR = {1976},
    NUMBER = {1},
     PAGES = {1--33},
      ISSN = {0002-9947},
   MRCLASS = {22D05 (22D10 22D30)},
  MRNUMBER = {414775},
MRREVIEWER = {J. M. G. Fell},
       DOI = {10.2307/1997540},
       URL = {https://doi.org/10.2307/1997540},
}

@article {MR448101,
    AUTHOR = {Connes, A.},
     TITLE = {Periodic automorphisms of the hyperfinite factor of type
              {II}1},
   JOURNAL = {Acta Sci. Math. (Szeged)},
  FJOURNAL = {Acta Universitatis Szegediensis. Acta Scientiarum
              Mathematicarum},
    VOLUME = {39},
      YEAR = {1977},
    NUMBER = {1-2},
     PAGES = {39--66},
      ISSN = {0001-6969},
   MRCLASS = {46L10},
  MRNUMBER = {448101},
MRREVIEWER = {Yoshinori Haga},
}

@article {MR1432548,
    AUTHOR = {Evans, David E. and Kishimoto, Akitaka},
     TITLE = {Trace scaling automorphisms of certain stable {AF} algebras},
   JOURNAL = {Hokkaido Math. J.},
  FJOURNAL = {Hokkaido Mathematical Journal},
    VOLUME = {26},
      YEAR = {1997},
    NUMBER = {1},
     PAGES = {211--224},
      ISSN = {0385-4035},
   MRCLASS = {46L40 (46L05)},
  MRNUMBER = {1432548},
MRREVIEWER = {W\l adys\l aw Adam Majewski},
       DOI = {10.14492/hokmj/1351257815},
       URL = {https://doi.org/10.14492/hokmj/1351257815},
}

@article {MR1156672,
    AUTHOR = {Kawahigashi, Yasuyuki and Takesaki, Masamichi},
     TITLE = {Compact abelian group actions on injective factors},
   JOURNAL = {J. Funct. Anal.},
  FJOURNAL = {Journal of Functional Analysis},
    VOLUME = {105},
      YEAR = {1992},
    NUMBER = {1},
     PAGES = {112--128},
      ISSN = {0022-1236},
   MRCLASS = {46L55},
  MRNUMBER = {1156672},
MRREVIEWER = {Yoshikazu Katayama},
       DOI = {10.1016/0022-1236(92)90074-S},
       URL = {https://doi.org/10.1016/0022-1236(92)90074-S},
}

@article {MR780502,
    AUTHOR = {Bezuglyi, S. I. and Golodets, V. Ya.},
     TITLE = {Groups of measure space transformations and invariants of
              outer conjugation for automorphisms from normalizers of type
              {${\rm III}$} full groups},
   JOURNAL = {J. Funct. Anal.},
  FJOURNAL = {Journal of Functional Analysis},
    VOLUME = {60},
      YEAR = {1985},
    NUMBER = {3},
     PAGES = {341--369},
      ISSN = {0022-1236},
   MRCLASS = {46L40 (28D20 46L55)},
  MRNUMBER = {780502},
MRREVIEWER = {Colin E. Sutherland},
       DOI = {10.1016/0022-1236(85)90044-8},
       URL = {https://doi.org/10.1016/0022-1236(85)90044-8},
}

@article{golodets1983existence,
  title={Existence and uniqueness of cocycles of an ergodic automorphism with dense ranges in amenable groups},
  author={Golodets, V Ya and Sinel’shchikov, SD},
  journal={Preprint},
  pages={19--83},
  year={1983}
}

@article {MR1978219,
    AUTHOR = {Aoi, Hisashi},
     TITLE = {A construction of equivalence subrelations for intermediate
              subalgebras},
   JOURNAL = {J. Math. Soc. Japan},
  FJOURNAL = {Journal of the Mathematical Society of Japan},
    VOLUME = {55},
      YEAR = {2003},
    NUMBER = {3},
     PAGES = {713--725},
      ISSN = {0025-5645,1881-1167},
   MRCLASS = {46L10 (37A20)},
  MRNUMBER = {1978219},
MRREVIEWER = {Baruch\ Solel},
       DOI = {10.2969/jmsj/1191418999},
       URL = {https://doi.org/10.2969/jmsj/1191418999},
}

@article {MR1002543,
    AUTHOR = {Packer, Judith A. and Raeburn, Iain},
     TITLE = {Twisted crossed products of {$C^*$}-algebras},
   JOURNAL = {Math. Proc. Cambridge Philos. Soc.},
  FJOURNAL = {Mathematical Proceedings of the Cambridge Philosophical
              Society},
    VOLUME = {106},
      YEAR = {1989},
    NUMBER = {2},
     PAGES = {293--311},
      ISSN = {0305-0041,1469-8064},
   MRCLASS = {46L55 (22D25 46L40)},
  MRNUMBER = {1002543},
MRREVIEWER = {John\ Quigg},
       DOI = {10.1017/S0305004100078129},
       URL = {https://doi.org/10.1017/S0305004100078129},
}

@article {MR0574031,
    AUTHOR = {Sutherland, Colin E.},
     TITLE = {Cohomology and extensions of von {N}eumann algebras. {I},
              {II}},
   JOURNAL = {Publ. Res. Inst. Math. Sci.},
  FJOURNAL = {Kyoto University. Research Institute for Mathematical
              Sciences. Publications},
    VOLUME = {16},
      YEAR = {1980},
    NUMBER = {1},
     PAGES = {105--133, 135--174},
      ISSN = {0034-5318,1663-4926},
   MRCLASS = {46L10 (46L40 46M20)},
  MRNUMBER = {574031},
MRREVIEWER = {B.\ E.\ Johnson},
       DOI = {10.2977/prims/1195187501},
       URL = {https://doi.org/10.2977/prims/1195187501},
}

@article {MR4243020,
    AUTHOR = {Bj\"{o}rklund, Michael and Kosloff, Zemer and Vaes, Stefaan},
     TITLE = {Ergodicity and type of nonsingular {B}ernoulli actions},
   JOURNAL = {Invent. Math.},
  FJOURNAL = {Inventiones Mathematicae},
    VOLUME = {224},
      YEAR = {2021},
    NUMBER = {2},
     PAGES = {573--625},
      ISSN = {0020-9910,1432-1297},
   MRCLASS = {37A05 (20F65 37A25 46L55)},
  MRNUMBER = {4243020},
MRREVIEWER = {Judith\ A.\ Packer},
       DOI = {10.1007/s00222-020-01014-0},
       URL = {https://doi.org/10.1007/s00222-020-01014-0},
}

\end{document}